\documentclass[12pt]{article}
\usepackage{amsthm}
\usepackage{times}
\usepackage{thm-restate}
\usepackage{bbding}
\usepackage{bbm}

\usepackage[ruled,vlined]{algorithm2e}
\makeatletter
\def\BState{\State\hskip-\ALG@thistlm}
\makeatother

\usepackage{authblk}
\usepackage[all]{xy}
\usepackage[english]{babel}
\usepackage[margin=1in]{geometry}
\usepackage[utf8]{inputenc}
\usepackage{amsmath,accents}

\usepackage{framed}
\usepackage[normalem]{ulem}
\usepackage{tikz}
\usepackage{tikz-cd}
\usepackage{mathtools}

\usepackage{amsfonts}
\usepackage{graphicx}
\usepackage{amssymb}
\usepackage{url}
\usetikzlibrary{decorations.fractals}
\usepackage{pgfplots}
\usepackage[colorlinks]{hyperref}
\hypersetup{colorlinks,linkcolor={red},citecolor={blue},urlcolor={gray}}
\usepackage{cleveref}
\usepackage{setspace}
\numberwithin{equation}{section}
\usepackage[intoc, english, refpage]{nomencl}  
\makenomenclature

\providecommand{\keywords}[1]
{
  \small	
  \textbf{\textit{Keywords---}} #1
}

\newcommand{\N}{\mathbb{N}}

\newcommand{\ufin}{\ums^\mathrm{fin}}
\newcommand{\dX}{d_X}

\newcommand{\uX}{u_X}

\newcommand{\ct}[1]{_{\mathfrak{c}\left(#1\right)}}
\newcommand{\ot}[1]{_{\mathfrak{o}\lc #1\rc}}

\newtheorem{claim}{Claim}

\newtheorem{theorem}{Theorem}[section]
\newtheorem{remark}[theorem]{Remark}
\newtheorem{lemma}[theorem]{Lemma}
\newtheorem{corollary}[theorem]{Corollary}

\newtheorem{proposition}[theorem]{Proposition}
\newtheorem{definition}[theorem]{Definition}

\newtheorem{example}[theorem]{Example}
\newtheorem{conjecture}[theorem]{Conjecture}

\newcommand{\R}{\mathbb{R}}

\newcommand{\spec}{\mathrm{spec}}

\newcommand{\dis}{\mathrm{dis}}
\newcommand{\codis}{\mathrm{codis}}
\newcommand{\disp}[1]{\dis_{#1}}
\newcommand{\hdgh}{\widehat{d}_{\mathrm{GH}}}
\newcommand{\disi}{\mathrm{dis}_\mathrm{I}}
\newcommand{\codisi}{\mathrm{codis}_\mathrm{I}}
\newcommand{\disip}[1]{\mathrm{dis}_{\mathrm{I},#1}}
\newcommand{\codisip}[1]{\mathrm{codis}_{\mathrm{I},#1}}
\newcommand{\ums}{\mathcal{U}}
\newcommand{\ms}{\mathcal{M}}
\newcommand{\msp}{\mathcal{M}_p}

\newcommand{\dlp}{\dgh^{\scriptscriptstyle{(p)}}}
\newcommand{\dghp}[1]{\dgh^{\scriptscriptstyle{(#1)}}}

\newcommand{\dip}{d_{\mathrm{I},p}}
\newcommand{\dxp}{d_X^{\scriptscriptstyle{(p)}}}
\newcommand{\dyp}{d_Y^{\scriptscriptstyle{(p)}}}

\newcommand{\dgh}{d_{\mathrm{GH}}}
\newcommand{\ugh}{u_{\mathrm{GH}}}
\newcommand{\dint}{d_{\mathrm{I}}}
\newcommand{\eps}{\varepsilon}
\newcommand{\diam}{\mathrm{diam}}
\newcommand{\dH}{d_\mathrm{H}}
 \newcommand{\ps}[1]{%
  \boxplus_{#1}}
\newcommand{\Rp}{\mathbb{R}_{\geq0}}

\newcommand{\lc}{\left(}
\newcommand{\rc}{\right)}

\newcommand{\dI}{d_\mathrm{I}}
\newcommand{\dIp}[1]{d_{\mathrm{I},#1}}

\newcommand{\ls}{\left|}
\newcommand{\rs}{\right|}
\newcommand{\lb}{\left\{}
\newcommand{\rb}{\right\}}

\title{On $p$-metric spaces and the $p$-Gromov-Hausdorff distance}

\author[1]{Facundo M\'emoli}
\author[2]{Zhengchao Wan}

\affil[1]{Department of Mathematics and Department of Computer Science and Engineering,
 	The Ohio State University\\ 	\texttt{memoli@math.osu.edu}}

\affil[2]{Department of Mathematics,
 	The Ohio State University\\
 	\texttt{wan.252@osu.edu}}

\date{\today}

\begin{document}
\maketitle 

\begin{abstract}
For each given $p\in[1,\infty]$ we investigate certain sub-family $\mathcal{M}_p$ of the collection of all compact metric spaces $\mathcal{M}$ which are characterized by the satisfaction of a strengthened form of the triangle inequality which encompasses, for example, the strong triangle inequality satisfied by ultrametric spaces.
We identify a one parameter family of Gromov-Hausdorff like distances $\{d_{\mathrm{GH}}^{\scriptscriptstyle{(p)}}\}_{p\in[1,\infty]}$ on $\mathcal{M}_p$ and study geometric and topological properties of these distances as well as the stability of certain canonical projections $\mathfrak{S}_p:\mathcal{M}\rightarrow \mathcal{M}_p$. For the collection $\mathcal{U}$ of all compact ultrametric spaces, which corresponds to the case $p=\infty$ of the family $\msp$, we explore a one parameter family of interleaving-type distances and reveal their relationship with $\{d_{\mathrm{GH}}^{\scriptscriptstyle{(p)}}\}_{p\in[1,\infty]}$. 
\end{abstract}


\tableofcontents

\section{Introduction}\label{sec:intro} 
The notion of metric space is a fundamental concept in mathematics, computer science, and applied disciplines such as data science, where metric spaces serve as a model for datasets \cite{deza}. A metric space is a pair $(X,d_X)$ consisting of a set $X$ and a function $d_X:X\times X\rightarrow\R$ satisfying the following three conditions: for any $x,x',x''\in X$,
\begin{enumerate}
    \item $d_X(x,x')\geq 0$ and $d_X(x,x')=0$ if and only if $x=x'$.
    \item $d_X(x,x')=d_X(x',x)$.
    \item $d_X(x,x')+d_X(x',x'')\geq d_X(x,x'')$.
\end{enumerate}
The function $d_X$ is referred to as the metric (or distance function) on $X$. Common examples of metric spaces include subsets of Euclidean spaces, Riemannian manifolds, and metric graphs. In this paper, we are mostly interested in compact metric spaces.

An important notion regarding metric spaces is that of isometric embedding.

\begin{definition}
A set map $f:X\rightarrow Y$ between two metric spaces is called an \textit{isometric embedding} if for any $x,x'\in X$, $ d_Y\left(f(x),f(x')\right)=d_X(x,x'). $ We use the notation $f:X\hookrightarrow Y$ to denote isometric embeddings. If moreover $f$ is bijective,  we then say that $f$ is an \textit{isometry}. Whenever an isometry exists between $X$ and $Y$ we say that $X$ is isometric to $Y$ and  denote this as $X\cong Y$.
\end{definition}

We denote by $\ms$ the collection of all (isometry classes of) compact metric spaces.\label{sym:ms}

\nomenclature{$\ms$}{The collection of all (isometry classes of) compact metric spaces}

One natural question in metric geometry and in data analysis is how to compare two given metric spaces, or more precisely, how to define a metric structure on $\ms$ that quantifies how far two spaces are from being isometric. Edwards \cite{edwards1975structure} and Gromov \cite{gromov1981groups} independently introduced the notion called \textit{Gromov-Hausdorff distance} to compare metric spaces. This distance is based on the Hausdorff distance.

\begin{definition}[Hausdorff distance]\label{def:hauss}
Given a metric space $Z$, the Hausdorff distance $\dH^Z$ between two subsets $A,B\subseteq Z$ is defined as  
$$d_\mathrm{H}^Z(A,B)=\inf\{r>0:\,B\subseteq A^r,A\subseteq B^r\}, $$
where $A^r\coloneqq\{x\in X:\,d_X(x,A)\leq r\}$ is called the $r$-neighborhood of $A$.
\end{definition}

To compare two metric spaces, we then first isometrically embed them into a common ambient metric space, compute the Hausdorff distance, and then infimize over all such ambient spaces and embeddings. More precisely, we have the following definition.

\begin{definition}[Gromov-Hausdorff distance]\label{def:dgh}
The Gromov-Hausdorff distance $\dgh$ between two compact metric spaces $X$ and $Y$ is defined as 
\begin{equation}\label{eq:dgh-hausdorff}
    \dgh(X,Y)=\inf d^Z_\mathrm{H}(\varphi(X),\psi(Y)),
\end{equation}
where the infimum is taken over all $Z\in\ms$ and isometric embeddings $\varphi:X\hookrightarrow Z$ and $\psi:Y\hookrightarrow Z$.
\end{definition}

\begin{example}
An \textit{$\eps$-net} $S$ of a compact metric space $X$ for $\eps>0$ is a set such that for any $x\in X$, there exists $s\in S$ with $d_X(x,s)\leq \eps$. In other words, $\dH^X(S,X)\leq\eps$ and thus $\dgh(S,X)\leq \eps$.
\end{example}

\begin{remark}\label{rem:dgh-du}
In the definition above, it is enough to restrict $Z$ to the disjoint union $X\sqcup Y$, and then infimize over all metrics $d$ on the disjoint union such that $d|_{X\times X} = d_X$ and $d|_{Y\times Y} = d_Y$ \cite{burago2001course}. We denote by $\mathcal{D}(d_X,d_Y)$ the collection of all such metrics $d$.
\end{remark}

\begin{remark}\label{rem:ub}
For any metric spaces $X$ and $Y$, one can then see that always 
\begin{equation}\label{eq:diamdgh}
    \dgh(X,Y)\leq \frac{1}{2}\max\left(\diam(X),\diam(Y)\right).
\end{equation}

Indeed, by \Cref{rem:dgh-du}, it is enough to consider the metric $d\in\mathcal{D}(d_X,d_Y)$ such that 
$$d(x,y) = \frac{1}{2} \max\left(\diam(X),\diam(Y)\right)$$
for $x\in X$ and $y\in Y$. That the resulting $d$ is a proper metric on the disjoint union $X\sqcup Y$ is easy to see. That the claim in Equation (\ref{eq:diamdgh}) above is true follows now from Definition \ref{def:hauss}.
\end{remark}

It is not hard to check that $\dgh(X,Y)=0$ if and only if $X$ is isometric to $Y$. Moreover, $\dgh$ is a legitimate metric on the collection $\ms$ of isometric classes of compact metric spaces. 

\begin{theorem}[Theorem 7.3.30 in \cite{burago2001course}]\label{thm:legmetric}
$\dgh$ defines a metric on the space $\ms$ of isometry classes of compact metric spaces.
\end{theorem}

It turns out that the Gromov-Hausdorff distance admits a characterization in terms of  distortion of correspondences \cite{burago2001course} as follows. Given two metric spaces $(X,d_X)$ and $(Y,d_Y)$, a correspondence $R$ between the underlying sets $X$ and $Y$ is any subset of $X\times Y$ such that the images of $R$ under the canonical projections are full: $p_X(R)=X$ and $p_Y(R)=Y$. We define the distortion of $R$ with respect to $d_X$ and $d_Y$ as follows:
\begin{equation}\label{eq:dist}
    \dis\lc R,d_X,d_Y\rc\coloneqq\sup_{(x,y),(x',y')\in R}|d_X(x,x')-d_Y(y,y')|.
\end{equation}
We will abbreviate $\dis\lc R,d_X,d_Y\rc$ to $\dis(R)$ whenever the metric structures are clear from the context. Then, the Gromov-Hausdorff distance can be characterized via distortion of correspondences as follows:
\begin{equation}\label{eq:dgh-distortion}
    \dgh(X,Y)=\frac{1}{2}\inf_{R}\dis(R).
\end{equation}
It is shown in \cite{chowdhury2018explicit} that the infimum can always be realized by a closed correspondence.

\begin{remark}\label{rem:lb}
Let $*$ denote the one point metric space. Then, one can prove that for all $X\in\mathcal{M}$, 
$$\dgh(X,\ast) = \frac{1}{2}\diam(X).$$
Furthermore, for all $X$ and $Y$ in $\mathcal{M}$ one has the bound 
$$\frac{1}{2}\big|\diam(X)-\diam(Y)\big|\leq \dgh(X,Y) .$$

To prove the first claim note that the unique correspondence between $X$ and $\ast$ is $R_\ast = X\times \{\ast\}.$ Its distortion is $\dis(R_\ast) = \sup_{x,x'\in X} d_X(x,x') = \diam(X)$ hence the first claim holds.  To prove the second claim note that, since $\dgh$ satisfies the triangle inequality, then $|\dgh(X,\ast)-\dgh(Y,\ast)|\leq \dgh(X,Y) $. The second claim then follows by invoking the first claim.
\end{remark}

Though being theoretically interesting, it is known that computing $\dgh$ is equivalent to solving a quadratic assignment problem \cite{memoli2007use} which turns out to be NP-hard \cite{schmiedl2015shape,schmiedl2017computational,agarwal2018computing}. More precisely,  Schmiedl proved in \cite{schmiedl2015shape,schmiedl2017computational} the following computational complexity for approximating $\dgh$:
\begin{theorem}[{\cite[Corollary 3.8]{schmiedl2017computational}}]\label{thm:NP-hard}
The Gromov-Hausdorff distance between general finite metric spaces cannot be approximated within any factor in $[1,3)$
in polynomial time, unless $\mathcal{P}=\mathcal{NP}$.
\end{theorem}
In fact, the proof of this result reveals that the claim still holds even in the case of \emph{ultrametric spaces}. An ultrametric space $(X,d_X)$ is a metric space which satisfies the \emph{strong triangle inequality}: 
\[\forall x,x',x''\in X, \, d_X(x,x')\leq\max\lc d_X(x,x''),d_X(x'',x')\rc.\]
Ultrametric spaces often arise in the context of data analysis in the form of \textit{dendrograms}: a dendrogram is a certain hierarchical representation of a dataset. It is shown in \cite{carlsson2010characterization} that there exists a structure preserving bijection between the set of dendrograms and the set of ultrametrics on a given finite set.

Being a well understood and highly structured type of metric spaces, we are particularly interested in exploiting possible advantages associated to \emph{adapting} $\dgh$ to the collection of ultrametric spaces. In our previous work \cite{memoli2021gromov}, we identified a method for lowering the factor 3 eventually to 1 in \Cref{thm:NP-hard} by considering a specially tailored family of variants of $\dgh$ on the collection of ultrametric spaces. We explain this as follows.

Let $(X,d_X)$ and $(Y,d_Y)$ be two metric spaces. We reexamine \Cref{thm:NP-hard} via the following modification of \Cref{eq:dgh-distortion}: given $p\in[1,\infty)$, define a quantity $\dghp{p}(X,Y)$ as follows:

\begin{equation}\label{eq:dghp-no-dist}
    \dghp{p}(X,Y)\coloneqq 2^{-\frac{1}{p}}\inf_{R}\sup_{(x,y),(x',y')\in R}\left|(d_X (x,x'))^p- (d_Y (y,y'))^p\right|^\frac{1}{p}.
\end{equation}

It is easy to see that whenever $d_X$ and $d_Y$ are ultrametrics, $(d_X)^p$ and $(d_Y)^p$ are also ultrametrics and thus metrics for any $p>1$. Then, as a consequence of Theorem \ref{thm:NP-hard} and \Cref{eq:dgh-distortion}, we have the following complexity results regarding approximating $\dghp{p}$ when restricted to $\ufin$, the collection of all finite ultrametric spaces:

\nomenclature{$\ufin$}{The collection of all (isometry classes of) finite ultrametric spaces}

\begin{corollary}[{\cite[Corollary 2]{memoli2021gromov}}]\label{thm:approximate}
For each $p\in[1,\infty)$ and for any $X,Y\in\ufin$, $\dghp{p}(X,Y)$ cannot be approximated within any factor in $\left[1,3^\frac{1}{p}\right)$
in polynomial time, unless $\mathcal{P}=\mathcal{NP}$.
\end{corollary}
Note that as $p\rightarrow\infty$, the interval $\left[1,3^\frac{1}{p}\right)$ consisting of `bad' multiplicative factors shrinks to the empty set. In this way, whereas for any $p\in[1,\infty)$, $\dghp p$ is NP-hard to approximate, $\dghp{\infty}\coloneqq\lim_{p\rightarrow\infty}\dghp{p}$ could potentially be a computationally tractable quantity. This turns out to be the case. In \cite{memoli2021gromov} we have established that $\dghp{\infty}$ between any finite ultrametric spaces $X$ and $Y$ can be computed in time $O(n\log(n))$, where $n\coloneqq\max(\#X,\#Y)$, and we discovered that $\dghp{\infty}$ turns out to coincide with the Gromov-Hausdorff ultrametric $\ugh$ first defined by Zarichnyi in \cite{zarichnyi2005gromov} (cf. \Cref{rmk:ugh zarichnyi}) when restricted to the collection $\ums$ of all compact ultrametric spaces. In the sequel, we will hence use $\ugh$ and $\dghp{\infty}$ interchangeably. 

\nomenclature[M]{$\ugh$}{The Gromov-Hausdorff ultrametric}

This observation motivated us to investigate in depth the one parameter family $\{\dghp{p}\}_{p=1}^\infty$. Whereas in our previous work \cite{memoli2021gromov} we focused on devising algorithms for computing $\dghp{p}$ between finite ultrametric spaces, in this paper we focus on the theoretical properties of the one parameter family $\{\dghp{p}\}_{p=1}^\infty$ for general compact metric spaces. We remark that although $\dghp{p}$ is defined for general metric spaces, it is most `compatible' with the collection of the so-called $p$-metric spaces: a metric space $(X,d_X)$ is a $p$-metric space if for any $x,x',x''$ we have that
$$ (d_X(x,x'))^p\leq (d_X(x,x''))^p+(d_X(x'',x'))^p.$$
This point will be discussed in detail in \Cref{sec:dghp}. We summarize our main contributions in the section below.

\subsection{Overview of our results}
We now provide an overview of our results and a discussion of related work.

\paragraph{\Cref{sec:p-arithmetic}.} In this section, we study some elementary properties of $p$-metric spaces and determine notation which will be used throughout the paper. We in particular provide a detailed introduction to ultrametric spaces, including its relationship with \textit{dendrograms} and one of the fundamental operations called \textit{closed quotient operation}.

\paragraph{Section \ref{sec:pms}.} For each $p\in[1,\infty]$, we denote by $\msp$ the collection of all (isometry classes of) compact $p$-metric spaces. Associated with each $p\in[1,\infty]$, there is a natural projection\label{sym:sp} $\mathfrak{S}_p:\ms\rightarrow\msp$ (whose restriction to finite spaces was studied in \cite{Segarra2016metric}) which generalizes the construction of the so-called maximal subdominant ultrametric. Projections such as $\mathfrak{S}_p$ encode a certain notion of \emph{simplification} of a metric space.  We study the relationship between projections with different parameters $p$ and explore some properties of the \emph{kernel} of $\mathfrak{S}_p:$ the kernel of $\mathfrak{S}_p$ is defined as the set of all those metric spaces which are mapped to the one point metric space  under $\mathfrak{S}_p$. Understanding the kernel of $\mathfrak{S}_p$ is interesting because it tells us which metric spaces will be simplified ``too much". 

\nomenclature{$\msp$}{The collection of all (isometry classes of) compact $p$-metric spaces}

\paragraph{Section \ref{sec:dghp}.} In this section, we focus on metric properties of $\dghp{p}$. In particular, in analogy with Theorem \ref{thm:legmetric} we prove that $\dlp$ is a $p$-metric on $\ms$. It is known \cite{kalton1999distances} that the Gromov-Hausdorff distance can be characterized by distortions of maps. We found similar characterizations for $\dlp$. It turns out that when restricted to $\msp$, $\dghp p$ can be characterized by a formula involving Hausdorff distances and isometric embeddings in a way similar to \Cref{eq:dgh-hausdorff} for $\dgh$. We further establish the continuity of the family $\dghp p$ w.r.t. $p$ under mild conditions and show various types of relationships between $\dghp p$ and $\dgh$. Finally, we end this section by establishing connections between $\dghp p$ and approximate isometries.

\paragraph{Section \ref{sec:structural and computation of ugh}.} The ultrametric $\ugh$  has some special properties that make it quite singular among all the metrics $\dlp$. In particular, we establish a structural result (\Cref{thm:ugh-structure}) for $\ugh$ whose restriction to finite spaces allowed us to find a poly time  algorithm for its computation in \cite{memoli2021gromov}. We remark that this structural result is the main tool for the second author to prove that $(\ums,\ugh)$ is an Urysohn universal ultrametric space in \cite{wan2021novel}. In this section, we also relate $\ugh$ to the curvature sets defined by Gromov \cite{gromov2007metric} and study $\widehat{u}_{\mathrm{GH}}$, a modified version of $\ugh$ thus extending work from \cite{memoli2012some}. We also establish a structural result for Hausdorff distances on ultrametric spaces. Finally, we prove that the usual \emph{codistortion} terms in the Gromov-Hausdorff distance is unnecessary for $\ugh$ and that thus $\ugh =\widehat{u}_{\mathrm{GH}}$ (cf. Theorem \ref{thm:ughm}).

\paragraph{Section \ref{sec:interleavings}.} The authors of \cite{carlsson2010characterization} have established a bijective equivalence between ultrametric spaces and dendrograms. It turns out that there exists a natural distance called the \textit{interleaving distance} $\dint$ which can be used to measure discrepancy between two dendrograms. The collection of ultrametric spaces thus inherits this interleaving distance through the bijection mentioned above. The interleaving distance has been widely used in the community of topological data analysis \cite{chazal2009proximity,morozov2013interleaving,bubenik2014categorification,bubenik2015metrics} for comparing persistence modules (a notion which in some instances is related to dendrograms). Because of this, it is interesting to compare the structure and properties of the interleaving distance for ultrametric spaces with those of the Gromov-Hausdorff type distances -- the central object in this paper. In this section, we first reformulate the interleaving distance between ultrametric spaces in a clear form in Theorem
\ref{thm:ibu} which allows us to obtain a characterization of $\dint$ in terms of distortions of maps. We extend this usual interleaving distance to $p$-interleaving distance $\dip$ in a manner similarly to how we define $\dlp$. It turns out that the characterization of $\dint$ in terms of distortions of maps can be extended  to $\dip$. With the help of this characterization, we prove that when restricted to $\ums$, $\dlp$ and $\dip$ are bi-Lipschitz equivalent and in particular $\dgh^{\scriptscriptstyle{(\infty)}}=\ugh$.

\paragraph{Section \ref{sec:topology}.} 
We study convergent sequences of $\dlp$ and establish a pre-compactness theorem for $\dlp$: we prove that any class $\mathfrak{X}$ of $p$-metric spaces  satisfying mild conditions is pre-compact. This implies that $\mathfrak{X}$ is actually \emph{totally bounded}, i.e., for any $\eps>0$, there exists a positive integer $K(\eps)$ and $p$-metric spaces $X_1,\cdots,X_{K(\eps)}$ in $\mathfrak{X}$ such that for any $X\in\mathfrak{X}$, one can find $1\leq i\leq K(\eps)$ such that $\dlp(X,X_i)\leq\eps$. The concept of total boundedness is interesting from the point of view of studying geometric methods for data analysis in that it guarantees that for any given scale parameter $\eps$, one can shatter a given dataset into a \emph{finite} number of pieces each with size not larger than $\eps$.

The pre-compactness theorem also provides us with tools to study the topology of $(\msp,\dlp)$. In particular, we show that $(\msp,\dlp)$ is complete and separable for $1\leq p<\infty$ and, once again $(\ums,\ugh)$ exhibits singular behavior in that it is complete but not separable. This suggests that $\ums$ is rather singular among all other $\msp$. Moreover, we study the subspace topology of $\msp\subseteq\ms_q$ when $p>q$.

We will further study one geometric property of these metric spaces, the geodesic property. It is known \cite{ivanov2016gromov,chowdhury2018explicit} that $(\ms,\dgh)$ is a geodesic space. We introduce a notion called $p$-geodesic spaces which is a generalization of geodesic spaces. We prove that $(\msp,\dlp)$ is a $p$-geodesic space when $p\in[1,\infty)$. Though $(\ums,\ugh)$ is not geodesic, as an application of stability result of the projection $\mathfrak{S}_\infty$, we show that $(\ums,\dgh)$ is geodesic. In the end, we show that $(\ums,\dint)$ is not geodesic.


\paragraph{Related work}

Segarra thoroughly studied finite ultrametric and finite $p$-metric spaces in his PhD thesis \cite{Segarra2016metric}; see also \cite{segarra2015metric}. He was particularly interested in projecting finite networks onto $p$-metric spaces, in the process of which he identified a  canonical projection map $\mathfrak{S}_p$ which we will define in the next section. In the context of finite metric spaces, Segarra proved that such a projection is unique under certain conditions. Segarra considered generalizations of metric spaces beyond $p$-metric spaces and in particular he identified the so-called \textit{dioid metric spaces}. The idea was to generalize not only the addition operator but also the multiplication operator of $\mathbb{R}$. He also studied some theoretical properties of projection maps between different classes of dioid metric spaces. 

The Gromov-Hausdorff ultrametric $\ugh$ on the collection of all compact ultrametric spaces was first introduced by Zarichnyi \cite{zarichnyi2005gromov} in 2005 as an ultrametric counterpart of Gromov-Hausdorff distance $\dgh$. He defined $\ugh$ via the Hausdorff distance formulation (Definition \ref{def:dgh}). He proved that $\ugh$ is an ultrametric on the collection of isometry classes of ultrametric spaces and showed that the space $(\mathcal{U},\ugh)$ is complete but not separable.

Qiu further studied theoretical properties of metric structure induced by $\ugh$ in his 2009 paper \cite{qiu2009geometry}. He found a distortion based description of $\ugh$ in analogy to Equation (\ref{eq:dgh-distortion}), where the infimum is taken over a certain special subset of all correspondences which he called \textit{strong correspondences}. Qiu also established several characterizations of $\ugh$ as Burago et al. did for $\dgh$ in Chapter 7 of \cite{burago2001course}. For example, Qiu modified the definition of $\eps$-isometry and $(\eps,\delta)$-approximation to the so-called strong $\eps$-isometry and strong $\eps$-approximation. He proved that $\ugh(X,Y)<\eps$ if and only if there exists a strong $\eps$-isometry between $X$ and $Y$ if and only if $X$ is a strong $\eps$-approximation of $Y$ which are counterparts to Corollary 7.3.28 and Proposition 7.4.11 of \cite{burago2001course}. More interestingly, Qiu has also found a suitable version of Gromov’s pre-compactness theorem for $(\ums,\ugh)$.


\section{$p$-metric spaces, ultrametric spaces and pseudometric spaces}\label{sec:p-arithmetic}

In this section, we introduce some preliminary results about $p$-metric spaces, ultrametric spaces and pseudometric spaces.

\subsection{$p$-arithmetic and $p$-metric spaces}
We generalize the usual addition operator on $\R_{\geq 0}$ to the $p$-sum as follows.
$$ a\ps p b \coloneqq  
\begin{cases}
(a^p+b^p)^\frac{1}{p},& p\in[1,\infty)\\
\max(a,b),&p=\infty
\end{cases}\quad \forall a,b\geq 0.$$
In fact, $\lc\R_{\geq 0},\ps{p}\rc$ is a commutative monoid (see \cite{howie-book} for general background on monoids), that is, for any $a,b,c\geq 0$, there is an identity element 0 such that $0\ps p a=a\ps p 0=a$; $\big(a\boxplus_p b\big)\boxplus_p c=a\boxplus_p\big(b\boxplus_p c\big)$; and $a\ps p b=b\ps p a$. By associativity, we can add several numbers simultaneously and use the symbol $\ps p$ in the same way as the summation symbol $\Sigma$:
$$\mathop{\ps{\mathrlap{p}}}_{i=1}^n\,a_i=a_1\boxplus_p a_2\boxplus_p \cdots\boxplus_p a_n. $$
An immediate computation will give us that for any $a>0$, $1\leq p\leq \infty$ and $n\in\mathbb{N}$, 
$\displaystyle\mathop{\ps{\mathrlap{p}}}_{i=1}^n\,a=(n)^\frac{1}{p}a$, where we adopt the convention that $\frac{1}{\infty}=0$.

For convenience, we represent the \emph{absolute $p$-difference} between non-negative numbers $a$ and $b$ as \label{sym:lambdap}
$$\Lambda_p(a,b) \coloneqq |a^p-b^p|^\frac{1}{p},\quad \mbox{for }p\in[1,\infty)  $$
and as follows for $p=\infty$ :

\begin{equation*}
\Lambda_\infty(a,b) \coloneqq\begin{cases}
\max(a,b),&a\neq b\\
0, & a= b
\end{cases}
\end{equation*}

\begin{remark}
It is not hard to show that for any $a,b\geq 0$,  $\lim_{p\rightarrow\infty} a\ps{p} b = a\ps{\infty} b$ and $\lim_{p\rightarrow\infty} \Lambda_p(a,b) = \Lambda_\infty(a,b)$.
\end{remark}

Note that for any $a\geq 0$ and any $p\in[0,\infty]$ one has $\Lambda_p(a, 0) = a.$ 

\begin{lemma}[Monotonicity of $\Lambda_p$]\label{lm:increasing lambdap}
For any $a>b> 0$, $\Lambda_p(a,b)$ is an increasing function w.r.t. $p\in[1,\infty]$.
\end{lemma}
\begin{proof}
When $p<\infty$, 
\[\Lambda_p(a,b)=|a^p-b^p|^\frac{1}{p}=a\cdot\ls 1-\lc\frac{b}{a}\rc^p\rs^\frac{1}{p}\leq a=\Lambda_\infty(a,b).\]
Consider the function $f(p)\coloneqq(1-x^p)^\frac{1}{p}$, for a fixed $0< x<1$. Then,
\[f'(p) = \frac{1}{p^2}(1-x^p)^{\frac{1-p}{p}}\left(-x^p\ln(x^p)-(1-x^p)\ln(1-x^p)\right)> 0.\]
Therefore, $f$ is increasing on $[1,\infty)$. If we let $x=\frac{b}{a}$, then we have that for any $1\leq p<q<\infty$
\[\Lambda_p(a,b)\leq \Lambda_q(a,b).\]
Thus, $\Lambda_p(a,b)$ is an increasing function w.r.t. $p\in[1,\infty]$.
\end{proof}

The following result will be used in the sequel.
\begin{proposition}\label{prop:simple-inv-triang}
Let $a,b,c\geq 0$ and let $p\in[1,\infty]$. Then, $a\ps{p} b\geq c$ and $a\ps{p} c\geq b$ hold if and only if $a\geq \Lambda_p(b,c).$
\end{proposition}
\begin{proof}
The statement holds trivially when $p\neq \infty$. Hence, we assume that $p=\infty$.

We first assume that $a\ps{\infty} b\geq c$ and $a\ps{\infty} c\geq b$. Then, we have the following two cases.
\begin{enumerate}
    \item If $b=c$, then $a\geq 0=\Lambda_\infty(b,c).$
    \item If $b\neq c$, we assume without loss of generality that $b>c$. Then, $\max(a,c)=a\ps{\infty} c\geq b$ implies that $a\geq b=\max(b,c)=\Lambda_\infty(b,c)$.
\end{enumerate}

Conversely, assume that $a\geq \Lambda_\infty(b,c).$ Then, we also have two cases.
\begin{enumerate}
    \item If $b=c$, then $a\ps{\infty} b\geq b=c$ and similarly, $a\ps{\infty} c\geq b$.
    \item If $b\neq c$, we assume without loss of generality that $b>c$. Then, $a\geq \Lambda_\infty(b,c)=b$. Therefore, $a\ps{\infty} b\geq b>c$ and $a\ps{\infty} c\geq b\ps{\infty} c=b>c$.
\end{enumerate}
\end{proof}

We also define an asymmetric version of $p$-difference which we will use later.\label{sym:ap}
\begin{equation}\label{eq:asymmetric}
    A_p(a, b) \coloneqq  
\begin{cases}
\Lambda_p(a,b),& a> b\\
0, &a\leq b
\end{cases}
\end{equation}

\begin{proposition}\label{prop:asyineq}
Given $a,b,c\geq 0$ and $p\in[1,\infty]$, assume that $\Lambda_p(a,b)\leq c$. Then,
$$a\geq  A_p(b,c)\text{ and } b\geq  A_p(a,c).$$
\end{proposition}
\begin{proof}
We only need to prove the leftmost inequality. The rightmost inequality follows from essentially the same proof.

When $b\leq c$, by Equation (\ref{eq:asymmetric}), we have $a\geq 0=  A_p(b,c).$

When $b> c$, by Equation (\ref{eq:asymmetric}), we have $  A_p(b,c)=\Lambda_p(b,c).$ We need to consider the following two cases:
\begin{enumerate}
    \item $p=\infty$. If $a< b$, then $\Lambda_\infty(a,b)=b\leq c$ contradicts with $b>c$. So $a\geq b$ and thus $a\geq b=\Lambda_\infty(b,c)$ since $b>c$.
    \item $p\in[1,\infty)$. Then, $\Lambda_p(a,b)\leq c$ results in $ \left|a^p-b^p\right|\leq c^p$. Hence $a^p\geq b^p-c^p=|b^p-c^p|$ and thus, $a\geq \Lambda_p(b,c)$.
\end{enumerate}
\end{proof}

Now, for all $p\in[1,\infty]$, we have the following definition of $p$-metric spaces.

\begin{definition}[$p$-metric space]\label{def:pms}
For $1\leq p \leq \infty$, a $p$-metric space is a metric space $(X,d_X)$ satisfying the \emph{$p$-triangle inequality}:
for all $x,x',x''$  
$$d_X(x,x'')\leq d_X(x,x')\ps{p} d_X(x',x'').$$
\end{definition}

Note that the $1$-triangle inequality is simply the usual triangle inequality and the $\infty$-triangle inequality is exactly the strong triangle inequality for defining ultrametric spaces. We let $\msp$ denote the collection of all (isometry classes of) compact $p$-metric spaces. Then, it is obvious that $\ms_1=\ms$ and $\ms_\infty=\ums$.

\begin{example}[Product $p$-metric]\label{rmk:product-pm}
Let $X,Y\in\msp$. Then, $\lc X\times Y,d_X\ps{p} d_Y\rc\in\msp$. 
\end{example}

\begin{example}\label{ex:pm is pm}
$\Lambda_p$ actually defines a $p$-metric on $\R_{\geq 0}$. $\R^n_{\geq 0}\coloneqq \underbrace{\Rp\times\cdots\times\Rp}_{n\text{ copies of }\Rp}$ has a natural $p$-metric $\Lambda_p^n$ as the product $p$-metric of $\Lambda_p$ defined in Remark \ref{rmk:product-pm}.
\end{example}

\begin{example}
As a generalization of the fact that $(\ums,\ugh)$ is an ultrametric space, we will establish later that $(\msp,\dghp p)$ is a $p$-metric space for all $p\in[1,\infty]$.
\end{example}

\begin{remark}
In the setting of standard metric spaces (i.e. when $p=1$) one has the inequality $|d_X(x,x'')-d_X(x',x'')|\leq d_X(x,x')$ for all $x,x',x''$ in $X$. By Proposition \ref{prop:simple-inv-triang} we have the following general inequality for   any $p\in[1,\infty]$ and any $(X,d_X)\in\mathcal{M}_p$:
$$\Lambda_p(d_X (x,x''),d_X (x',x''))\leq d_X(x,x'),$$
for all $x,x',x'' \in X$.
\end{remark}

The following relation utilizing the $p$-triangle inequality turns out to be useful in the sequel.

\begin{lemma}\label{lm:trick p triangle}
For any $p\in[1,\infty]$, let $X\in\msp$. Then, for any $x_1,x_2,x_3,x_4\in X$ we have that
\[\Lambda_p(d_X(x_1,x_2),d (x_3,x_4)) \leq d_X (x_1,x_3)\ps{p} d_X (x_2,x_4).\]
\end{lemma}
\begin{proof}
Note that by the $p$-triangle inequality, we have that
\[\lc d_X (x_1,x_3)\ps{p} d_X (x_2,x_4)\rc\ps{p}d_X(x_3,x_4)\geq d_X(x_1,x_2) \]
and
\[\lc d_X (x_1,x_3)\ps{p} d_X (x_2,x_4)\rc\ps{p}d_X(x_1,x_2)\geq d_X(x_3,x_4). \]
Therefore, by \Cref{prop:simple-inv-triang} we have that
\[\Lambda_p(d_X(x_1,x_2),d (x_3,x_4)) \leq d_X (x_1,x_3)\ps{p} d_X (x_2,x_4).\]
\end{proof}

Finally, we establish the following relationship between collections $\ms_p$ with different values of $p$:

\begin{proposition}\label{prop:inclusionpq}
For $1\leq q\leq p\leq\infty$ the following inclusions hold:  
$$\mathcal{U}\subseteq\mathcal{M}_p\subseteq \mathcal{M}_q\subseteq \mathcal{M}.$$
\end{proposition}
\begin{proof}
Given $X\subseteq\msp$, we need to show that $X$ satisfies the $q$-triangle inequality for $q\leq p$. Now, for any $x,x',x''\in X$, we have $d_X(x,x')\leq d_X(x,x'')\ps{p} d_X(x'',x')$. Then, it is sufficient to show that $d_X(x,x'')\ps{p} d_X(x'',x')\leq d_X(x,x'')\ps{q} d_X(x'',x')$. We will show in general that $a\ps{p} b\leq a\ps{q} b$ for any $a,b\geq 0$.

The case when $a=0$ is trivial that both sides equal $b$ and the equality will hold. Now, we assume $a>0$. We will consider the function $f(p)\coloneqq(1+x^p)^{\frac{1}{p}}$, for a fixed $x>0$,  $p\geq 1.$ The derivative of this function is
$$f'(p) = \frac{1}{p^2}(1+x^p)^{\frac{1-p}{p}}\left(x^p\ln(x^p)-(1+x^p)\ln(1+x^p)\right)\leq 0.$$
Hence $f$ is non-increasing. Therefore, if we take $x=\frac{b}{a}$, we have $\left(1+\left(\frac{b}{a}\right)^p\right)^\frac{1}{p}\leq \left(1+\left(\frac{b}{a}\right)^q\right)^\frac{1}{q} $ and $a\ps{p} b\leq a\ps{q} b$ for $a,b\geq 0$.
\end{proof}


\paragraph{The snowflake transform}\label{sec:snowflake}

$p$-metric spaces are a special case of a more general notion called \emph{$p$-snowflake metric spaces}. A $p$-snowflake metric space $X$ is a metric space that is bi-Lipschitz equivalent to a $p$-metric space \cite{tyson2005characterizations}. The name snowflake stems from the classical example of fractal metric spaces, the Koch snowflake (see Figure \ref{fig:koch}), which turns out to be a $\lc\log_3{4}\rc$-snowflake metric space.

One way of generating $p$-metric spaces is via the following \textit{snowflake transform}.

\begin{definition}[Snowflake transform \cite{david1997fractured}]\label{ex:snowflake}
For each $p>0$, we define a function $S_p:\Rp\rightarrow\Rp$ by $x\mapsto x^p$. 
For any metric space $(X,d_X)$ and $0<p<\infty$,  we abuse the notation and denote by $S_p(X)$ the new space $(X,S_p\circ\dX)$, where
$$S_p\circ\dX(x,x'):=\lc d_X(x,x')\rc^p,\quad\forall x,x'\in X.$$ 
The map sending $(X,d_X)$ to $\lc X, S_p\circ\dX \rc$ is called the $p$-snowflake transform. 
\end{definition}

Fix any $p\in[1,\infty)$. If $X$ is a compact metric space, it is easy to check that $S_\frac{1}{p}(X)$ is a compact $p$-metric space. Conversely, if $X$ is a compact $p$-metric space, then $S_p(X)$ is a compact metric space. Therefore, the snowflake transform defines the following two maps:
\[S_\frac{1}{p}:\ms\rightarrow \ms_{p}\text{ and }S_p:\ms_p\rightarrow \ms, \,\forall p\in[1,\infty).\]

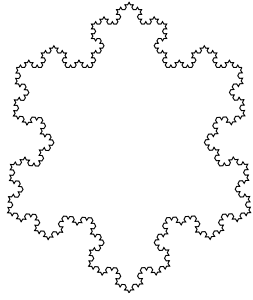
\begin{figure}
    \centering
    \begin{tikzpicture}[decoration=Koch snowflake]
    \draw decorate{decorate{decorate{decorate{(0,0) -- (3,0)}}}};
    \draw decorate{decorate{decorate{decorate{(3,0) -- (1.5,-3)}}}};
    \draw decorate{decorate{decorate{decorate{(1.5,-3) -- (0,0)}}}};
    \end{tikzpicture}
    \caption{\textbf{Koch snowflake.}}
    \label{fig:koch}
\end{figure}

\subsection{Ultrametric spaces}\label{sec:pre-ultrametric}

One of the central topics of this paper is about the ultrametric spaces. Recall that an ultrametric space $(X,d_X)$ is a metric space which satisfies the \emph{strong triangle inequality}: 
\[\forall x,x',x''\in X, \, d_X(x,x')\leq\max\lc d_X(x,x''),d_X(x'',x')\rc.\]
We henceforth use $\uX$ instead of $d_X$ to denote an ultrametric.

The following basic properties of ultrametric spaces are direct consequences of the strong triangle inequality.

\begin{proposition}[Basic properties of ultrametric spaces]\label{prop:basic property ultrametric}
Let $X$ be an ultrametric space. Then, $X$ satisfies the following basic properties:
\begin{enumerate}
    \item (\emph{Isosceles triangles}) Any three distinct points $x,x',x''\in X$ constitute an isosceles triangle, i.e., two of $u_X(x,x'),u_X(x,x'')$ and $u_X(x',x'')$ are the same and are greater than the rest.
    \item (\emph{Center of closed balls}) Consider the closed ball $B_t(x)$ centered at $x\in X$ with radius $t\geq 0$. Then, for any $x'\in B_t(x)$ we have that $B_t(x')=B_t(x)$.
    \item (\emph{Relation between closed balls}) For any two closed balls $B$ and $B'$ in $X$, if $B\cap B'\neq\emptyset$, then either $B\subseteq B'$ or $B'\subseteq B$.
    \item (\emph{Cardinality of spectrum}) Suppose $X$ is a finite space. Then, $\#\spec(X)\leq \#X$. Here the symbol $\#A$ denotes the cardinality of the set $A$.
\end{enumerate}
\end{proposition}

The first three properties follow easily from the strong triangle inequality. As
for the fourth property, see for example \cite[Corollary 3]{gurvich2012characterizing}.

Next, we introduce two important notions for ultrametric spaces: \emph{dendrograms} and \emph{quotient operations}.

\subsubsection{Dendrograms}\label{sec:ultra-dendrogram}

One essential mental picture to evoke when thinking about ultrametric spaces is that of a \emph{dendrogram} (see Figure \ref{fig:ultra-dendro}). To introduce the notion of dendrograms, we first define partitions of a set.
\begin{definition}[Partitions]
Given any set $X$, a \emph{partition} $P$ of $X$ is a set of nonempty subsets $P=\{B_i\subseteq X:\,i\in I\}$, where $I$ denote an index set, such that $X=\cup_{i\in I}B_i$ and $B_i\cap B_j=\emptyset$ when $i\neq j$. We call each $B_i\in P$ a \emph{block} of $P$. We denote by $\mathbf{Part}(X)$ the collection of all partitions of $X$. Given two partitions $P_1,P_2\in\mathbf{Part}(X)$, we say that $P_1$ is a \emph{refinement} of $P_2$, or equivalently, that $P_2$ is \emph{coarser} than $P_1$, if every block in $P_1$ is contained in some block in $P_2$.
\end{definition}

Next we recall the concept that the authors of \cite{carlsson2010characterization} called \textit{dendrograms}. We chose to rename their dendrograms as \textit{finite dendrograms} because their definition assumed all underlying sets to be finite. Since in this paper we need to contend with possibly infinite underlying sets, we will reserve the name `dendrograms' for a more general construction which we introduce in the next section.

\begin{definition}[Finite dendrograms]\label{def:dendrogram}
A finite dendrogram $\theta_X$ over a \emph{finite} set $X$ is any function $\theta_X:[0,\infty)\rightarrow \mathbf{Part}(X)$ satisfying the following conditions:
\begin{enumerate}
    \item[(1)] $\theta_X(0)=\{\{x_1\},\ldots,\{x_n\}\}.$
    \item[(2)] For any $s<t$, $\theta_X(s)$ is a refinement of $\theta_X(t)$.
    \item[(3)] There exists $t_X>0$ such that $\theta_X(t_X)=\{X\}.$
    \item[(4)] For any $t\geq0$, there exists $\eps>0$ such that $\theta_X(t')=\theta_X(t)$ for $t'\in[t,t+\eps].$
\end{enumerate}
\end{definition}

There exists a close relationship between finite dendrograms and finite ultrametric spaces. Fix a finite set $X$, by $\mathcal{U}(X)$ denote the collection of all ultrametrics over $X$ and by $\mathcal{D}(X)$ denote the collection of all dendrograms over $X$. We define a map $\Delta_X:\mathcal{U}(X)\rightarrow\mathcal{D}(X)$ by sending $u$ to a finite dendrogram $\theta_u$ as follows:
for any $t\geq 0$, consider an equivalence relation $\sim\ct{t}$ (this is called $t$-closed equivalence relation in \Cref{sec:quotient operation}) defined by 
$$x\sim\ct{t}x' \text{ iff } u(x,x')\leq t.$$
Then, we define $\theta_u(t)$ to be the partition induced by this equivalence relation. It turns out that the map $\Delta_X$ is bijective. In fact, the inverse $\Upsilon_X:\mathcal{D}(X)\rightarrow\mathcal{U}(X)$ of $\Delta_X$ is the following map: for any dendrogram $\theta$, $u_\theta\coloneqq\Upsilon(\theta_X)$ is defined by 
$$    u_\theta(x,x')\coloneqq\inf\{t\geq 0:\, [x]_{t}^{\theta}=[x']_{t}^{\theta}\}\text{ for any }x,x'\in X,$$
where $[x]_{t}^{\theta}\in\theta(t)$ denotes the block containing $x$\footnote{Sometimes, we abbreviate $[x]_t^\theta$ to $[x]_t$ when the underlying dendrogram is clear from the context.}. We summarize our discussion above into the following theorem.

\nomenclature[S]{$[x]_t$}{The block in the $t$-slice of a dendrogram containing $x$}

\begin{figure}[ht]
    \centering
    \includegraphics[width=0.7\textwidth]{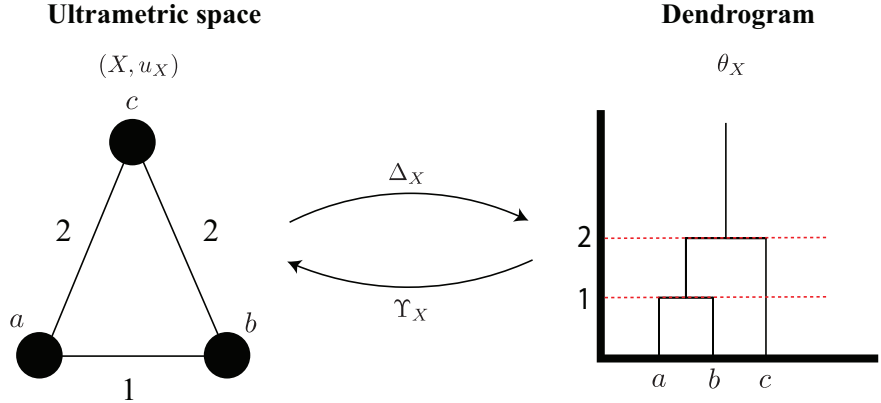}
    \caption{\textbf{Transforming ultrametric spaces into dendrograms.}}
    \label{fig:ultra-dendro}
\end{figure}

\begin{theorem}[Finite dendrograms as finite ultrametric spaces,  {\cite[Theorem 9]{carlsson2010characterization}}]\label{thm:dendroultra}
Given a finite set $X$, then $\Delta_X:\mathcal{U}(X)\rightarrow\mathcal{D}(X)$ is bijective with inverse $\Upsilon_X:\mathcal{D}(X)\rightarrow\mathcal{U}(X)$.
\end{theorem}

Theorem \ref{thm:dendroultra} above establishes that finite dendrograms and finite ultrametric spaces are equivalent concepts -- a point of view which helps to formulate subsequent ideas in this paper.

\subsubsection{Compact ultrametric spaces and dendrograms}

In order to generalize \Cref{thm:dendroultra} to the case of compact ultrametric spaces, we consider the following notion of \textit{dendrograms} as a generalization of finite dendrograms to the case of sets with possibly infinite cardinality. This concept of dendrograms was first studied and called \textit{proper dendrograms} in \cite{memoli2021ultrametric}.

\begin{definition}[Dendrograms]\label{def:proper dendrogram}
Given a set $X$ (not necessarily finite), a \emph{dendrogram} $\theta_X:[0,\infty)\rightarrow \mathbf{Part}(X)$ is a map satisfying the following conditions:
\begin{enumerate}
\item[(1)] $\theta_X(0)$ is the finest partition consisting only singleton sets;
\item[(2)] $\theta_X(s)$ is finer than $\theta_X(t)$ for any $0\leq s<t<\infty$;
\item[(3)] There exists $t_X>0$ such that for any $t\geq t_X$, $\theta_X(t_X)=\{X\}$ is the trivial partition;			
\item[(4')] For each $t> 0$, there exists $\eps>0$ such that $\theta_X(t)=\theta_X(t')$ for all $t'\in[t,t+\eps]$.
	\item[(5)] For any distinct points $x,x'\in X$, there exists $t_{xx'}>0$ such that $x$ and $x'$ belong to different blocks in $\theta_X(t_{xx'})$.
			\item[(6)] For each $t>0$, $\theta_X(t)$ consists of only finitely many blocks.
			\item[(7)] Let $\{t_n\}_{n\in\mathbb N}$ be a decreasing sequence such that $\lim_{n\rightarrow\infty}t_n=0$ and let $X_n\in \theta_X(t_n)$. If for any $1\leq n<m$, $X_m\subseteq X_n$, then $\bigcap_{n\in\mathbb N}X_n\neq\emptyset$. 
		\end{enumerate}
	\end{definition}
	
\begin{remark}[Finite dendrograms v.s. dendrograms]
First notice the difference between condition (4') in \Cref{def:proper dendrogram} and condition (4) in \Cref{def:dendrogram} that the inequality for $t$ in condition (4') is strict. When the underlying set $X$ is finite, it is not hard to see that condition (4') and condition (5) together imply condition (4) in \Cref{def:dendrogram}. Hence, a dendrogram over a finite set is exactly a finite dendrogram defined in \Cref{def:dendrogram}. Conversely, a finite dendrogram is obviously a dendrogram over a finite set.
\end{remark}

Similar to \Cref{thm:dendroultra}, who considered the relationship between finite ultrametric spaces and {dendrograms}, there is a structure preserving bijection between compact ultrametric spaces and dendrograms \cite[Theorem 2.2]{memoli2021ultrametric}. 

\begin{theorem}\label{thm:compact ultra-dendro}
Given a set $X$, denote by $\mathcal{U}(X)$ the collection of all compact ultrametrics on $X$ and $\mathcal{D}(X)$ the collection of all dendrograms over $X$. For any $\theta\in\mathcal{D}(X)$, consider $u_\theta$ defined as follows:
\[\forall x,x'\in X,\,\,\, u_\theta(x,x')\coloneqq\inf\{t\geq 0\,:\,x,x' \text{ belong to the same block of } \theta(t)\}.\]
Then, $u_\theta\in\mathcal{U}(X)$ and the map $\Upsilon_X:\mathcal{D}(X)\rightarrow\mathcal{U}(X)$ sending $\theta$ to $u_\theta$ is a bijection.
\end{theorem}

\subsubsection{The closed quotient operation}\label{sec:quotient operation}
There is one special and fundamental equivalence relation on ultrametric spaces, whose induced quotient operation will be useful in later sections.

\paragraph{A `closed' equivalence relation.}\label{closed relation}For any ultrametric space $(X,u_X)$, we introduce a relation $\sim\ct{t}$ on $X$ such that 
\begin{equation}\label{eq:closed relation}
    x\sim\ct{t} x'\text{ iff }u_X(x,x')\leq t.
\end{equation}
This equivalence relation was used in defining the bijective map $\Delta_X:\mathcal{U}(X)\rightarrow\mathcal{D}(X)$ in the previous section. Due to the strong triangle inequality, $\sim\ct{t}$ is an equivalence relation, which we call the \emph{closed equivalence relation}. For each $x\in X$ and $t\geq 0$, denote by $[x]\ct{t}^X$ the equivalence class of $x$ under $\sim\ct{t}$. We abbreviate $[x]\ct{t}^X$ to $[x]\ct{t}$ whenever the underlying set is clear from the context. Consider the set $X\ct{t}\coloneqq\{[x]\ct{t}:\,x\in X\}$ of all $\sim\ct{t}$ equivalence classes.

\nomenclature[S]{$X\ct{t}$}{$t$-closed quotient of $X$}

\nomenclature[S]{$[x]\ct{t}$}{$t$-closed equivalence class of $x$}

\begin{remark}[Relationship with closed balls]\label{rmk:relation with closed ball}
Note that for any $x\in X$ and any $t\geq 0$, the equivalence class $[x]\ct{t}^X$ satisfies $[x]\ct{t}^X=\{x'\in X:\,u_X(x,x')\leq t\}$. This implies that $[x]\ct{t}$ coincides with the closed ball $ {B}_t(x)\coloneqq\{x'\in X:\,u_X(x,x')\leq t\}$. 
\end{remark}

Now, we introduce the function $u_{X\ct{t}}:X\ct{t}\times X\ct{t}\rightarrow\mathbb R_{\geq 0}$ defined as follows: 
\begin{equation}\label{eq:closed quotient}
    u_{X\ct{t}}\left([x]\ct{t},[x']\ct{t}\right) \coloneqq  \left\{
\begin{array}{cl}
u_X(x,x') & \mbox{if $[x]\ct{t}\neq[x']\ct{t}$}\\
0 & \mbox{if $[x]\ct{t}=[x']\ct{t}$.}
\end{array}
\right.
\end{equation}

It is clear that $u_{X\ct{t}}$ is an ultrametric on $X\ct{t}$. 
\begin{definition}[$t$-closed quotient]\label{def:ultraquotient}
For any ultrametric space $(X,u_X)$ and $t\geq 0$, 
we call $\lc X\ct{t},u_{X\ct{t}}\rc $ the $t$-\emph{closed quotient of $X$}.
\end{definition}

\begin{remark}
Consider the dendrogram $\theta_X$ associated to $(X,u_X)$. Then, for any $t\geq 0$ and $x\in X$, it turns out that the block $[x]_{t}^{\theta}\in\theta_X(t)$ introduced in the previous section coincides with the equivalence class $[x]\ct{t}$. Moreover, the dendrogram associated to $X\ct{t}$ is closely related to $\theta_X$ as shown in \Cref{fig:closed quotient}.
\end{remark}

\begin{remark}[Open equivalence relation]
One can also define an open equivalence relation by requiring strict inequality in \Cref{eq:closed relation}. In \cite{memoli2021gromov}, this open equivalence relation turns out to be useful in devising algorithms for computing $\dgh$.
\end{remark}

\begin{figure}
    \centering
    \includegraphics[width=0.6\linewidth]{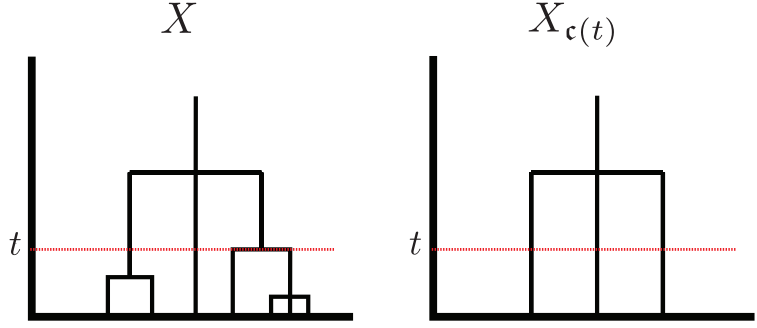}
    \caption{\textbf{Illustration of the $t$-closed quotient.} The dendrogram associated to $X\ct{t}$ is obtained by forgetting details below $t$ of the dendrogram associated to $X$.}
    \label{fig:closed quotient}
\end{figure}

It is worth noting that the quotient of a compact ultrametric space is still compact, and the quotient of a Polish ultrametric space remains Polish \cite{wan2021novel}. Furthermore, we have the following two more refined results. The first proposition was already mentioned in \cite[Lemma A.7]{memoli2021ultrametric} and we include it here for completeness.

\begin{proposition}\label{lm:compact-quotient}
Let	$X$ be a complete ultrametric space. Then, $X$ is \emph{compact} if and only if for any $t>0$, $X\ct{t}$ is a \emph{finite} space. 
\end{proposition}
\begin{proof}
Assume that $X$ is compact. Then, for any $t>0$ there exists a finite $t$-net $X_N\subseteq X$, i.e., for any $x\in X$, there exists $x_N\in X_N$ such that $u_X(x,x_N)\leq t$. Hence, $[x]\ct{t}=[x_N]\ct{t}$. Therefore, $X\ct{t}\subseteq\{[x_N]\ct{t}:\,x_N\in X_N\}$ and thus $X\ct{t}$ is a finite set.
	    
Conversely, we assume that $X\ct{t}$ is finite for all $t>0$. We only need to prove that $X$ is totally bounded to conclude that $X$ is compact. For any $t>0$, $X\ct{t}$ is a finite set and thus there exists $x_1,\ldots,x_n\in X$ such that $X\ct{t}=\left\{[x_1]\ct{t},\ldots,[x_n]\ct{t}\right\}.$ Now, for any $x\in X$, there exists $i\in\{1,\ldots,n\}$ such that $x\in[x_i]\ct{t}$. This implies that $u_X(x,x_i)\leq t$. Therefore, the set $\{x_1,\ldots,x_n\}\subseteq X$ is a $t$-net of $X$. Hence, $X$ is totally bounded and thus compact. 
\end{proof}

\begin{proposition}\label{lm:countable-quotient}
Let $X$ be a complete ultrametric space. Then, $X$ is \emph{separable} (and thus Polish) if and only if for any $t>0$, $X\ct{t}$ is a \emph{countable} space.
\end{proposition}
\begin{proof}
Assume that $X$ is separable. Let $X_c\subseteq X$ be a countable dense subset. Then, for any $t>0$, $\{[x_c]\ct{t}:\,x_c\in X_c\}$ is a countable subset of $X\ct{t}$. For any $x\in X$, there exists $x_c\in X_c$ such that $u_X(x,x_c)\leq t$ since $X_c$ is dense. Hence, $[x]\ct{t}=[x_c]\ct{t}$. Therefore, $X\ct{t}\subseteq\{[x_c]\ct{t}:\,x_c\in X_c\}$ and thus $X\ct{t}$ is countable.

Conversely, consider any positive sequence $\{t_n\}_{n=1}^\infty$ strictly decreasing to $0$. For each $t_n$, we choose for each equivalence class in $X\ct{t_n}$ an arbitrary representative, and denote by $\hat{X}\ct{t_n}$ the set of all the chosen representatives. Then, $\hat{X}\ct{t_n}$ is of course a countable set. Let $X_c\coloneqq \cup_{n=1}^\infty \hat{X}\ct{t_n}$. Then, $X_c$ is obviously countable since each $\hat{X}\ct{t_n}$ is countable. Now, for any $x\in X$, there exist $x_1,x_2,\cdots\in X_c$ such that for each $n=1,\ldots$, $x\in[x_n]\ct{t_n}$. Then, 
\[\lim_{n\rightarrow\infty}u_X(x,x_n)\leq\lim_{n\rightarrow\infty}t_n=0. \]
This implies that $X_c$ is dense in $X$ and thus $X$ is separable.
\end{proof}

\nomenclature[S]{$\ums$}{The collection of all (isometry classes) of compact ultrametric spaces}

The following lemma is useful in constructing maps between closed quotient spaces.

\begin{lemma}\label{lm:induce}
Let $f:X\rightarrow Y$ be a 1-Lipschitz map between two ultrametric spaces. For any $t\geq 0$, define $f_t:X\ct{t}\rightarrow Y\ct{t}$ by $f_t\lc[x]^X\ct{t}\rc\coloneqq [f(x)]^Y\ct{t}$. Then, $f_t$ is well-defined, i.e., whenever $[x]\ct{t}^X=[x']\ct{t}^X$, we have that $[f(x)]^Y\ct{t}=[f(x')]^Y\ct{t}$.
\end{lemma}

\begin{proof}
If $x'\in [x]^X\ct{t}$, then we have $u_X(x,x')\leq t$. Since $f$ is 1-Lipschitz, we have that 
$$u_Y(f(x'),f(x))\leq u_X(x,x')\leq t,$$
and thus $f(x')\in[f(x)]^Y\ct{t}$. Therefore, $f_t$ is well-defined.
\end{proof}


\subsection{Pseudometric spaces}\label{sec:pseudometric space}
In this section, we provide certain details to address nuances between the notion of metric spaces and the notion of pseudometric spaces. A pseudometric space is a pair $(X,d_X)$ consisting of a set $X$ and a function called pseudometric $d_X:X\times X\rightarrow\R$ satisfying the following three conditions, for any $x,x',x''\in X$:
\begin{enumerate}
    \item $d_X(x,x')\geq 0$.
    \item $d_X(x,x')=d_X(x',x)$.
    \item $d_X(x,x')+d_X(x',x'')\geq d_X(x,x'')$.
\end{enumerate}
A pseudometric $d_X$ is a metric if $d_X$ satisfies that $d_X(x,x')=0$ implies $x=x'$.

There is a canonical way of transforming a pseudometric space into a metric space \cite[Proposition 1.1.5]{burago2001course}. Given any pseudometric space $(X,d_X)$, we introduce an equivalence relation $\sim_0$\footnote{If $(X,d_X)$ is moreover an ultrametric space, then $\sim_0$ is exactly the same as the $0$-closed equivalence relation $\sim\ct{0}$ introduced in \Cref{sec:quotient operation}.} as follows: $x\sim_0 x'$ if $d_X(x,x')=0$. Then, we let $[x]$ denote the equivalence class (and thus an element in $X_0$) of $x\in X$ and define the quotient space $X_0:=X/\sim_0$. Define a function $d_{X_0}:X_0\times X_0\rightarrow\mathbb R_{\geq 0}$ as follows:
\begin{equation}\label{eq:quotientmetric}
    d_{X_0}([x],[x'])\coloneqq\begin{cases}
 	d_X(x,x')&\text{if }d_X(x,x')\neq 0\\
 	0&\text{otherwise}
	\end{cases}.
\end{equation}
$d_{X_0}$ turns out to be a metric on $X_0$.
In the sequel, the metric space $(X_0,d_{X_0})$ is referred to as the \emph{metric space induced by the pseudometric space} $(X,d_X)$.

\paragraph{Quotient metric.}
One way of generating pseudometric spaces out of a metric space is through the quotient metric (cf. \cite[Definition 3.1.12]{burago2001course}). 

\begin{definition}[Quotient metric]\label{def:quotientmetric}
Let $(X,d_X)$ be a metric space and let $\sim_R$ be any equivalence relation on $X$. We denote by $X_R\coloneqq X/\sim_R$ the set of equivalence classes. Then, we define the quotient metric $d_{X_R}:X_R\times X_R\rightarrow\Rp$ as follows (see Figure \ref{fig:quotientmetric} for an illustration):
$$d_{X_R} ([x],[x'])\coloneqq \inf \sum_{i=1}^nd_X(x_i,y_i),$$
where the infimum is taken over all sequences of points $x=x_1,y_1,x_2,y_2,\cdots,x_n,y_n=x'$ in $X$ such that $y_i\sim_R x_{i+1}$ for all $i=1,\cdots,n-1$.
\end{definition}

It is not hard to see that $d_{X_R}$ is indeed a pseudometric on $X_R$ (cf. \cite[Exercise 3.1.13]{burago2001course}).

\begin{figure}
    \centering
    \includegraphics[width=12cm]{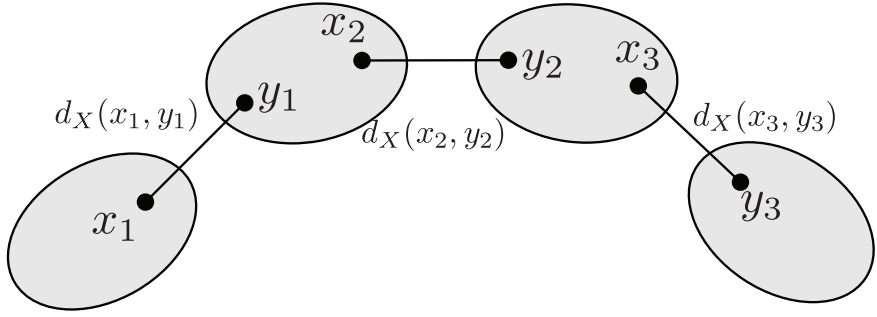}
    \caption{\textbf{Illustration of the quotient metric.} In this figure, each ball represents an equivalence class of $\sim_R$ on a metric space $X$. Here, we represent one sequence of points $x=x_1,y_1,\cdots,y_n=x'$ between $x$ and $x'$ with $n=3$. Then, $d_{X_R} ([x],[x'])$ is the infimum of the sum $\sum_{i=1}^n d_X(x_i,y_i)$ over all such sequences of points.}
    \label{fig:quotientmetric}
\end{figure}
\begin{remark}\label{rmk:quotientineq}
It is easy to see that for any $x,x'\in X$, we have that
$$d_{X_R}([x],[x'])\leq d_X(x,x'). $$
\end{remark}

\section{The projections $\mathfrak{S}_p:\mathcal{M}\rightarrow \msp$}\label{sec:pms}

For each $p\in[1,\infty]$, we define below a canonical projection ${\mathfrak{S}_p}:{\ms}\rightarrow{\msp}$ sending a compact metric space $X$ to a $p$-metric space $\lc\hat{X},\hat{d}_X^{\scriptscriptstyle{(p)}}\rc $.

Given $({X},{d}_X)\in {\ms}$, define for any $x,x'\in X$

\begin{equation}\label{eq:defdxp}
    \dxp(x,x')\coloneqq \inf\left\{\mathop{\ps{\mathrlap{p}}}_{\,i=0}^{n-1}d_X(x_i,x_{i+1}):\,x=x_0,x_1,\cdots,x_n=x'\right\}.
\end{equation}

\begin{remark}\label{rmk:dpleqd}
Obviously, $\dxp(x,x')\leq d_X(x,x')$ for any $x,x'\in X$.
\end{remark}

$(X,\dxp)$ may happen to be a pseudometric space instead of a metric space, i.e., there may exist $x\neq x'\in X$ such that $\dxp(x,x')=0$. To remedy this, consider the induced metric space $\lc\hat{X},\hat{d}_X^{\scriptscriptstyle{(p)}}\rc $ (cf. \Cref{sec:pseudometric space}) and define $\mathfrak{S}_p((X,d_X))\coloneqq\lc\hat{X},\hat{d}_X^{\scriptscriptstyle{(p)}}\rc .$

\begin{remark}\label{rmk:finitesp}
If $X$ is a finite space, then $\lc\hat{X},\hat{d}_X^{\scriptscriptstyle{(p)}}\rc =(X,\dxp)$.
\end{remark}

\begin{remark}[Single-linkage ultrametric]\label{rmk:single linkage ultrametric}
When $X$ is a finite space, the metric $\dX^{\scriptscriptstyle{(\infty)}}=\hat{d}_X^{\scriptscriptstyle{(\infty)}}$ is an ultrametric and it is called the \emph{single-linkage ultrametric} induced by $d_X$.
\end{remark}

We now verify that indeed $\mathfrak{S}_p$ maps compact metric spaces to compact $p$-metric spaces and thus $\mathfrak{S}_p$ is indeed a map $\mathfrak{S}_p:\ms\rightarrow\msp$.

\begin{proposition}
For every $(X,d_X)\in\ms$,  $\lc\hat{X},\hat{d}_X^{\scriptscriptstyle{(p)}}\rc $ is a compact $p$-metric space.
\end{proposition}

\begin{proof}
We first prove that $\hat{d}_X^{\scriptscriptstyle{(p)}}$ is a $p$-metric. Given $x,x',x''\in X$, for any two sequences of points $x=x_0,\cdots,x_n=x'$ and $x'=y_0,\cdots,y_m=x''$ in $X$, we concatenate them to obtain the following sequence of points between $x$ and $x''$: $x=z_0,\cdots,z_{m+n+1}=x''$. Then,
$$\lc\mathop{\ps{\mathrlap{p}}}_{i=0}^{n-1}d_X(x_i,x_{i+1})\rc\ps{p} \lc\mathop{\ps{\mathrlap{p}}}_{\,j=0}^{m-1}d_X(y_j,y_{j+1})\rc= \mathop{\ps{\mathrlap{p}}}_{\,k=0}^{n+m}d_X(z_k,z_{k+1})\geq \dxp(x,x'').$$
Infimizing the left-hand side over all possible sequences of points, one has that
$$\dxp(x,x')\ps{p}\dxp(x',x'')\geq\dxp(x,x''). $$
It then follows directly from \Cref{eq:quotientmetric} that $\hat{d}_X^{\scriptscriptstyle{(p)}}$ satisfies the $p$-triangle inequality and thus is a $p$-metric.

The identity map $\iota:(X,d_X)\rightarrow(X,\dxp)$ is continuous due to Remark \ref{rmk:dpleqd}. The canonical projection $\Psi:(X,\dxp)\rightarrow\lc\hat{X},\hat{d}_X^{\scriptscriptstyle{(p)}}\rc $ is also continuous (cf. \Cref{sec:pseudometric space}). Since $(X,d_X)$ is compact, we then have that $\lc\hat{X},\hat{d}_X^{\scriptscriptstyle{(p)}}\rc =\Psi\circ\iota\left((X,d_X)\right)$ is compact.
\end{proof}

\begin{proposition}[Basic facts about ${\mathfrak{S}}_p$]\label{prop:propofsp}
We have the following properties about ${\mathfrak{S}}_p$:

\begin{enumerate}
    \item For $1\leq p\leq \infty$, when restricted to ${\msp}$, ${\mathfrak{S}}_p$ coincides with the identity map. 
    \item For $1\leq q<p\leq \infty$, one has ${\mathfrak{S}}_p\circ{\mathfrak{S}}_q={\mathfrak{S}}_p ={\mathfrak{S}}_q\circ{\mathfrak{S}}_p$.
    \item Given $X\in{\msp}$ and $c>0$, ${\mathfrak{S}}_p(c\cdot X)=c\cdot {\mathfrak{S}}_p(X)$, where $c\cdot X$ denotes the metric space $(X,c \cdot d_X ).$
    
    \item $\mathfrak{S}_\infty$  commutes with the snowflake transform $S_p$ for any $1\leq p<\infty$. More precisely, for any $X\in \msp$, we have
    $$\mathfrak{S}_\infty\circ S_p(X)=S_p\circ\mathfrak{S}_\infty(X).$$
\end{enumerate}
\end{proposition}

The following theorem shows that $p$-metric spaces can be viewed as a certain interpolation between metric spaces and ultrametric spaces. 

\begin{proposition}
Given any finite metric space $X$, the curve $\gamma:[0,\infty]\rightarrow(\ms,\dgh)$ defined by $p\mapsto \mathfrak{S}_p(X)$ is continuous.
\end{proposition}
\begin{proof}
Since $X$ is finite, by \Cref{rmk:finitesp}, we have that $\hat{d}_X^{\scriptscriptstyle{(p)}}={d}_X^{\scriptscriptstyle{(p)}}$. Now,
for any $x,x'\in X$ and any sequence of points $x=x_0,x_1,\cdots,x_n=x'$ in $X$, $\displaystyle\mathop{\ps{\mathrlap{p}}}_{i=0}^{n-1}\,d_X(x_i,x_{i+1})$ is continuous with respect to change of $p\in[1,\infty]$. Then, we have by definition
$$\dxp(x,x')\coloneqq \inf\left\{\mathop{\ps{\mathrlap{p}}}_{\,i=0}^{n-1}d_X(x_i,x_{i+1}):\,x=x_0,x_1,\cdots,x_n=x'\right\}$$
is also continuous with respect to $p\in[1,\infty]$ due to the finiteness of $X$. Then, by finiteness again, $\sup_{x,x'\in X}\dxp(x,x')$ is continuous with respect to $p\in[1,\infty]$. Therefore, by Example 7.4.2 in \cite{burago2001course}, $\gamma$ is also continuous with respect to $\dgh$.
\end{proof}

Below,  $\ms^{\mathrm{fin}}$ and $\msp^{\mathrm{fin}}$ denote the collection of finite metric spaces and $p$-metric spaces, respectively.

In the setting of finite spaces, Segarra et al. \cite{segarra2015metric} proved that $\mathfrak{S}_p$ is actually the unique projection satisfying the following two reasonable conditions.

\begin{theorem}[{\cite{segarra2015metric,Segarra2016metric}}]
Let $p\in[1,\infty]$ and $\Phi_p:\ms^{\mathrm{fin}}\rightarrow\msp^{\mathrm{fin}}$ be any map satisfying the following two conditions:
\begin{enumerate}
    \item Any $p$-metric space is a fixed point of $\Phi_p$.
    \item Any $1$-Lipschitz map in $\ms^{\mathrm{fin}}$ remains $1$-Lipschitz in $\msp^{\mathrm{fin}}$ after applying $\Phi_p$.
\end{enumerate} Then, $\Phi_p$ exactly coincides with the restriction $\mathfrak{S}_p|_{\ms^{\mathrm{fin}}}.$
\end{theorem}

\paragraph{An alternative description of $\mathfrak{S}_p$.}
We now provide an alternative description of $\mathfrak{S}_p$. Let $\sim_{\dxp}$ denote the canonical equivalence relation on $X$ induced by the pseudometric $\dxp$ (cf. \Cref{sec:pseudometric space}). We denote by $\hat{d}_X$ the quotient metric on $\hat{X}\coloneqq X/\sim_{\dxp}$ defined in \Cref{def:quotientmetric}.

\begin{lemma}\label{lm:quotientdxpineq}
For any compact metric space $X$, we have that $\hat{d}_X\geq \hat{d}_X^{\scriptscriptstyle{(p)}}$.
\end{lemma}

\begin{proof}
By Remark \ref{rmk:dpleqd}, we have that for any equivalence classes $[x],[x']\in \hat{X}$ 
\begin{align*}
    \hat{d}_X ([x],[x'])&= \inf \sum_{i=1}^nd_X(x_i,y_i)\geq \inf \sum_{i=1}^n\dxp(x_i,y_i)\\
    &=\inf \sum_{i=1}^{n}\dxp(x_i,x_{i+1})\geq \dxp(x,x')=\hat{d}_X^{\scriptscriptstyle{(p)}}([x],[x']),
\end{align*}
where the second equality follows from the fact that $\dxp(y_i,x_{i+1})=0$.
\end{proof}

\begin{corollary}\label{coro:diam}
Given $(X,d_X)\in\mathcal{M}$, for any $1\leq p\leq \infty$ we have 
$$\diam\left(\mathfrak{S}_p(X)\right)\leq \diam(X). $$
\end{corollary}
\begin{proof}
This follows directly from Remark \ref{rmk:quotientineq} and \Cref{lm:quotientdxpineq}.
\end{proof}

\begin{lemma}\label{lm:quotientmetric}
$\hat{d}_X$ is a metric on $\hat{X}$.
\end{lemma}
\begin{proof}
Since $\hat{d}_X$ is known to be a pseudometric, we only need to prove that $\hat{d}_X$ satisfies identity of indiscernibles, i.e., $\hat{d}_X(x,x')$ implies that $x=x'$. This follows directly from Lemma \ref{lm:quotientdxpineq} and the fact that $\hat{d}_X^{\scriptscriptstyle{(p)}}$ is a metric on $\hat{X}$.
\end{proof}

\begin{proposition}\label{prop:subd}
For a compact metric space $X$, we have that $\hat{d}_X^{\scriptscriptstyle{(p)}}=\left(\hat{d}_X\right)^{\scriptscriptstyle{(p)}}$, that is, we can obtain $\mathfrak{S}_p((X,d_X))$ by first taking the quotient of $X$ with respect to $\sim_{\dxp}$ and then applying the transformation defined in Equation (\ref{eq:defdxp}). See Figure \ref{fig:spillustration} for an illustration.
\end{proposition}

\begin{proof}
By Lemma \ref{lm:quotientdxpineq}, we have that $\hat{d}_X\geq \hat{d}_X^{\scriptscriptstyle{(p)}}$. Then, it is obvious from Equation (\ref{eq:defdxp}) that $$\left(\hat{d}_X^{\scriptscriptstyle{(p)}}\right)^{\scriptscriptstyle{(p)}}\leq\left(\hat{d}_X\right)^{\scriptscriptstyle{(p)}}. $$
By item 1 of Proposition \ref{prop:propofsp}, we have that $\left(\hat{d}_X^{\scriptscriptstyle{(p)}}\right)^{\scriptscriptstyle{(p)}}=\hat{d}_X^{\scriptscriptstyle{(p)}}$. Thus $\hat{d}_X^{\scriptscriptstyle{(p)}}\leq\left(\hat{d}_X\right)^{\scriptscriptstyle{(p)}}.$

On the other hand, we know by Remark \ref{rmk:quotientineq} that for any $[x],[x']\in\hat{X}$, $\hat{d}_X([x],[x'])\leq d_X(x,x')$. Then, it follows again from Equation (\ref{eq:defdxp}) that 
$$\left(\hat{d}_X\right)^{\scriptscriptstyle{(p)}}([x],[x'])\leq \dxp(x,x')=\hat{d}_X^{\scriptscriptstyle{(p)}}([x],[x']) $$ which  concludes the proof.
\end{proof}

\begin{figure}
    \centering
    \begin{tikzcd}
    \ms \arrow[r,"\tilde{\mathfrak{S}}_p"] \arrow[d,"\mathfrak{T}_p"] \arrow[dr, dashed, "\mathfrak{S}_p"]
    & \tilde{\ms}_p \arrow[d,"\mathfrak{T}"] \\
    \ms \arrow[r,"\tilde{\mathfrak{S}}_p"]
    & \msp
    \end{tikzcd}
    \caption{\textbf{Illustration of two ways of generating $\mathfrak{S}_p$}. By $\tilde{\mathcal{M}}_p$ denote the collection of $p$-pseudometric spaces, by $\tilde{\mathfrak{S}}_p:\ms\rightarrow\tilde{\ms}_p$ denote the map sending $(X,d_X)$ to $(X,\dxp)$, by $\mathfrak{T}:\tilde{\ms}_p\rightarrow\msp$ denote the canonical quotient map sending a pseudometric space to its induced metric space (cf. \Cref{sec:pseudometric space}), and by $\mathfrak{T}_p:\ms\rightarrow\ms$ denote the map sending $(X,d_X)$ to $(\hat{X},\hat{d}_X)$ under the relation $\sim_{\dxp}$. Then, $\mathfrak{S}_p$ is such that the above diagram commutes.}
    \label{fig:spillustration}
\end{figure}

\paragraph{Subdominant properties.} We establish the subdominant property of $\hat{d}_X^{\scriptscriptstyle{(p)}}.$ For $p=\infty$, the subdominant property of $\mathfrak{S}_\infty$ restricted to the collection of all finite metric spaces was already established in \cite{carlsson2010characterization}. First of all, we introduce some notation. For any set $X$, let $\mathcal{D}_p(X)$ denote the collection of all $p$-metrics on $X$. We define a partial order $\leq$ on $\mathcal{D}_p(X)$ by letting $d_1\leq d_2$ for $d_1,d_2\in \mathcal{D}_p(X)$ if and only if $\forall x,x'\in X$, $d_1(x,x')\leq d_2(x,x')$.

\begin{proposition}[Maximal subdominant $p$-metric]
Given $(X,d_X)\in\ms$ and $p\in[1,\infty]$, consider $\lc\hat{X},\hat{d}_X^{\scriptscriptstyle{(p)}}\rc $, the $p$-metric space generated by $\mathfrak{S}_p$. Then, $$\hat{d}_X^{\scriptscriptstyle{(p)}}=\max\lb d:\,d\in\mathcal{D}_p\lc\hat{X}\rc,\text{ and }d\leq \hat{d}_X\rb.$$
\end{proposition}

\begin{proof}
By Lemma \ref{lm:quotientdxpineq}, we have $\hat{d}_X^{\scriptscriptstyle{(p)}}\leq \hat{d}_X$. Now, let $d\in\mathcal{D}_p\lc\hat{X}\rc$ , then $\hat{d}^{\scriptscriptstyle{(p)}}=d$ by item 1 of Proposition \ref{prop:propofsp}. Moreover, if $d\leq \hat{d}_X$, then it is easy to check that $d^{\scriptscriptstyle{(p)}}\leq \left(\hat{d}_X\right)^{\scriptscriptstyle{(p)}}=\dxp$ where the last equality follows from Proposition \ref{prop:subd}. Therefore, $d=d^{\scriptscriptstyle{(p)}}\leq\dxp$.
\end{proof}

\subsection{Stability of $\mathfrak{S}_p$}
For the projection $\mathfrak{S}_\infty:\ms\rightarrow\ums$, the following  stability result is already in the literature\footnote{Although \cite[Proposition 26]{carlsson2010characterization} states it only for finite cases, the same proof works for compact spaces.}.

\begin{proposition}[{\cite[Proposition 26]{carlsson2010characterization}}]\label{thm:slstab} The map $\mathfrak{S}_\infty:\mathcal{M}\rightarrow \mathcal{U}$ is $1$-Lipschitz, i.e., 
for any $X,Y\in\ms$,
$$\dgh(\mathfrak{S}_\infty(X),\mathfrak{S}_\infty(Y))\leq\dgh(X,Y). $$
\end{proposition}
 
As a generalization of \Cref{thm:slstab}, we prove the following stability result for $\mathfrak{S}_p$ for all $p\in(1,\infty].$

\begin{proposition}\label{thm:optimal}
Given two finite metric spaces $X$ and $Y$ with $\#X=m$ and $\#Y=n$, and $p> 1$, we have 
$$\dgh(\mathfrak{S}_p(X),\mathfrak{S}_p(Y))\leq \left(\max(m,n)-1\right)^{\frac{1}{p}}\,\dgh(X,Y). $$
\end{proposition}

\begin{remark}
\Cref{thm:optimal} does not include the case $p=1$ because $\mathfrak{S}_1$ is just the identity map by Proposition \ref{prop:propofsp}. Obviously, Theorem  \ref{thm:optimal} recovers \Cref{thm:slstab} by choosing $p=\infty$ when the spaces are finite.
\end{remark}

\begin{proof}[Proof of \Cref{thm:optimal}]
The case when $m=n=1$ reduces to comparing two one-point sets, a case in which the equality obviously holds. If either $m=1$ or $n=1$, then we obtain the inequality by invoking Corollary \ref{coro:diam} and the fact that $\dgh(X,*)=\frac{1}{2}\diam(X)$, where $*$ denotes the one point metric space.

Now, we assume that $m,n>1$. By Remark \ref{rmk:finitesp} we know that $\mathfrak{S}_p(X)=(X,\dxp)$ and $\mathfrak{S}_p(Y)=(Y,\dyp)$. Let $R\in\mathcal{R}(X,Y)$ be an optimal correspondence for $\dgh(X,Y)$. Let $\eta\coloneqq\mathrm{dis}(R,d_X,d_Y)=2\dgh(X,Y)$. Then, for any $(x,y),(x',y')\in R$, $|d_X(x,x')-d_Y(y,y')|\leq\eta$. Now, let us bound $\big|\dxp(x,x')-\dyp(y,y')\big|$. Let $y=y_0,y_1,\cdots,y_k=y'$ be a sequence of points in $Y$ such that
$\displaystyle\dyp(y,y')=\mathop{\ps{\mathrlap{p}}}_{i=0}^{k-1}d_Y(y_i,y_{i+1}),$
whose existence follows from the fact that $Y$ is finite. Then, we choose any sequence of points $x=x_0,x_1,\cdots,x_k=x'$ in $X$ such that $(x_i,y_i)\in R$ for all $i=0,\cdots,k$. Therefore, by definition of $\mathfrak{S}_p$, we have

\begin{align*}
\dxp(x,x')& \leq\mathop{\ps{\mathrlap{p}}}_{i=0}^{k-1}d_X(x_i,x_{i+1})\leq \mathop{\ps{\mathrlap{p}}}_{i=0}^{k-1}\left(d_Y(y_i,y_{i+1})+\eta\right)\\
& \leq \mathop{\ps{\mathrlap{p}}}_{i=0}^{k-1}d_Y(y_i,y_{i+1})+\mathop{\ps{\mathrlap{p}}}_{i=0}^{k-1}\eta=\dyp(y,y')+k^{\frac{1}{p}}\eta,
\end{align*}
where the third inequality follows from the Minkowski inequality. 

Note that sequences of points determining $\dyp$ in $Y$ can always be chosen such that $k\leq m-1$ hence we have 
$$\dxp(x,x')\leq \dyp(y,y')+(m-1)^{\frac{1}{p}}\,\eta.$$
Similarly, $\dyp(y,y')\leq \dxp(x,x')+(n-1)^{\frac{1}{p}}\eta.$ Therefore, we have
$$ \big|\dxp(x,x')-\dyp(y,y')\big|\leq \left(\max(m,n)-1\right)^{\frac{1}{p}}\,\eta.$$

Then, 
\begin{align*}
    &\dgh(\mathfrak{S}_p(X),\mathfrak{S}_p(Y))\leq\frac{1}{2}\mathrm{dis}(R,\dxp,\dyp)\\
    \leq&\frac{1}{2}\left(\max(m,n)-1\right)^{\frac{1}{p}}\,\mathrm{dis}(R,d_X,d_Y)= \left(\max(m,n)-1\right)^{\frac{1}{p}}\,\dgh(X,Y).
\end{align*}

\end{proof}

\begin{example}[The coefficient in \Cref{thm:optimal} is optimal]\label{ex:opt-bound}
Let $(L_n,d_{n})$ be the finite subspace $\{0,1,\cdots,n\}\subseteq\R$ endowed with the Euclidean metric $d_{n}$. Let $p>1$ and let $L_n^{\scriptscriptstyle{(p)}}\coloneqq \mathfrak{S}_p(L_n)$. Then, $\diam(L_n^{\scriptscriptstyle{(p)}})=d_{n}^{\scriptscriptstyle{(p)}}(0,n)=n^\frac{1}{p}$ and $d_{n}^{\scriptscriptstyle{(p)}}(i,i+1)=1$ for $i=0,\cdots,n-1$.

Let $\tilde{L}_n^{\scriptscriptstyle{(p)}}\coloneqq  n^{\frac{1}{p}-1}\cdot L_n$. Then, $\diam\!\lc\tilde{L}_n^{\scriptscriptstyle{(p)}}\rc=\diam(L_n^{\scriptscriptstyle{(p)}})$.  
By considering the diagonal correspondence $R$ between $L_n^{\scriptscriptstyle{(p)}}$ and $\tilde{L}_n^{\scriptscriptstyle{(p)}}$, we have $\dgh\!\lc L_n^{\scriptscriptstyle{(p)}},\tilde{L}_n^{\scriptscriptstyle{(p)}}\rc\leq\frac{1}{2}\mathrm{dis}(R)=\frac{1}{2}\left(1-n^{\frac{1}{p}-1}\right)$.

Note that  $L_n^{\scriptscriptstyle{(p)}}\in\mathcal{M}_p$. Thus by Proposition \ref{prop:propofsp} we have $\mathfrak{S}_p(L_n^{\scriptscriptstyle{(p)}})=L_n^{\scriptscriptstyle{(p)}}$. For $\tilde{L}_n^{\scriptscriptstyle{(p)}}$, we have that $\diam\left(\mathfrak{S}_p\!\lc\tilde{L}_n^{\scriptscriptstyle{(p)}}\rc\right)=n^{\frac{2}{p}-1}$. Hence by \Cref{rem:lb} we have
$$\dgh\!\left(\mathfrak{S}_p(L_n^{\scriptscriptstyle{(p)}}),\mathfrak{S}_p\!\lc\tilde{L}_n^{\scriptscriptstyle{(p)}}\rc\right)\geq\frac{1}{2}\big|\diam(L_n^{\scriptscriptstyle{(p)}})-\diam\left(\mathfrak{S}_p\!\lc\tilde{L}_n^{\scriptscriptstyle{(p)}}\rc\right)\big|=\frac{1}{2}\left(n^{\frac{1}{p}}-n^{\frac{2}{p}-1}\right).$$

Therefore, we have that 
$$\frac{\dgh\!\left(\mathfrak{S}_p(L_n^{\scriptscriptstyle{(p)}}),\mathfrak{S}_p\!\lc\tilde{L}_n^{\scriptscriptstyle{(p)}}\rc\right)}{\dgh\!\lc L_n^{\scriptscriptstyle{(p)}},\tilde{L}_n^{\scriptscriptstyle{(p)}}\rc}\geq \frac{n^{\frac{1}{p}}-n^{\frac{2}{p}-1}}{1-n^{\frac{1}{p}-1}}=n^{\frac{1}{p}}, $$
which can be rewritten as 
 $${\dgh\!\left(\mathfrak{S}_p(L_n^{\scriptscriptstyle{(p)}}),\mathfrak{S}_p\!\lc\tilde{L}_n^{\scriptscriptstyle{(p)}}\rc\right)}\geq n^{\frac{1}{p}}\, {\dgh\!\lc L_n^{\scriptscriptstyle{(p)}},\tilde{L}_n^{\scriptscriptstyle{(p)}}\rc}. $$

By \Cref{thm:optimal}, we have ${\dgh\!\left(\mathfrak{S}_p(L_n^{\scriptscriptstyle{(p)}}),\mathfrak{S}_p\!\lc\tilde{L}_n^{\scriptscriptstyle{(p)}}\rc\right)}\leq \left((n+1)-1\right)^{\frac{1}{p}}\, {\dgh\!\lc L_n^{\scriptscriptstyle{(p)}},\tilde{L}_n^{\scriptscriptstyle{(p)}}\rc}. $ Hence we have that 
$${\dgh\!\left(\mathfrak{S}_p(L_n^{\scriptscriptstyle{(p)}}),\mathfrak{S}_p\!\lc\tilde{L}_n^{\scriptscriptstyle{(p)}}\rc\right)}= n^{\frac{1}{p}}\, {\dgh\!\lc L_n^{\scriptscriptstyle{(p)}},\tilde{L}_n^{\scriptscriptstyle{(p)}}\rc}. $$
Therefore, the bound in \Cref{thm:optimal} is tight. Since for $1<p<\infty$ the sequence $\{n^\frac{1}{p}\}_{n\in\mathbb{N}}$ is unbounded. This example also shows that the map $\mathfrak{S}_p:\ms\rightarrow\msp$ is not Lipschitz.
\end{example}

Finally, if we restrict $\mathfrak{S}_\infty$ to $\msp\subseteq\ms$ for any $p\in[1,\infty]$, we have the following stability result utilizing the $p$-Gromov-Hausdorff distance $\dghp{p}$ which we define later in \Cref{sec:dghp}. We postpone its proof to \Cref{sec:relation with dgh}.

\begin{theorem}\label{thm:p-stable}
For any $p\in[1,\infty]$ and any $X,Y\in\msp$, we have that
$$\dghp{p}\left(\mathfrak{S}_\infty(X),\mathfrak{S}_\infty(Y)\right)\leq \dghp{p}(X,Y). $$
\end{theorem}

\subsection{The kernel of $\mathfrak{S}_p$}\label{sec:kernel}

In this section we study the notion of \emph{kernel} of maps $\mathfrak{S}_p:\ms\rightarrow\msp$ for $p\in[1,\infty]$. The \emph{kernel} $\ker(\mathfrak{S}_p)$ is defined as the collection of compact metric spaces whose image under $\mathfrak{S}_p$ is the one point space. It is obvious that $X\in\ker(\mathfrak{S}_p)$ iff $\dxp(x,x')= 0$ for all $x,x'\in X$.

A metric space $(X,d_X)$ is said to be \emph{chain connected} if for any $x,x'\in X$ and any $\varepsilon>0$ there exists a finite sequence $x=x_0,x_1,\ldots,x_n=x'$ such that $d_X(x_i,x_{i+1})\leq\varepsilon$ for all $i$. Then, it follows directly from the definition of $\mathfrak{S}_\infty$ that $\ker(\mathfrak{S}_\infty) = \mathcal{M}^\mathrm{chain}$, where $\mathcal{M}^\mathrm{chain}$ refers to the collection of all compact chain connected metric spaces. Since we are only considering compact metric spaces, a result in \cite{alonso2008varepsilon} shows that chain connectedness is equivalent to connectedness. Therefore, we have the following result for the kernel of $\mathfrak{S}_\infty$. Below, $\ms^\mathrm{conn}$ denotes the collection of all compact connected  metric spaces.

\begin{proposition}\label{prop:inftykernel}
$\ker(\mathfrak{S}_\infty)= \mathcal{M}^\mathrm{conn}.$ 
\end{proposition}

\begin{remark}\label{rmk:kernel-inf-proj}
Any geodesic metric space is connected. Therefore, any compact geodesic metric space lies in the kernel of $\mathfrak{S}_\infty.$
\end{remark}

Proposition \ref{prop:inftykernel} will not hold for $p<\infty.$ In fact, we have the following result.

\begin{proposition}\label{prop:inclusion}
Given $1<q<p\leq\infty$, we have that
$$ \ker(\mathfrak{S}_q)\subsetneq\ker(\mathfrak{S}_p).$$
\end{proposition}

\begin{proof}
By Proposition \ref{prop:propofsp}, $\mathfrak{S}_p\circ\mathfrak{S}_q=\mathfrak{S}_p$. Thus $\ker(\mathfrak{S}_q)\subseteq\ker(\mathfrak{S}_p).$

Now, consider the following example. Let $X=([0,1],d_X)\subseteq\R$ be a Euclidean subspace. Then, as mentioned in \Cref{sec:p-arithmetic}, $S_{\frac{1}{p}}(X)=\left([0,1],(d_X)^\frac{1}{p}\right)$ is a $p$-metric space for $1<p<\infty$. Hence, $S_{\frac{1}{p}}(X)\notin\ker\mathfrak{S}_p$ since $\mathfrak{S}_p\lc S_{\frac{1}{p}}(X)\rc=S_{\frac{1}{p}}(X)$. Obviously, $S_{\frac{1}{p}}(X)$ is connected, thus $S_{\frac{1}{p}}(X)\in\ker(\mathfrak{S}_\infty)$ which implies $ \ker(\mathfrak{S}_p)\subsetneq\ker(\mathfrak{S}_\infty).$ 

Next, we show that $S_{\frac{1}{q}}(X)\in\ker(\mathfrak{S}_p)$ for $1\leq q<p<\infty.$ For any $0\leq x<x'\leq 1$ let $l\coloneqq x'-x$. Subdivide the interval $[x,x']$ into $n$ equal subintervals to  obtain $x=x_0,\cdots,x_n=x'$ such that $x_{i+1}-x_i=\frac{l}{n}$ for $i=0,\cdots,n-1.$ Then, we have that
\begin{align*}
    \lc(d_X)^\frac{1}{q}\rc^{\scriptscriptstyle{(p)}}(x,x')\leq\mathop{\ps{\mathrlap{p}}}_{i=0}^{n-1}(d_X)^\frac{1}{q}(x_i,x_{i+1})=n^{\frac{1}{p}-\frac{1}{q}}\,l^\frac{1}{q}.
\end{align*}
Since $\frac{1}{p}-\frac{1}{q}<0$, as $n\rightarrow\infty$, we derive that $((d_X)^\frac{1}{q})^{\scriptscriptstyle{(p)}}(x,x')=0$. Therefore, $S_{\frac{1}{q}}(X)\in\ker(\mathfrak{S}_p)$. Since $S_{\frac{1}{q}}(X)\notin\ker(\mathfrak{S}_q)$, we have that $\ker(\mathfrak{S}_q)\subsetneq\ker(\mathfrak{S}_p).$
\end{proof}

Proposition \ref{prop:inclusion} above leads us to considering the following object: $$\bigcap_{p>1}\ker(\mathfrak{S}_p).$$ We have not yet fully described this set but we conjecture that it coincides exactly with $\ms^1$, the collection of all  $1$-connected compact metric spaces (defined below):

\begin{conjecture}\label{conj:M1 = intersection}
 $\ms^1= \mathop{\bigcap}_{p>1}\ker(\mathfrak{S}_p).$
\end{conjecture}

Next, we will define $1$-connected compact metric spaces and provide a partial answer to the conjecture above.

We first recall the definition of \emph{Hausdorff dimension}. See for example \cite[Section 1.7]{burago2001course} for more details. Let $(X,d_X)\in\mathcal{M}$ and let $s\geq 0$. For any $A\subseteq X$ the \emph{$s$-dimensional Hausdorff content} of $A$ is defined by

$$C_\mathcal{H}^s(A)\coloneqq \inf\left\{\sum_{i\in I} r_i^s :\,A\subseteq \bigcup_{i\in I} B\ot{r_i}(x_i),\text{ where }I\text{ is a countable index set} \right\}.$$

The Hausdorff dimension of $A$ is defined as
$$\dim_\mathcal{H}(A)\coloneqq \inf\{s\geq 0:\,C^s_\mathcal{H}(A)=0\}.$$

\begin{definition}[$1$-connected metric spaces]
We say that $(X,d_X)\in \mathcal{M}$ is \emph{$1$-connected} if for all $x$ and $x'$ in $X$, there exists a closed connected subset $C \subseteq X$ such that $C\ni x,x'$ and $\dim_\mathcal{H}(C)=1$. By $\ms^1$ we denote the collection of all compact 1-connected metric spaces.
\end{definition}

\begin{example}
Compact geodesic spaces are obviously $1$-connected. An example of a compact metric space which is $1$-connected yet not geodesic is   the unit circle in $\mathbb{R}^2$ endowed with the Euclidean metric. 
\end{example}

\begin{example}[Nonexamples via the snowflake transform] For any $1< p<\infty$, the space $S_{\frac{1}{p}}(X)=\left([0,1],(d_X)^\frac{1}{p}\right)$ constructed in the proof of Proposition \ref{prop:inclusion} has Hausdorff dimension $p$ (e.g., see \cite{semmes2003notes}) and is therefore not $1$-connected.
\end{example}

As a partial answer to \Cref{conj:M1 = intersection}, we have:
\begin{proposition}
 $\ms^1\subseteq \bigcap_{p>1}\ker(\mathfrak{S}_p).$
\end{proposition}

\begin{proof}
Let $X\in\ms^1$. For any $x,x'\in X$, let $K$ be a 1-dimensional connected closed (and thus compact) subset of $X$ containing them. Then, for any $p>1$, $C_\mathcal{H}^p(K)=0$. Hence for any $0<\eps<1$ we may find a finite\footnote{Finiteness follows from compactness of $K$.} cover $\{B\ot{r_i}(x_i)\}_{i\in I}$ of $K$ such that $\sum_{i\in I}r_i^p<\eps^p$. By connectedness of $K$, the 1-skeleton of the nerve of this cover is a connected graph, hence any two vertices (balls) are connected by a path on the graph. Without loss of generality, assume that $x\in B\ot{r_1}(x_1)$ and $x'\in B\ot{r_k}(x_k)$, and that $\{B\ot{r_1}(x_1),B\ot{r_2}(x_2),\cdots,B\ot{r_k}(x_k)\}$ is a path in the nerve. Choose $y_i\in B\ot{r_i}(x_i)\cap B\ot{r_{i+1}}(x_{i+1})$ for $i=1,\cdots,k-1$ and then construct a sequence of points $x,x_1,y_1,x_2,y_2,\cdots,y_{k-1},x_k,x'$. Then, we have 
\begin{align*}
\dxp(x,x')\leq&d_X(x,x_1)\ps{p} d_X(x_1,y_1)\ps{p} d_X(y_1,x_2)\ps{p}\cdots\ps{p} d_X(x_k,x') \\
< & r_1\ps{p} r_1\ps{p} r_2\ps{p} r_2\ps{p}\cdots \ps{p} r_k=2^\frac{1}{p}\mathop{\ps{\mathrlap{p}}}_{i=1}^k\,r_i<2^\frac{1}{p}\eps.
\end{align*}
Since $\eps$ is arbitrary, we have that $d_X^{\scriptscriptstyle{(p)}}(x,x')=0$.
\end{proof}


\section{$\dghp{p} $: $p$-Gromov-Hausdorff distance} \label{sec:dghp}
Recall that in \Cref{sec:intro}, we have defined for any metric spaces $X$ and $Y$ and for any $p\in[1,\infty)$ a Gromov-Hausdorff like quantity
\[\dghp{p}(X,Y)\coloneqq 2^{-\frac{1}{p}}\inf_{R}\sup_{(x,y),(x',y')\in R}\left|(d_X (x,x'))^p- (d_Y (y,y'))^p\right|^\frac{1}{p}.\]
In this section, we study various properties and characterizations of $\dghp p$.

Note that we only define $\dghp p$ for $p\in[1,\infty)$ through the above equation. Below, we introduce the notion of $p$-distortion for all $p\in[1,\infty]$ and redefine $\dghp p$ for all $p\in[1,\infty]$ in a unified way.

Given any $p\in[1,\infty]$, for two metric spaces $(X,d_X)$ and $(Y,d_Y)$, the $p$-distortion of a correspondence $R$ between $X$ and $Y$ is defined as 

\begin{equation}\label{eq:distortionp}
    \disp{p}(R,d_X,d_Y)\coloneqq \sup_{(x,y),(x',y')\in R}\Lambda_p( d_X (x,x'), d_Y (y,y')).
\end{equation}
When the underlying metric structures are clear, we will abbreviate $\disp{p}(R,d_X,d_Y)$ as $\disp{p}(R)$.

Note that for $p=1$, $\disp{1}$ coincides with the usual notion of distortion $\dis$ of a correspondence given in Equation (\ref{eq:dist}).  Via the $p$-distortion, $\dghp{p}$ can be rewritten as follows:
\begin{equation}\label{eq:dghp distortion formula}
    \dghp{p}(X,Y)\coloneqq 2^{-\frac{1}{p}}\inf_{R}\disp{p}(R).
\end{equation}
Similarly, when $p=\infty$, we define
\begin{equation}\label{eq:dgh infinity distortion formula}
    \dghp{\infty}(X,Y)\coloneqq \inf_{R}\disp{\infty}(R),
\end{equation}
where  the implicit coefficient is $2^{-\frac{1}{\infty}}=1$. 

\begin{remark}[Monotonicity of $\dghp p$]\label{rmk:increading dghp}
By \Cref{lm:increasing lambdap} and the definition of $\dghp p$, for any metric spaces $X,Y$ and any $1\leq p<q\leq \infty$, we have that
\[\dghp p(X,Y)\leq \dghp q(X,Y).\]
\end{remark}

It turns out that $\dghp{p}$ is indeed a metric on $\ms$. More precisely, $\dghp p$ is a $p$-metric. We will hence call $\dghp p$ the $p$-Gromov-Hausdorff distance.

\begin{theorem}
For any $p\in[1,\infty]$, $(\ms,\dghp p)$ is a $p$-metric space.
\end{theorem}
\begin{proof}
When $X\cong Y$, it is easy to see that $\dghp{p}(X,Y)=0$. Conversely, recall that by Remark \ref{rmk:increading dghp}, $\dghp{p}\geq \dgh$. Hence, $\dghp{p}(X,Y)=0$ implies that $\dgh(X,Y)=0$ and thus $X\cong Y$. 

Given any $X,Y,Z\in\mathcal{M}$, we prove that  $$\dghp{p}(X,Y)\leq\dghp{p}(X,Z)\ps{p} \dghp{p}(Y,Z).$$

Let $R_1\in\mathcal{R}(X,Z)$ and $R_2\in\mathcal{R}(Y,Z)$. Define the correspondence $R$ between $X$ and $Y$ by 
$$R:=\{(x,y)\in X\times Y:\,\exists z\in Z, \text{ such that } (x,z)\in R_1\text{ and }(y,z)\in R_2\}. $$ 
Now, for any $(x_1,y_1),(x_2,y_2)\in R$, there exist $z_1,z_2\in Z$ such that $(x_i,z_i)\in R_1$ and $(y_i,z_i)\in R_2$ for $i=1,2$. Then,
\begin{align*}
    \Lambda_p(d_X(x_1,x_2),d_Y(y_1,y_2))&\leq \Lambda_p(d_X(x_1,x_2),d_Z(z_1,z_2))\ps p\Lambda_p(d_Y(y_1,y_2),d_Z(z_1,z_2))\\
    &\leq \disp p(R_1)\ps p\disp p(R_2).
\end{align*}
Thus, $\disp{p}(R)\leq\disp{p}(R_1)\ps{p} \disp{p}(R_2)$, which implies that $$\dghp{p}(X,Y)\leq\dghp{p}(X,Z)\ps{p}\dghp{p}(Y,Z).$$ 
\end{proof}

\paragraph{A Kalton and Ostrovskii type characterization for $\dghp p$.} Given two metric spaces $X$ and $Y$, we define the distortion of any map $\varphi:X\rightarrow Y$ by
\begin{equation}\label{eq:dismap}
    \dis(\varphi,d_X,d_Y)\coloneqq\sup_{x,x'\in X}\big|d_X(x,x')-d_Y(\varphi(x),\varphi(x'))\big|.
\end{equation}
Given another map $\psi:Y\rightarrow X$, we define the codistortion of the pair of maps  $(\varphi,\psi)$ by
\begin{equation}\label{eq:codismap}
    \codis(\varphi,\psi,d_X,d_Y)\coloneqq \sup_{x\in X,y\in Y}\big|d_X(x,\psi(y))-d_Y(\varphi(x),y)\big|. 
\end{equation}
When the underlying metric structures are clear from the context, we will usually abbreviate $\dis(\varphi,d_X,d_Y)$ and $\codis(\varphi,\psi,d_X,d_Y)$ to $\dis(\varphi)$ and $\codis(\varphi,\psi)$, respectively.

Then, one has the following formula, \cite[Theorem 2.1]{kalton1999distances}:

\begin{equation}\label{eq:distmap}
    \dgh(X,Y)=\frac{1}{2}\inf_{\substack{\varphi:X\rightarrow Y \\ \psi:Y\rightarrow X}}\max\left(\dis(\varphi),\dis(\psi),\codis(\varphi,\psi)\right). 
\end{equation}

In the case of $\dghp p$, a similar formula also holds. Let $X,Y\in\ms$. We then define the $p$-distortion of any map $\varphi:X\rightarrow Y$ by
\begin{equation}\label{eq:disp}
    \disp{p}(\varphi,d_X,d_Y)\coloneqq \sup_{x,x'\in X}\Lambda_p( d_X (x,x'), d_Y (\varphi(x),\varphi(x'))).
\end{equation}
Similarly, given any map $\psi:Y\rightarrow X$, we define the $p$-codistortion of the pair $(\varphi,\psi)$ by
\begin{equation}
    \codis_p(\varphi,\psi,d_X,d_Y)\coloneqq \sup_{x\in X,y\in Y}\Lambda_p(d_X (x,\psi(y)), d_Y (\varphi(x),y)).
\end{equation}
We will use abbreviations $\disp{p}(\varphi)$ and $\codis_p(\varphi,\psi)$ when the underlying metric structures are clear from the context.

Then, one can easily derive from \Cref{eq:dghp distortion formula} and \Cref{eq:dgh infinity distortion formula} the following Kalton and Ostrovskii type formula for $\dghp p$, which is analogous to Equation (\ref{eq:distmap}). 

\begin{proposition}\label{thm:dist}
For $X,Y\in\ms$ and $p\in[1,\infty]$, one has that
\begin{equation}
    \dghp{p}(X,Y)= 2^{-\frac{1}{p}}\inf_{\substack{\varphi:X\rightarrow Y \\ \psi:Y\rightarrow X}}\max\left(\disp{p}(\varphi),\disp{p}(\psi),\codis_p(\varphi,\psi)\right).
\end{equation}
\end{proposition}

\subsection{Characterization of $\dghp p$ on $\msp$}\label{sec:charact dghp on msp}
Recall from \Cref{def:dgh} that the Gromov-Hausdorff distance $\dgh$ is defined via isometric embeddings and the Hausdorff distance. It is natural to wonder whether $\dghp p$ can be similarly characterized. The answer is yes if we restrict $\dghp p$ to $\msp$. Due to this reason, when referring to $\dghp p$ in the sequel, $\dghp p$ is usually restricted to $\msp$. 

\begin{theorem}\label{thm:dGHp}
Given any $p\in[1,\infty]$ and $X,Y\in\mathcal{M}_p$, the $p$-Gromov-Hausdorff distance between them can be characterized as follows:
$$\dghp{p}\left(X,Y\right)=\inf_{Z}\dH^Z\left(\varphi_X\left(X\right),\varphi_Y\left(Y\right)\right), $$
where the infimum is taken over all $Z\in\msp$ and isometric embeddings $\varphi_X:X\hookrightarrow Z$ and $\varphi_Y:Y\hookrightarrow Z$.
\end{theorem}

\begin{remark}[Alternative formulation]
The collection of $Z$ in \Cref{thm:dGHp} can be restricted as follows. For $X,Y\in\mathcal{M}_p$ let $\mathcal{D}_p(d_X,d_Y)$ denote the set of all $p$-metrics  $ d :X\sqcup Y \times X\sqcup Y\rightarrow \R_{\geq 0}$ such that $ d |_{X\times X} =  d_X $ and $ d |_{Y\times Y} =  d_Y $. Then, it is easy to see that
$$\dghp{p}\left(X,Y\right)=\inf_{ d \in\mathcal{D}_p(d_X,d_Y)}\dH^{(X\sqcup Y, d )}\left(\varphi_X\left(X\right),\varphi_Y\left(Y\right)\right). $$
\end{remark}

\begin{example}\label{ex:dlpnet}
Fix any $p\in[1,\infty]$ and any $X\in\msp$. Given any $\eps>0$ and any $\eps$-net $S$ in $X$, we have that
\[\dghp{p}(S,X)\leq\dH^X(S,X)\leq \eps.\]
\end{example}

\begin{remark}\label{rmk:msp is pms}
Since $\dghp p$ is a $p$-metric on $\ms$, its restriction to $\msp$, still denoted by $\dghp p$, is also a $p$-metric, i.e., $(\msp,\dghp p)$ is a $p$-metric space. This result is structurally satisfactory in that the collection of compact $p$-metric spaces itself is a $p$-metric spaces.
\end{remark}

Zarichnyi introduced in \cite{zarichnyi2005gromov} the \textit{Gromov-Hausdorff ultrametric}, which we denote by $\ugh$, in the same way as we characterize the $\infty$-Gromov-Hausdorff distance in \Cref{thm:dGHp}:

\begin{definition}[Gromov-Hausdorff ultrametric]
For any $X,Y\in\mathcal{U}$, the Gromov-Hausdorff ultrametric $\ugh$ between them is defined as follows:
$$\ugh\left(X,Y\right):=\inf_{Z}\dH^Z\left(\varphi_X\left(X\right),\varphi_Y\left(Y\right)\right), $$
where the infimum is taken over all $Z\in\ums$ and isometric embeddings $\varphi_X:X\hookrightarrow Z$ and $\varphi_Y:Y\hookrightarrow Z$.
\end{definition}

\Cref{thm:dGHp} is then a generalization of the following result:
\begin{proposition}[{\cite[Theorem 26]{memoli2021gromov}}]\label{rmk:ugh zarichnyi}
For any $X,Y\in\mathcal{U}$, we have that
\[\dghp{\infty}(X,Y)=\ugh(X,Y).\]
\end{proposition}

We will henceforth use $\dghp\infty$ and $\ugh$ interchangeably in this paper.


Now, we finish this section by proving \Cref{thm:dGHp}.

\begin{proof}[Proof of \Cref{thm:dGHp}]
To proceed with the proof, we need the following claim which is obvious from the definition of $\dghp{p}$ and \cite[Proposition 2.1]{memoli2011gromov}.

\begin{claim}\label{claim:dghp inf R,d}
$\dghp{p}(X,Y)\coloneqq \inf_{R, d } \sup_{(x,y)\in R} d (x,y),$
where $R\in\mathcal{R}(X,Y)$ and $ d \in\mathcal{D}_p(d_X,d_Y).$
\end{claim}

Assume that $\eta>\dghp{p}(X,Y)$ for some $\eta>0$. Let $ d \in\mathcal{D}_p(d_X,d_Y)$ and $R\in\mathcal{R}(X,Y)$ be such that $ d (x,y)<\eta$ for all $(x,y)\in R$. Then, one has for any $(x,y),(x',y')\in R$ that
$$\Lambda_p(d_X (x,x'), d_Y (y,y')) = \Lambda_p(d (x,x'),d (y,y')) \leq d (x,y)\ps{p} d (x',y')<\eta\ps{p}\eta=2^\frac{1}{p}\eta,$$
where the first inequality follows from \Cref{lm:trick p triangle} and we use the convention $\frac{1}{\infty}=0$.
Thus, by taking supremum over all pairs $(x,y),(x',y')\in R$ on the left-hand side, one has
$$\disp{p}(R)\leq 2^{\frac{1}{p}}\,\eta.$$

By taking infimum of the left-hand side over all correspondences $R$ between $X$ and $Y$ and letting $\eta$ approach $\dghp{p}(X,Y)$, we obtain that $\dghp{p}(X,Y) \geq {2^{-\frac{1}{p}}}\inf_{R}\disp{p}(R).$

For the opposite inequality, assume that $R\in\mathcal{R}(X,Y)$ and $\eta>0$ are such that $\disp{p}(R)\leq {2^{\frac{1}{p}}\,\eta}$. Consider $ d \in\mathcal{D}_p(d_X,d_Y)$ given by
\begin{enumerate}
        \item $d|_{X\times X}\coloneqq d_X$ and $d|_{Y\times Y}\coloneqq d_Y$;
        \item for any $(x,y)\in X\times Y$, $d (x,y)\coloneqq\inf_{(x',y')\in R}d_X(x,x')\ps{p} d_Y (y',y)\ps{p}\eta$;
        \item for any $(y,x)\in Y\times X$, we let $d(y,x)\coloneqq d(x,y)$.
\end{enumerate}

That $ d $ is indeed a $p$-metric on $X\sqcup Y$ is proved as follows. By the symmetric roles of $X$ and $Y$, we only need to check the following two cases:
\begin{enumerate}
\item $ d (x,y)\leq d(x,x')\ps{p} d (x',y),x,x'\in X,y\in Y$.

\item $ d (x,x')\leq d(x,y)\ps{p} d(x',y),x,x'\in X,y\in Y$.
\end{enumerate}

For the first case,
\begin{align*}
 d (x,x')\ps{p} d (x',y) & =  d (x,x')\ps{p}\inf_{(x_1,y_1)\in R}\left( d_X (x',x_1)\ps{p} d_Y (y_1,y)\ps{p}\eta\right)\\
& = \inf_{(x_1,y_1)\in R}\left( d (x,x')\ps{p} d_X (x',x_1)\ps{p} d_Y (y_1,y)\ps{p}\eta\right)\\
& \geq \inf_{(x_1,y_1)\in R}\left( d(x,x_1)\ps{p} d_Y (y_1,y)\ps{p}\eta\right)\\
& =  d (x,y).
\end{align*}

For the second case, 
\begin{align*}
&d (x,y)\ps{p} d (x',y)\\
 =& \inf_{(x_1,y_1)\in R}\left( d_X(x,x_1)\ps{p} d_Y  (y_1,y)\ps{p}\eta \right)\ps{p} \inf_{(x_2,y_2)\in R}\left( d_X  (x_2,x')\ps{p}  d_Y  (y_2,y)\ps{p} \eta \right)\\
 = & \inf_{(x_1,y_1),(x_2,y_2)\in R}\left( d_X  (x,x_1)\ps{p}  d_Y  (y_1,y)\ps{p} \eta \ps{p}  d_X  (x_2,x')\ps{p}  d_Y  (y_2,y)\ps{p} \eta \right)\\
 \geq & \inf_{(x_1,y_1),(x_2,y_2)\in R}\left( d_X  (x,x_1)\ps{p}  d_X  (x_2,x')\ps{p}  d_Y  (y_1,y_2)\ps{p} 2^\frac{1}{p}\eta \right)\\
 \geq & \inf_{(x_1,y_1),(x_2,y_2)\in R}\left( d_X  (x,x_1)\ps{p}  d_X  (x_2,x')\ps{p}  d_X  (x_1,x_2)\right)\geq d  (x,x').
\end{align*}
The second inequality follows from the fact that $\disp{p}(R)<  2^{\frac{1}{p}}\,\eta$ and \Cref{prop:simple-inv-triang}; the last inequality follows directly from the $p$-triangle inequality.

Note that $ d (x,y)=\eta$ for any $(x,y)\in R$. Therefore, by Claim \ref{claim:dghp inf R,d}, 
$\dghp{p}(X,Y)\leq \eta$. By a standard limit argument, one can then derive that $\dghp{p}(X,Y) \leq {2^{-\frac{1}{p}}}\inf_{R}\disp{p}(R).$
\end{proof}

\subsection{Relationship with $\dgh$}\label{sec:relation with dgh}
In this section, we study the relation between $\dghp p$ and $\dgh$ when restricted to $\msp$.

\paragraph{Isometry between $\dghp{p}$ and $\dgh$ via the snowflake transform.} One can directly relate $\dghp{p}$ and $\dgh$ in the following way via the snowflake transform:

\begin{proposition}\label{thm:dgh-dlp-eq}
Given $1\leq p<\infty$ and $X,Y\in\ms$, one has
$$\dgh(X,Y)=\lc\dghp{p}\lc S_\frac{1}{p}(X),S_\frac{1}{p}(Y)\rc\rc^p. $$
Conversely, 
for any $X,Y\in\msp$, one has 
$$\dghp{p}(X,Y) = \left(\dgh(S_p(X),S_p(Y))\right)^\frac{1}{p}.$$
\end{proposition}

\begin{proof}
Assume that $X,Y\in\ms$. For any $R\in\mathcal{R}(X,Y)$, we have
\begin{align*}
    \dis(R,d_X,d_Y)&=\sup_{(x,y),(x',y')\in R}|d_X(x,x')-d_Y(y,y')|\\
    &=\sup_{(x,y),(x',y')\in R}\left|\left((d_X)^\frac{1}{p}(x,x')\right)^p-\left((d_Y)^\frac{1}{p}(y,y')\right)^p\right|\\
    &=\left(\disp{p}\left(R,(d_X)^\frac{1}{p},(d_Y)^\frac{1}{p}\right)\right)^p
\end{align*}
Therefore, we have 
$$\dgh(X,Y)=\left(\dghp{p}\left(S_\frac{1}{p}(X),S_\frac{1}{p}(Y)\right)\right)^p.$$

Similarly, if $X,Y\in\msp$, then for any $R\in\mathcal{R}(X,Y)$, we have that
$$\disp{p}\lc R,d_X,d_Y\rc=\lc\dis\lc R,(d_X)^p,(d_Y)^p\rc\rc^\frac{1}{p}. $$
This implies that
$$ \dghp{p}(X,Y) = \lc\dgh(S_p(X),S_p(Y))\rc^\frac{1}{p}.$$
\end{proof}

This proposition has many interesting consequences. For example, the proposition immediately gives rise to the following analogue to \Cref{rem:ub} and \Cref{rem:lb} for $\dghp p$: 

\begin{proposition}[{\cite[Proposition 22]{memoli2021gromov}}]\label{prop:diam-dghp}
For any $p\in[1,\infty)$ and compact ultrametric spaces $X$ and $Y$, we have that
\[2^{-\frac{1}{p}}\Lambda_p(\diam(X),\diam(Y))\leq\dghp p(X,Y)\leq 2^{-\frac{1}{p}}\max(\diam(X),\diam(Y)). \]
\end{proposition}

\Cref{thm:dgh-dlp-eq} in particular establishes the stability result (cf. \Cref{thm:p-stable}) of $\mathfrak{S}_\infty$ when restricted to $\msp$.

\begin{proof}[Proof of \Cref{thm:p-stable}]
The inequality follows directly from Proposition \ref{prop:propofsp}, \Cref{thm:slstab} and \Cref{thm:dgh-dlp-eq}:
\begin{align*}
    \dghp{p}(\mathfrak{S}_\infty(X),\mathfrak{S}_\infty(Y))&=\lc\dgh\lc S_p(\mathfrak{S}_\infty(X)),S_p(\mathfrak{S}_\infty(Y))\rc\rc^\frac{1}{p}\\
    & = \lc\dgh\lc \mathfrak{S}_\infty(S_p(X)),\mathfrak{S}_\infty(S_p(Y))\rc\rc^\frac{1}{p}\\
    &\leq (\dgh(S_p(X),S_p(Y)))^\frac{1}{p}=\dghp{p}(X,Y).
\end{align*}
\end{proof}

As an application of \Cref{thm:dgh-dlp-eq}, we prove the continuity of $\dghp p$.
\begin{proposition}[Continuity of $\dghp{p}$]
Fix any $p_0\in[1,\infty]$. For any $X,Y\in\ms_{p_0}$, $\dghp{p}(X,Y)$ is continuous w.r.t. $p\in [1,p_0]$.
\end{proposition}
\begin{proof}
When $p_0<\infty$, the continuity follows directly from \Cref{thm:dgh-dlp-eq}. Now, we assume that $p_0=\infty$ and we only need to prove that for any $X,Y\in\ums$, $\ugh(X,Y)=\lim_{p\rightarrow\infty}\dghp{p}(X,Y)$.

The statement holds obviously when $X,Y$ are finite ultrametric spaces. When $X,Y\in\ums$, by \Cref{rmk:increading dghp}, $\lim_{p\rightarrow\infty}\dghp{p}(X,Y)$ exists and the limit is bounded above by $\ugh(X,Y)$. Define $\tilde{u}_\mathcal{GH}:\ums\times \ums\rightarrow\Rp$ by letting $\tilde{u}_\mathcal{GH}(X,Y)\coloneqq\lim_{p\rightarrow\infty}\dghp{p}(X,Y)$. Then, it is not hard to check that $\tilde{u}_\mathcal{GH}$ is an ultrametric on $\ums$ and that $\tilde{u}_\mathcal{GH}=\ugh$ on $\ufin\times\ufin$. By \Cref{ex:dlpnet}, for any $X,Y\in\ums$, there exist for every $n\in\N$, $X_n,Y_n\in\ufin$ such that $\ugh(X,X_n),\ugh(Y,Y_n)\leq\frac{1}{n}$. Then,
\[\ugh(X,Y)=\lim_{n\rightarrow\infty}\ugh(X_n,Y_n)=\lim_{n\rightarrow\infty}\tilde{u}_\mathcal{GH}(X_n,Y_n)=\tilde{u}_\mathcal{GH}(X,Y).\]
The last equality holds by continuity of metrics and by the fact that $\tilde{u}_\mathcal{GH}\leq\ugh$.
\end{proof}

The continuity of $\dghp{p}(X,Y)$ for finite ultrametric spaces $X$ and $Y$ are already mentioned in \cite[Proposition 17]{memoli2021gromov}. Hence, the proposition above is a generalization of \cite[Proposition 17]{memoli2021gromov} in two ways: we allow (1) possibly infinite spaces and (2) $p$-metric spaces $X$ and $Y$. 

Finally, we point out that \Cref{thm:dgh-dlp-eq} closely relates the two spaces $(\msp,\dghp{p})$ and $(\ms,\dgh)$ as follows.

\begin{corollary}\label{thm:isompms}
For any $p\in[1,\infty)$, we have $(\msp,\dghp{p})\cong \left(\ms,(\dgh)^\frac{1}{p}\right)$.
\end{corollary}
\begin{proof}
Consider the $p$-snowflake transform $S_p:\msp\rightarrow\ms$ sending $X$ to $S_p(X)$. By \Cref{thm:dgh-dlp-eq}, we have that for any $X,Y\in\msp$, $$\dghp{p}(X,Y)=(\dgh)^\frac{1}{p}(S_p(X),S_p(Y)).$$ 
Consider also the $\frac{1}{p}$-snowflake transform $S_\frac{1}{p}:\ms\rightarrow\msp$ sending $X$ to $S_\frac{1}{p}(X)$. Similarly, we will obtain 
$$(\dgh)^\frac{1}{p}(X,Y)=\dghp{p}\lc S_\frac{1}{p}(X),S_\frac{1}{p}(Y)\rc. $$
It is obvious that $S_\frac{1}{p}$ is the inverse of $S_p$ and thus $S_p$ is an isometry between $(\msp,\dghp{p})$ and $ \left(\ms,(\dgh)^\frac{1}{p}\right)$.
\end{proof}

\paragraph{H\"older equivalence between $\dghp{p}$ and $\dgh$.} We know from Remark \ref{rmk:increading dghp} that $\dghp{p}\geq\dgh$ for every $p\in[1,\infty]$. Naturally, one may wonder whether $\dgh$ could somehow upperbound $\dghp{p}$.  The answer is positive when  $p<\infty$. 

\begin{theorem}\label{thm:eqdgh}
There exist positive constants $C(p),D(p)$ and $E(p)$ depending only on $p\in[1,\infty)$ such that for any $X,Y\in\msp$, we have
$$\dghp{p}(X,Y)\leq C(p)\max\left(\diam(X),\diam(Y)\right)^{D(p)}\left(\dgh(X,Y)\right)^{E(p)}. $$

\end{theorem}

The proof follows from the following two simple lemmas regarding properties of the $p$-sum and the $p$-difference.

\begin{lemma}\label{lm:pminus-increasing}
For $a>b\geq 0$ and $1\leq p<\infty$, $f(p)\coloneqq \Lambda_p(a,b)$ is an increasing function with respect to $p$.
\end{lemma}
\begin{proof}
Let $g(p)=\ln f(p)=\frac{1}{p}\ln(a^p-b^p).$ Then, we have
\begin{align*}
    g'(p)& = \frac{1}{p^2}\left(\frac{a^p\ln a^p-b^p\ln b^p}{a^p-b^p} - \ln(a^p-b^p)\right)\\
    & =\frac{1}{p^2(a^p-b^p)}\left((a^p-b^p)(\ln a^p-\ln(a^p-b^p))+b^p\,(\ln b^p-\ln(a^p-b^p))\right)>0
\end{align*}
Therefore, $g$ is an increasing function and so is $f$.
\end{proof}

\begin{lemma}
For $M\geq a>b\geq 0$ and $1\leq p<\infty$, one has
$$\Lambda_p(a,b)\leq \lceil p\rceil^\frac{1}{\lceil p\rceil}\,M^{1-\frac{1}{\lceil p\rceil}}\,|a-b|^{\frac{1}{\lceil p\rceil}}, $$
where ${\lceil p\rceil}$ is the smallest integer greater than or equal to $p$.
\end{lemma}

\begin{proof}
First assume that $p$ is an integer. Then,
\begin{align*}
    a^p-b^p &=(a-b)(a^{p-1}+a^{p-2}b+\cdots+ab^{p-2}+b^{p-1})\\
    &\leq (a-b)\cdot p\,M^{p-1}.
\end{align*}
Hence $\Lambda_p(a,b)\leq  p^\frac{1}{ p}\,a^{1-\frac{1}{ p}}\,|a-b|^{\frac{1}{ p}}.$

Now, if $p\notin\N$, by the previous lemma $\Lambda_p(a,b)\leq \Lambda_{\lceil p\rceil}(a,b)$. The proof now follows.
\end{proof}

\begin{proof}[Proof of Theorem \ref{thm:eqdgh}]
Let $M\coloneqq \max\left(\diam(X),\diam(Y)\right)$.
For any $R\in\mathcal{R}(X,Y)$ and any $(x,y),(x',y')\in R$, we have that
\begin{align*}
    \Lambda_p(d_X (x,x'), d_Y (y,y'))\leq\lceil p\rceil^\frac{1}{\lceil p\rceil}\,M^{1-\frac{1}{\lceil p\rceil}}\,| d_X (x,x')- d_Y (y,y')|^{\frac{1}{\lceil p\rceil}} 
\end{align*}
Therefore, $\disp{p}(R)\leq \lceil p\rceil^\frac{1}{\lceil p\rceil}\,M^{1-\frac{1}{\lceil p\rceil}}\,\left(\dis(R)\right)^{\frac{1}{\lceil p\rceil}}$ and thus
$$\dghp{p}(X,Y)\leq \lceil p\rceil^\frac{1}{\lceil p\rceil}\,(2M)^{1-\frac{1}{\lceil p\rceil}}\,\left(\dgh(X,Y)\right)^{\frac{1}{\lceil p\rceil}}. $$
\end{proof}

Combining the inequality given by Theorem \ref{thm:eqdgh} with $\dghp{p}\geq\dgh$, one can conclude that when $p<\infty$, $\dghp{p}$ and $\dgh$ induce the same topology on $\msp$.

In contrast, the situation is quite different when $p=\infty$. The following example shows that $\ugh$ and $\dgh$ induce \emph{different} topologies on $\ums$.

\begin{example}\label{ex:2-pt-ms} 
Fix $\varepsilon> 0$. Consider the two-point metric space $\Delta_2(1)$ with interpoint distance 1 and the two-point metric space $\Delta_2(1+\eps)$ with interpoint distance $1+\eps$. These two spaces are obviously ultrametric spaces. Moreover, $\dgh(\Delta_2(1),\Delta_2(1+\varepsilon)) = \frac{\varepsilon}{2}$ and $\ugh(\Delta_2(1),\Delta_2(1+\varepsilon)) = 1+\varepsilon.$ Therefore, when $\eps$ approaches 0, $\Delta_2(1+\eps)$ will converge to $\Delta_2(1)$ in the sense of $\dgh$ but not in the sense of $\ugh$.
\end{example}

\nomenclature{$\Delta_n(a)$}{The $n$-point space with interpoint distance being $a$}

In conclusion, for $p\in[1,\infty)$, $\dghp{p}$ is topologically equivalent to $\dgh$ on $\msp$, whereas $\ugh$ induces a topology on $\ums$ which is coarser than the one induced by $\dgh$. 

\subsection{$\dghp{p}$ and approximate isometries}
Aside from \Cref{thm:dist} above,  as a counterpart to \cite[Corollary 7.3.28]{burago2001course}, there is another one-sided distortion characterization of $\dghp{p}$.

\begin{definition}\label{def:epsisom}
For $1\leq p\leq \infty$, let $X,Y\in\msp$ and let $\eps>0$. A map $f:X\rightarrow Y$ is called an \emph{$(\eps,p)$-isometry} if $\disp{p}(f)\leq \eps$ and $f(X)$ is an $\eps$-net of $Y$.
\end{definition}

We then have:

\begin{proposition}\label{coro:isometrychar}
For $1\leq p\leq \infty$, let $X,Y\in\msp$ and let $\eps>0$. Then,
\begin{enumerate}
    \item If $\dghp{p}(X,Y)<\eps$, then there exists a $\lc 2^\frac{1}{p}\eps,p\rc$-isometry from $X$ to $Y$.
    \item If there exists an $(\eps,p)$-isometry from $X$ to $Y$, then $\dghp{p}(X,Y)<2^\frac{1}{p}\eps$.
\end{enumerate}
\end{proposition}

\begin{proof}
\begin{enumerate}
    \item Let $R\in\mathcal{R}(X,Y)$ be such that $\disp{p}(R)<2^\frac{1}{p}\eps$. Define $f:X\rightarrow Y$ as follows. For every $x\in X$, choose any $y\in Y$ such that $(x,y)\in R$. Then, let $f(x)\coloneqq y$. Then, obviously, $\disp{p}(f)\leq\disp{p}(R)<2^\frac{1}{p}\eps$. Now, we show that $f(X)$ is a $2^\frac{1}{p}\eps$-net for $Y.$ Indeed, for any $y\in Y$, if we choose any $x\in X$ such that $(x,y)\in R$, then $$d_Y(y,f(x))=\Lambda_p(d_X(x,x), d_Y(y,f(x)))\leq\disp{p}(R)<2^\frac{1}{p}\eps.$$
    \item Let $f$ be an $(\eps,p)$-isometry. Define $R\subseteq X\times Y$ by
    $$R\coloneqq\{(x,y)\in X\times Y:\, d_Y(y,f(x))\leq \eps\}. $$
    $R\in \mathcal{R}(X,Y)$ since $f(X)$ is an $\eps$-net of $Y$. It is easy to see that $(x,f(x))\in R$ for any $x\in X$. If $(x,y),(x',y')\in R$, then we have
    \begin{align*}
        \Lambda_p(d_X(x,x'), d_Y(y,y'))&\leq \Lambda_p(d_X(x,x'), d_Y(f(x),f(x')))\ps{p}\Lambda_p(d_Y(f(x),f(x')), d_Y(y,y'))\\
        &\leq \disp{p}(f)\ps{p} d_Y(y,f(x))\ps{p} d_Y(y',f(x'))\leq 3^\frac{1}{p}\eps,
    \end{align*}
    where the first inequality follows from the fact that $(\Rp,\Lambda_p)$ is a $p$-metric space and the second inequality follows from \Cref{lm:trick p triangle}.
    Hence, $\disp{p}(R)\leq 3^\frac{1}{p}\eps$ and thus we have $\dghp{p}(X,Y)\leq \left(\frac{3}{2}\right)^\frac{1}{p}\eps<2^\frac{1}{p}\eps$.
    \end{enumerate}
\end{proof}

\begin{remark}
When $p=\infty$, this corollary essentially recovers Theorem 2.23 in Qiu's paper \cite{qiu2009geometry}. Qiu requires a slightly different condition on $f$ called \emph{strong $\eps$-isometry}. This notion is actually a variant of  $(\eps,\infty)$-isometry which arises when one replaces $\dis_\infty(f)\leq\eps$ with $\dis_\infty(f)<\eps$ in Definition \ref{def:epsisom}.
\end{remark}



\section{Special structural properties of $\ugh$}\label{sec:structural and computation of ugh}

Recall the definition of the closed quotient in \Cref{sec:quotient operation}. Then, it turns out that $\ugh$ can be completely characterized via the closed quotients of the given ultrametric spaces.
\begin{restatable}[Structural theorem for $\ugh$]{theorem}{structthm}
\label{thm:ugh-structure}
For all $X,Y\in\mathcal{U}$ one has that 
$$\ugh(X,Y) = \min\left\{t\geq 0:\, (X\ct{t},u_{X\ct{t}})\cong (Y\ct{t},u_{Y\ct{t}})\right\}.$$
\end{restatable}

\begin{figure}[ht]
    \centering
    \includegraphics[width=\textwidth]{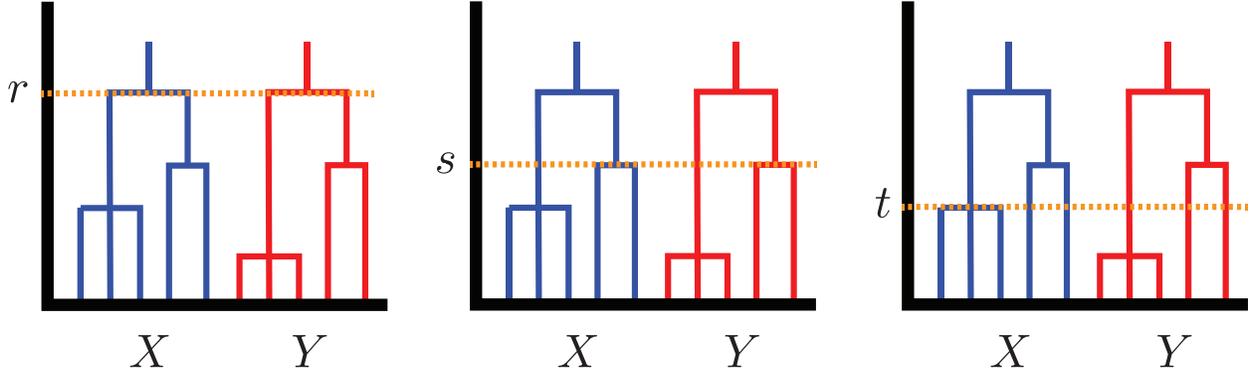}
    \caption{\textbf{Illustration of Theorem \ref{thm:ugh-structure}.} We represent two ultrametric spaces $X$ and $Y$ as dendrograms (See Theorem \ref{thm:dendroultra} for more details.). Imagine that the yellow line is scanning from right to left to obtain quotient spaces described in Definition \ref{def:ultraquotient}. It is easy to see from the figure that $X\ct{r}\cong Y\ct{r}$, $X\ct{s}\cong Y\ct{s}$, $X\ct{t}\cong Y\ct{t}$, and that $t$ is the minimum value such that the quotients are isometric. Thus, $\ugh(X,Y)=t.$ }

    \label{fig:strutural theorem}
\end{figure}

\begin{remark}[Computational implication]
The structure theorem for $\ugh$ allows us to reduce the problem of computing $\ugh(X,Y)$ between finite ultrametric spaces $X$ and $Y$ to essentially scanning all possible distance values $t$ of $X$ and $Y$ and checking whether $X\ct{t}\cong Y\ct{t}$. This process can be accomplished in time $O(n\log(n))$ where $n\coloneqq\max(\#X,\#Y)$; see \cite{memoli2021gromov} for more details.
\end{remark}

It follows easily from the structural theorem that one can directly determine the $\ugh$ distance between two ultrametric spaces with different diameters:
\begin{corollary}\label{coro:diam-ugh}
If $X$ and $Y$ are compact ultrametric spaces such that $\diam(X)\neq\diam(Y)$, then
$$\ugh(X,Y)=\max(\diam(X),\diam(Y)).$$
\end{corollary}

\begin{proof}
Assume without loss of generality that $\diam(X)<\diam(Y)$. Given $t=\diam(Y)$, we have $X\ct{t}\cong\ast\cong Y\ct{t}$. When $\diam(X)<t<\diam(Y)$, we have that $X\ct{t}\cong\ast$ but $Y\ct{t}\not\cong\ast$, thus $X\ct{t}\not\cong Y\ct{t}$. Therefore, by Theorem \ref{thm:ugh-structure}, $\ugh(X,Y)=\diam(Y)$.
\end{proof}

Moreover, it is direct to generalize \Cref{prop:diam-dghp} to the case of $p=\infty$.
\begin{proposition}\label{prop:diam-ugh}
For any compact ultrametric spaces $X$ and $Y$, we have that
\[\Lambda_\infty(\diam(X),\diam(Y))\leq\ugh(X,Y)\leq \max(\diam(X),\diam(Y)). \]
\end{proposition}
\begin{proof}
If $\diam(X)=\diam(Y)$, then by \Cref{thm:ugh-structure} we must have that $\ugh(X,Y)\leq\diam(X)$ since the $\diam(X)$-closed quotients of both $X$ and $Y$ become the one point space. Then,
\[\Lambda_\infty(\diam(X),\diam(Y))=0\leq\ugh(X,Y)\leq \max(\diam(X),\diam(Y)). \]
Otherwise, by \Cref{coro:diam-ugh}, we have that
\[\Lambda_\infty(\diam(X),\diam(Y))=\max(\diam(X),\diam(Y))=\ugh(X,Y)= \max(\diam(X),\diam(Y)). \]
\end{proof}

Now, we provide a proof of \Cref{thm:ugh-structure}. The proof is an adaptation of the proof of the finite version of the structural theorem \cite[Theorem 3]{memoli2021gromov} to non-necessarily finite setting.

\begin{proof}[Proof of Theorem \ref{thm:ugh-structure}]
We first prove a weaker version (with $\inf$ instead of $\min$): 
\begin{equation}\label{eq:ughinf}
    \ugh(X,Y) = \inf\left\{t\geq 0:\, (X\ct{t},u_{X\ct{t}})\cong (Y\ct{t},u_{Y\ct{t}})\right\}.
\end{equation}
This version then be used to establish the desired claim. Assume that $X_{\mathfrak{c}(t)}\cong Y_{\mathfrak{c}(t)}$ for some $t\geq 0$, i.e., there exists an isometry  $f_t:X_{\mathfrak{c}(t)}\rightarrow Y_{\mathfrak{c}(t)}$. Define 
\[R_t\coloneqq \left\{(x,y)\in X\times Y:\,[y]_{\mathfrak{c}(t)}^Y= f_t\lc [x]^X_{\mathfrak{c}(t)}\rc\right\}.\]
That $R_t\in\mathcal{R}(X,Y)$ is clear since $f_t$ is bijective. For any $(x,y),(x',y')\in R_t$, if $u_X(x,x')\leq t$, then we already have $u_X(x,x')\leq \max(t,u_Y(y,y'))$. Otherwise, if $u_X(x,x')>t$, then we have $[x]^X_{\mathfrak{c}(t)}\neq[x']^X_{\mathfrak{c}(t)}$. Since $f_t$ is bijective, we have that $[y]^Y_{\mathfrak{c}(t)}=f_t\!\lc[x]^X_{\mathfrak{c}(t)}\rc\neq f_t\!\lc[x']^X_{\mathfrak{c}(t)}\rc=[y']^Y_{\mathfrak{c}(t)}.$ Then,
\[u_Y(y,y')=u_{Y_{\mathfrak{c}(t)}}\lc [y]^Y_{\mathfrak{c}(t)},[y']^Y_{\mathfrak{c}(t)}\rc=u_{X_{\mathfrak{c}(t)}}\lc [x]^X_{\mathfrak{c}(t)},[x']^X_{\mathfrak{c}(t)}\rc=u_X(x,x').\]
Therefore, $u_X(x,x')\leq\max(t,u_Y(y,y'))$. Similarly, $u_Y(y,y')\leq\max(t,u_X(x,x'))$. Then,
$$\Lambda_\infty(u_X(x,x'),u_Y(y,y'))\leq t$$
and thus $\disp{\infty}(R_t)\leq t$. This implies that 
\[\ugh(X,Y) \leq \inf\left\{t\geq 0:\, X_{\mathfrak{c}(t)}\cong Y_{\mathfrak{c}(t)}\right\}.\]

Conversely, let $R\in\mathcal{R}(X,Y)$ and let $t\coloneqq\disp{\infty}(R)$. We define a map $f_t:X_{\mathfrak{c}(t)}\rightarrow Y_{\mathfrak{c}(t)}$ as follows: for each $[x]_{\mathfrak{c}(t)}^X\in X_{\mathfrak{c}(t)}$, choose any $y\in Y$ such that $(x,y)\in R$, then we let $f_t\lc[x]_{\mathfrak{c}(t)}^X\rc\coloneqq[y]_{\mathfrak{c}(t)}^Y$. $f_t$ is well-defined. Indeed, if $[x]_{\mathfrak{c}(t)}^X=[x']_{\mathfrak{c}(t)}^X$ and $y,y'\in Y$ are such that $(x,y),(x',y')\in R$, then 
\[\Lambda_\infty(u_X(x,x'),u_Y(y,y'))\leq\disp{\infty}(R)=t.\]
This implies that $u_Y(y',y)\leq t$ which is equivalent to $[y]^Y_{\mathfrak{c}(t)}=[y']^Y_{\mathfrak{c}(t)}$. Similarly, there is a well-defined map $g_t:Y_{\mathfrak{c}(t)}\rightarrow X_{\mathfrak{c}(t)}$ sending $[y]_{\mathfrak{c}(t)}^Y\in Y_{\mathfrak{c}(t)}$ to $[x]_{\mathfrak{c}(t)}^X$ whenever $(x,y)\in R$. It is clear that $g_t$ is the inverse of $f_t$ and thus $f_t$ is bijective.
Now, consider $x,x'\in X$ such that $u_X(x,x')>t$. Let $y,y'\in Y$ be such that $(x,y),(x',y')\in R$. Then, since $\Lambda_\infty(u_X(x,x'),u_Y(y,y'))\leq\disp{\infty}(R)=t$, we must have that $u_Y(y,y')=u_X(x,x')$. Therefore,
\[u_{Y_{\mathfrak{c}(t)}}\lc f_t\lc[x]^X_{\mathfrak{c}(t)}\rc,f_t\lc[x']^X_{\mathfrak{c}(t)}\rc\rc=u_{Y_{\mathfrak{c}(t)}}\lc [y]^Y_{\mathfrak{c}(t)},[y']^Y_{\mathfrak{c}(t)}\rc=u_{X_{\mathfrak{c}(t)}}\lc[x]^X_{\mathfrak{c}(t)},[x']^X_{\mathfrak{c}(t)}\rc.\] 
This proves that $f_t$ is an isometry and thus \[\ugh(X,Y) = \inf\left\{t\geq 0:\, X_{\mathfrak{c}(t)}\cong Y_{\mathfrak{c}(t)}\right\}.\]

Now, we finish the proof by showing that the infimum in \Cref{eq:ughinf} is attainable. 

Let $\delta \coloneqq\inf\lbrace t\geq 0 :\,X\ct{t} \cong Y\ct{t}\rbrace$. If $\delta>0$, let $\{t_n\}_{n\in\mathbb{N}}$ be a decreasing sequence converging to $\delta$ such that $X\ct{t_n}\cong Y\ct{t_n}$ for all $t_n$. Since $X\ct\delta$ and $Y\ct\delta$ are finite spaces (cf. \Cref{lm:compact-quotient}), we actually have that $X\ct{t_n}=X\ct{\delta}$ and $Y\ct{t_n}=Y\ct{\delta}$ when $n$ is large enough. This immediately implies that $X\ct\delta\cong Y\ct\delta$. Now, if $\delta=0$, then by \Cref{eq:ughinf} we have that $\ugh(X,Y)=\delta=0$. This implies that $X\cong Y$ and thus $X\ct\delta\cong Y\ct\delta$. Therefore, the infimum of $\inf\left\lbrace t\geq 0 :\,X\ct{t} \cong Y\ct{t}\right\rbrace$ is always attainable.
\end{proof}

\begin{remark}\label{rmk:stru-ult-arbitrary}
Note that in the proof of \Cref{eq:ughinf} above, we have not used the compactness of the spaces. Thus, $\ugh(X,Y) = \inf\left\{t\geq 0:\, (X\ct{t},u_{X\ct{t}})\cong (Y\ct{t},u_{Y\ct{t}})\right\}$ holds for arbitrary ultrametric spaces $X$ and $Y$. 
\end{remark}

\begin{example}[$\ugh$ between $p$-adic integer rings]
One direct consequence from \Cref{rmk:stru-ult-arbitrary} is that for any different prime numbers $p$ and $q$, $\ugh$ distance between $\mathbb{Z}_p$ and $\mathbb{Z}_q$ is 1, where $\mathbb{Z}_p$ denotes the $p$-adic integer ring.

We first recall the definition of $\mathbb{Z}_p$.
For any prime number $p$, the $p$-adic integer ring $\mathbb{Z}_p$ consists of all formal power series $x=a_0+a_1p+a_2p^2+\cdots$ such that $0\leq a_i<p$ for all $i\in\mathbb{N}$. The $p$-adic valuation $|x|_p$ is defined as follows:
\[|x|_p\coloneqq\begin{cases}
0 & \text{if }x=0\\
p^{-\min\{n:\,a_n\neq 0\}}& \text{otherwise}
\end{cases}.\]
This valuation defines a metric $d$ on $\mathbb{Z}_p$: $d(x,y)\coloneqq|x-y|_p$. $d$ turns out to be an ultrametric \cite{gouvea1997p}.

It is easy to see that $\diam(\mathbb{Z}_p)=1$. Then, for any two different prime numbers $p$ and $q$, $\ugh(\mathbb{Z}_p,\mathbb{Z}_q)\leq1$. For the converse inequality, indeed, for any $\max(p^{-1},q^{-1})<t<1$, we have $(\mathbb{Z}_p)_t=\{i+p\mathbb{Z}_p:i=0,\cdots,p-1\}$ and $(\mathbb{Z}_q)_t=\{i+q\mathbb{Z}_q:i=0,\cdots,q-1\}$. So $\#(\mathbb{Z}_p)_t=p$ whereas $\#(\mathbb{Z}_q)_t=q$ and thus $(\mathbb{Z}_p)_t\not\cong(\mathbb{Z}_q)_t$. Therefore, by Remark \ref{rmk:stru-ult-arbitrary} we have that $\ugh(\mathbb{Z}_p,\mathbb{Z}_q)=1$.
\end{example}

\subsection{Curvature sets permit calculating $\ugh$.}

In this section, we show that how the curvature sets defined in \cite{gromov2007metric} can be used to  calculate $\ugh$.

\begin{definition}[Curvature sets \cite{gromov2007metric}]\label{def:curvsets}
For a metric space $X$, and a positive integer $n$, let $\Psi_X^{(n)}:X^{ n}\rightarrow \Rp^{n\times n}$ be the function given by $(x_1,\ldots,x_n)\mapsto \left(d_X(x_i,x_j)\right)_{i,j=1}^n$. Then, the curvature set of order $n$ associated to $X$ is defined as 
$$\mathrm{K}_n(X)\coloneqq \mathrm{im}\left(\Psi_X^{(n)}\right).$$
When $n=2$, any element of $\mathrm{K}_n(X)$ is of the form of $\begin{pmatrix} 
0 & d \\
d & 0 
\end{pmatrix}$ for some $x,x'\in X$ and $d=d_X(x,x')$. Thus, $\mathrm{K}_2(X)$ is equivalent to the spectrum $\spec(X)=\{d_X(x,x'):\,x,x'\in X\}$.
\end{definition}

\nomenclature[S]{$\mathrm{K}_n(X)$}{The curvature set of order $n$ associated to $X$}

\begin{theorem}[Gromov's metric space reconstruction theorem]\label{thm:gromovcurv}
Given two compact metric spaces $X$ and $Y$, if $\mathrm{K}_n(X) = \mathrm{K}_n(Y)$ for every $n \in \mathbb{N}$, then $X \cong Y$.
\end{theorem}

Note that $\mathrm{K}_n:\ms\rightarrow\Rp^{n\times n}$ is a metric invariant, i.e., if $X\cong Y$, then $\mathrm{K}_n(X)=\mathrm{K}_n(Y)$. Therefore, as a direct consequence of \Cref{thm:ugh-structure} and \Cref{thm:gromovcurv}, we have that 
\begin{corollary}\label{coro:curvineq}
Given two compact ultrametric spaces $X$ and $Y$, we have
\begin{equation}\label{eq:curvugh}
    \ugh(X,Y) \geq \inf\lb t\geq 0:\,\mathrm{K}_n\lc X\ct{t}\rc=\mathrm{K}_n\lc Y\ct{t}\rc\rb.
\end{equation}
\end{corollary}
This corollary vastly generalizes \cite[item (1) of Theorem 4.2]{qiu2009geometry} which proves the following inequality similar to the case $n=2$ of (\ref{eq:curvugh}): 

\begin{corollary}\label{coro:app-ugh}
Given $X,Y\in\ums$, we have
$$\ugh(X,Y)\geq { \inf}\{\eps\geq 0:\,\spec_\eps(X)=\spec_\eps(Y)\}, $$
where for $\eps\geq 0$, $\spec_\eps(X)\coloneqq\{t\in\spec(X):\,t\geq \eps\}$.
\end{corollary}

Actually, the infimum in the above two corollaries can be replaced by minimum due to the following result:

\begin{lemma}
Given $X,Y\in\ums$, the infimum in $\inf\lb t\geq 0:\,\mathrm{K}_n\lc X\ct{t}\rc=\mathrm{K}_n\lc Y\ct{t}\rc\rb$ is attainable.
\end{lemma}
\begin{proof}
Let $t_0\coloneqq \inf\lb t\geq 0:\,\mathrm{K}_n\lc X\ct{t}\rc=\mathrm{K}_n\lc Y\ct{t}\rc\rb$. If $t_0>0$, then by compactness of $X$ and $Y$ as well as \Cref{lm:compact-quotient}, there exists $\eps>0$ such that for any $t\in[t_0,t_0+\eps]$, $X\ct{t}\cong X\ct{t_0}$ and $Y\ct{t}\cong Y\ct{t_0}$. Then, it is obvious that the infimum is attainable. Now, assume that $t_0=0$. For any tuple $(x_1,\ldots,x_n)$ in $X$, there exists $t_1>0$ such that for any $t\in[0,t_1]$, $u_{X\ct{t}}\lc[x_i]\ct{t},[x_j]\ct{t}\rc=u_X(x_i,x_j)$ for all $i,j$. Then, we choose any $t_2\in [0,t_1]$ such that $\mathrm{K}_n\lc X\ct{t_2}\rc=\mathrm{K}_n\lc Y\ct{t_2}\rc$. This implies that
\[\lc u_X(x_i,x_j)\rc_{i,j=1}^n=\lc u_{X\ct{t_2}}\lc[x_i]\ct{t_2},[x_j]\ct{t_2}\rc\rc_{i,j=1}^n\in\mathrm{K}_n\lc Y\ct{t_2}\rc\subseteq\mathrm{K}_n(Y).\]
Therefore, $\mathrm{K}_n(X)\subseteq \mathrm{K}_n(Y)$. Similarly, $\mathrm{K}_n(Y)\subseteq \mathrm{K}_n(X)$. Therefore,
\[\mathrm{K}_n\lc X\ct{t_0}\rc=\mathrm{K}_n(X)=\mathrm{K}_n(Y)=\mathrm{K}_n\lc Y\ct{t_0}\rc.\]
\end{proof}

The following example shows that the equality in \Cref{coro:curvineq} does not hold in general.

\begin{example}
Let $X=\Delta_2(1)$ and $Y=\Delta_3(1)$ be the 2-point space and 3-point space with distance 1 respectively. A simple calculation shows $\mathrm{K}_2(X)=\{0,1\}=\mathrm{K}_2(Y)$. Since $X=X_0$ and $Y=Y_0$, we have $\min\lb t\geq 0:\,\mathrm{K}_2\lc X\ct{t}\rc=\mathrm{K}_2\lc Y\ct{t}\rc\rb=0<1=\ugh(X,Y).$ Similarly, if we take $X=\Delta_n(1)$ and $Y=\Delta_{n+1}(1)$ for arbitrary $n\in\mathbb{N}$, then $\mathrm{K}_n(X)=\mathrm{K}_n(Y)$. Thus, $\min\lb t\geq 0:\,\mathrm{K}_n\lc X\ct{t}\rc=\mathrm{K}_n\lc Y\ct{t}\rc\rb=0<1=\ugh(X,Y).$
\end{example}

Taking one step further from the case $X=\Delta_2(1)$ and $Y=\Delta_3(1)$, we will see that $\mathrm{K}_3(X_0)\neq \mathrm{K}_3(Y_0)$ and in fact $\inf\lb t\geq 0:\,\mathrm{K}_3\lc X\ct{t}\rc=\mathrm{K}_3\lc Y\ct{t}\rc\rb=1=\ugh(X,Y).$ In fact, this phenomenon is not a coincidence. If we consider $\mathrm{K}_n$ for all $n\in\mathbb{N}$, then we will recover $\ugh$:

\begin{theorem}\label{thm:ughcurvset}
Given two compact ultrametric spaces $X$ and $Y$, we have
$$\ugh(X,Y) = \sup_{n\in\mathbb{N}}\min\lb t\geq 0:\,\mathrm{K}_n\lc X\ct{t}\rc=\mathrm{K}_n\lc Y\ct{t}\rc\rb.$$
\end{theorem}

\begin{proof}
Due to Corollary \ref{coro:curvineq}, we only need to show that 
$$\ugh(X,Y) \leq \sup_{n\in\mathbb{N}}\min\lb t\geq 0:\,\mathrm{K}_n\lc X\ct{t}\rc=\mathrm{K}_n\lc Y\ct{t}\rc\rb.$$
We begin with a simple observation following directly from the definition of curvature sets and the quotient construction described in Definition \ref{def:ultraquotient}. 
\begin{claim}\label{claim:easy}
If $\mathrm{K}_n\lc X\ct{t}\rc=\mathrm{K}_n\lc Y\ct{t}\rc$, then $\mathrm{K}_n(X\ct{s})=\mathrm{K}_n(Y\ct{s})$ for $s>t$.
\end{claim}

Suppose to the contrary that $\sup_{n\in\mathbb{N}}\min\{t\geq 0:\,\mathrm{K}_n\lc X\ct{t}\rc=\mathrm{K}_n\lc Y\ct{t}\rc\}=t_0<\ugh(X,Y)$. Then, by the claim above, there exists $\eps>0$ such that $t_1\coloneqq t_0+\eps<\ugh(X,Y)$ and $\mathrm{K}_n(X\ct{t_1})=\mathrm{K}_n(Y\ct{t_1})$ for all $n\in\mathbb{N}$. According to Gromov's reconstruction theorem (Theorem \ref{thm:gromovcurv}), one has that $X\ct{t_1}\cong Y\ct{t_1}$. This implies that $\ugh(X,Y)\leq t_1$ by Theorem \ref{thm:ugh-structure}, which contradicts the fact that $\ugh(X,Y)>t_1$.

\begin{proof}[Proof of Claim \ref{claim:easy}]
Given $\left(u_{X\ct{s}}\lc[x_i]^X\ct{s},[x_j]^X\ct{s}\rc\right)_{i,j=1}^n\in \mathrm{K}_n(X\ct{s}),$ consider
$$\left(u_{X\ct{t}}\lc [x_i]^X\ct{t},[x_j]^X\ct{t}\rc\right)_{i,j=1}^n\in \mathrm{K}_n\lc X\ct{t}\rc=\mathrm{K}_n\lc Y\ct{t}\rc.$$
Then, there exists a tuple $([y_1]^Y\ct{t},\cdots,[y_n]\ct{t}^Y)$ such that for any $1\leq i,j\leq n$ we have
$$u_{X\ct{t}}\lc [x_i]^X\ct{t},[x_j]^X\ct{t}\rc=u_{Y\ct{t}}\lc [y_i]^Y\ct{t},[y_j]^Y\ct{t}\rc. $$
We have the following two cases:
\begin{enumerate}
    \item $[x_i]^X\ct{t}\neq [x_j]^X\ct{t}$. Then, by construction of $u_{X\ct{t}}$ we have 
    $$u_X(x_i,x_j)=u_{X\ct{t}}\lc [x_i]^X\ct{t},[x_j]^X\ct{t}\rc=u_{Y\ct{t}}\lc [y_i]^Y\ct{t},[y_j]^Y\ct{t}\rc=u_Y(y_i,y_j). $$
    Hence for $s>t$, we have $u_{X\ct{s}}\lc[x_i]^X\ct{s},[x_j]^X\ct{s}\rc=u_{Y\ct{s}}\lc[y_i]^Y\ct{s},[y_j]^Y\ct{s}\rc. $
    \item $[x_i]^X\ct{t}= [x_j]^X\ct{t}$. Then, $u_{X\ct{t}}\lc [x_i]^X\ct{t},[x_j]^X\ct{t}\rc=u_{Y\ct{t}}\lc [y_i]^Y\ct{t},[y_j]^Y\ct{t}\rc$ also implies that $[y_i]^Y\ct{t}=[y_j]^Y\ct{t}$. Then, obviously for $s>t$ we have $[x_i]^X\ct{s}=[x_j]^X\ct{s}$ and $[y_i]^Y\ct{s}=[y_j]^Y\ct{s}$ and thus $u_{X\ct{s}}\lc[x_i]^X\ct{s},[x_j]^X\ct{s}\rc=0=u_{Y\ct{s}}\lc[y_i]^Y\ct{s},[y_j]^Y\ct{s}\rc. $
\end{enumerate}
The previous discussion then shows that $u_{X\ct{s}}\lc[x_i]^X\ct{s},[x_j]^X\ct{s}\rc=u_{Y\ct{s}}\lc[y_i]^Y\ct{s},[y_j]^Y\ct{s}\rc $ for any $1\leq i,j\leq n$. Therefore,
$$\left(u_{X\ct{s}}\lc[x_i]^X\ct{s},[x_j]^X\ct{s}\rc\right)_{i,j=1}^n=\left(u_{Y\ct{s}}\lc[y_i]^Y\ct{s},[y_j]^Y\ct{s}\rc\right)_{i,j=1}^n\in \mathrm{K}_n(Y\ct{s}),$$
so $\mathrm{K}_n(X\ct{s})\subseteq \mathrm{K}_n(Y\ct{s})$. Similarly, $\mathrm{K}_n(Y\ct{s})\subseteq \mathrm{K}_n(X\ct{s})$ and thus $\mathrm{K}_n(X\ct{s})=\mathrm{K}_n(Y\ct{s}).$
\end{proof}
\end{proof}

\subsection{A structural theorem for the Hausdorff distance on ultrametric spaces} 

There exists a structural theorem for the Hausdorff distance on ultrametric spaces in a spirit similar to the one in the structural theorem for $\ugh$ (cf. \Cref{thm:ugh-structure}).

\begin{theorem}\label{thm:hdf-strt}
Let $X$ be a (non-necessarily compact) ultrametric space. For any closed subsets $A,B\subseteq X$, we have that
$$\dH^X(A,B)=\inf\lb t\geq 0: A\ct{t}=B\ct{t}\rb, $$
where $A\ct{t}\coloneqq\lb[x]\ct{t}^X:\,x\in A\rb\subseteq X\ct{t}.$ When $X$ is compact, the infimum above can be replaced by minimum.
\end{theorem}

\begin{proof}
We need the following obvious observation regarding ultrametric spaces.

\begin{claim}\label{claim:at=bt}
$A\subseteq B^t,B\subseteq A^t$ if and only if $ A^t=B^t$. Here $A^t\coloneqq\{x\in X:\, u_X(x,A)\leq t\}$.
\end{claim}
\begin{proof}
The if part is obvious. As for the only if part, first note that $(A^t)^t=A^t$ since $X$ is an ultrametric. Then, $A\subseteq B^t$ implies $A^t\subseteq (B^t)^t=B^t$. Similarly, $B^t\subseteq A^t$ and thus $A^t=B^t$.
\end{proof}
Then, we have that
$$\dH^X(A,B)=\inf\lb t\geq 0: A^t\subseteq B^t, \, B^t\subseteq A^t\rb=\inf\lb t\geq 0: A^t=B^t\rb. $$
It remains to prove that $A^t=B^t$ if and only if $A\ct{t}=B\ct{t}$.
\begin{enumerate}
    \item Suppose $A^t=B^t$. For any $x\in A^t$, by closeness of $A$, there exists $x_0\in A$ such that $u_X(x,x_0)\leq t$. Hence, $[x]\ct{t}^X=[x_0]\ct{t}^X$. Thus, $(A^t)\ct{t}=A\ct{t}$. Similarly, $(B^t)\ct{t}=B\ct{t}$. Therefore, 
    $$A\ct{t}=(A^t)\ct{t}=(B^t)\ct{t}=B\ct{t}. $$
    \item Suppose that $A\ct{t}=B\ct{t}$. Since $A$ is closed, one has $A^t=\bigcup_{x\in A}[x]\ct{t}^X=\bigcup_{[x]\ct{t}^X\in A\ct{t}}[x]\ct{t}^X$. Therefore, $$A^t=\bigcup_{[x]\ct{t}^X\in A\ct{t}}[x]\ct{t}^X=\bigcup_{[x]\ct{t}^X\in B\ct{t}}[x]\ct{t}^X=B^t.$$ 
\end{enumerate}

Now, assume that $X$ is compact. Let $\delta \coloneqq\inf\lbrace t\geq 0 :\,A^t=B^t\rbrace$. Let $\{t_n\}_{n\in\mathbb{N}}$ be a decreasing sequence converging to $\delta$ such that $A^{t_n}= B^{t_n}$ for all $t_n$. Then, $A^\delta\subseteq A^{t_n}= B^{t_n}$ for all $t_n$. For any $x\in A^\delta$ and for each $n\in \N$, there exists $b_n\in B$ such that $u_X(x,b_n)\leq t_n$. Since $B$ is compact, the sequence $\{b_n\}_{n\in\N}$ contains limit points. Let $b\in B$ be such a limit point. Then, it is easy to see that $u_X(x,b)\leq\lim_{n\rightarrow\infty}t_n=\delta$. This implies that $x\in B^\delta$ and thus $A^\delta\subseteq B^\delta$. Similarly, $B^\delta\subseteq A^\delta$ and thus $A^\delta=B^\delta$. Therefore,
\[\dH^X(A,B)=\min\lb t\geq 0: A^t=B^t\rb=\min\lb t\geq 0: A\ct{t}=B\ct{t}\rb.\]
\end{proof}

Recall from \Cref{ex:pm is pm} that $\Lambda_\infty^n$ defines an ultrametric on $\R^n_{\geq 0}$. 
For any $n\in\N$ and any $X\in \ums$, the curvature set $\mathrm{K}_n(X)=\Psi_X^{(n)}(X^n)$ is obviously a closed subset of $\lc\R^{n^2}_{\geq 0},\Lambda_\infty^n\rc$. So we can compare curvature sets $\mathrm{K}_n(X)$ and $\mathrm{K}_n(Y)$ of two compact ultrametric spaces $X$ and $Y$ via the Hausdorff distance on $(\R^{n^2}_{\geq 0},\Lambda_\infty^{n^2})$. 

\begin{corollary}
For any $X,Y\in\ums$, we have that
$$\ugh(X,Y)=\sup_{n\in\mathbb{N}} \dH^{\left(\R^{n^2}_{\geq 0},\Lambda_\infty^{n^2}\right)}(\mathrm{K}_n(X),\mathrm{K}_n(Y)). $$
\end{corollary}

\begin{proof}
By Theorem \ref{thm:ughcurvset} and the proof of Theorem \ref{thm:hdf-strt}, we only need to prove that for any $t\geq 0$, $\mathrm{K}_n\lc X\ct{t}\rc=\mathrm{K}_n\lc Y\ct{t}\rc$ if and only if $(\mathrm{K}_n(X))^t=(\mathrm{K}_n(Y))^t$.

Assume that $\mathrm{K}_n\lc X\ct{t}\rc=\mathrm{K}_n\lc Y\ct{t}\rc$. For any $(a_{ij})_{i,j=1}^n\in(\mathrm{K}_n(X))^t$, there exist $x_1,\cdots,x_n\in X$ such that $\Lambda_\infty(a_{ij},u_X(x_i,x_j))\leq t$ for all $1\leq i,j\leq n$ by definition of $\Lambda_\infty^{n^2}$ (Remark \ref{rmk:product-pm}). Then, we have
\begin{align*}
 &\Lambda_\infty\lc a_{ij},u_{X\ct{t}}\lc [x_i]^X\ct{t},[x_j]^X\ct{t}\rc\rc\\
 \leq&\max\left(\Lambda_\infty(a_{ij},u_X(x_i,x_j)), \Lambda_\infty\lc u_{X\ct{t}}\lc [x_i]^X\ct{t},[x_j]^X\ct{t}\rc,u_X(x_i,x_j)\rc\right) \\
 \leq& \max\left(t, \Lambda_\infty\lc u_{X\ct{t}}\lc [x_i]^X\ct{t},[x_j]^X\ct{t}\rc,u_X(x_i,x_j)\rc\right)\leq t.
\end{align*}
The first inequality follows from the fact that $\Lambda_\infty$ is an ultrametric on $\R$. The last inequality follows from the definition of $(X\ct{t},u_{X\ct{t}})$ (Definition \ref{def:ultraquotient}) and the definition of $\Lambda_\infty$. Since $\mathrm{K}_n\lc X\ct{t}\rc=\mathrm{K}_n\lc Y\ct{t}\rc$, there exist $y_1,\cdots,y_n\in Y$ such that $u_{X\ct{t}}\lc[x_i]\ct{t}^X,[x_j]\ct{t}^X\rc=u_{Y\ct{t}}\lc[y_i]\ct{t}^Y,[y_j]\ct{t}^Y\rc.$ Hence, $\Lambda_\infty\lc a_{ij},u_{Y\ct{t}}\lc [y_i]^Y\ct{t},[y_j]^Y\ct{t}\rc\rc\leq t$ for all $i,j=1,\cdots,n$. Then, by an argument similar to the one above, we conclude that $\Lambda_\infty(a_{ij},u_Y(y_i,y_j))\leq t$ and thus $(a_{ij})_{i,j=1}^n\in (\mathrm{K}_n(Y))^t$. Therefore,  $(\mathrm{K}_n(X))^t\subseteq(\mathrm{K}_n(Y))^t$. Similarly  $(\mathrm{K}_n(Y))^t\subseteq(\mathrm{K}_n(X))^t$, so  $(\mathrm{K}_n(X))^t=(\mathrm{K}_n(Y))^t$.

Conversely, assume that  $(\mathrm{K}_n(X))^t=(\mathrm{K}_n(Y))^t$. Hence, for any sequence $x_1,\cdots,x_n\in X$, there exists a sequence $y_1,\cdots,y_n\in Y$ such that $\Lambda_\infty(u_X(x_i,x_j),u_Y(y_i,y_j))\leq t$ for all $1\leq i,j\leq n$. For any given $i,j$, we have the following two cases.
\begin{enumerate}
    \item If both $u_X(x_i,x_j),u_Y(y_i,y_j)\leq t$, then $u_{X\ct{t}}\lc[x_i]\ct{t}^X,[x_j]\ct{t}^X\rc=0=u_{Y\ct{t}}\lc[y_i]\ct{t}^Y,[y_j]\ct{t}^Y\rc$.
    \item If one of $u_X(x_i,x_j),u_Y(y_i,y_j)$ is greater than $t$, then by definition of $\Lambda_\infty$, we must have $u_X(x_i,x_j)=u_Y(y_i,y_j)$. Hence, $u_{X\ct{t}}\lc[x_i]\ct{t}^X,[x_j]\ct{t}^X\rc=u_{Y\ct{t}}\lc[y_i]\ct{t}^Y,[y_j]\ct{t}^Y\rc$.
\end{enumerate}
Therefore, $\lc u_{X\ct{t}}\lc[x_i]\ct{t}^X,[x_j]\ct{t}^X\rc\rc_{i,j=1}^n\in\mathrm{K}_n\lc Y\ct{t}\rc$ and thus $\mathrm{K}_n\lc X\ct{t}\rc\subseteq\mathrm{K}_n\lc Y\ct{t}\rc$. Similarly, $\mathrm{K}_n\lc Y\ct{t}\rc\subseteq\mathrm{K}_n\lc X\ct{t}\rc$ and thus $\mathrm{K}_n\lc X\ct{t}\rc=\mathrm{K}_n\lc Y\ct{t}\rc$.
\end{proof}


\subsection{A modified version of $\ugh$.} Theorem \ref{thm:ughcurvset} actually suggests a connection with a modified version of Gromov-Hausdorff distance introduced in \cite{memoli2012some}, which also possesses a characterization via curvature sets. We now describe this connection.

It is known from \Cref{eq:distmap} that $\dgh$ has the following distortion formula:
\[{\dgh}(X,Y) = \frac{1}{2}\inf_{\substack{\varphi:X\rightarrow Y \\ \psi:Y\rightarrow X}}\max\left(\mathrm{dis}(\varphi),\mathrm{dis}(\psi),\mathrm{codis}(\varphi,\psi)\right).\]
By omitting the codistortion part, the computation can be reduced to solving two decoupled problems which will allow acceleration in practical applications. Hence in \cite{memoli2012some}, the author proposed the following distance as a lower bound of $\dgh$:
$$\hdgh(X,Y) = \frac{1}{2}\inf_{\substack{\varphi:X\rightarrow Y \\ \psi:Y\rightarrow X}}\max\left(\mathrm{dis}(\varphi),\mathrm{dis}(\psi)\right)=\frac{1}{2}\max\left(\inf_{\varphi:X\rightarrow Y}\dis(\varphi),\inf_{\psi:Y\rightarrow X}\dis(\psi)\right).$$
It is shown that $\hdgh$ is a legitimate distance on the collection of isometry classes of $\ms$ and $\hdgh\leq\dgh$ whereas an inverse inequality does not exist in general. In fact, it was shown in \cite{memoli2012some} that there exist finite metric spaces for which the inequality is strict. 

This new distance is related to curvature sets via a structural theorem (Theorem 5.1 in \cite{memoli2012some}) which shows that $\hdgh$ is completely characterized by curvature sets of $X$ and $Y$.

Inspired by the construction of $\hdgh$, it is natural to consider the following modified version of $\ugh$:
$$\widehat{u}_{GH}(X,Y) = \inf_{\substack{\varphi:X\rightarrow Y \\ \psi:Y\rightarrow X}}\max\left(\mathrm{dis}_\infty(\varphi),\mathrm{dis}_\infty(\psi)\right)=\max\left(\inf_{\varphi:X\rightarrow Y}\disp{\infty}(\varphi),\inf_{\psi:Y\rightarrow X}\disp{\infty}(\psi)\right).$$

It is then an interesting fact that in contrast to  $\dgh\gneq \hat{d}_{\mathrm{GH}} $ in general, the modified distance $\widehat{u}_{\mathrm{GH}}$ always coincides with $\ugh$.
\begin{theorem}\label{thm:ughm}
For all $X$ and $Y$ in $\mathcal{U}$, we have that $$\widehat{u}_{GH}(X,Y)=\ugh(X,Y).$$
\end{theorem}

\begin{proof}
By \Cref{thm:dist}, we have that 
$$ \ugh(X,Y)=\min_{\varphi,\psi}\max\left(\disp{\infty}(\varphi),\disp{\infty}(\psi),\codis_\infty(\varphi,\psi)\right).$$
Hence, $\ugh(X,Y)\geq\widehat{u}_{GH}(X,Y)$. 

Conversely, if there exist $\varphi,\psi$ such that $\max\left(\disp{\infty}(\varphi),\disp{\infty}(\psi)\right)\leq\eta$, we need to show that $\ugh(X,Y)\leq\eta$. Since $\disp{\infty}(\varphi)\leq\eta$, we have that for any $x,x'\in X$,
$$\Lambda_\infty(u_X(x,x'), u_Y(\varphi(x),\varphi(x')))\leq\eta.$$ 
Thus, we have the following two possibilities:
\begin{enumerate}
\item  $u_X(x,x')\neq u_Y(\varphi(x),\varphi(x')),$ and in this case neither of them is larger than $\eta$; or
\item $u_X(x,x')=u_Y(\varphi(x),\varphi(x')).$ 
\end{enumerate}
In either case, whenever $u_X(x,x')\leq\eta$, we have that $u_Y(\varphi(x),\varphi(x'))\leq\eta$. This is equivalent to saying that $\varphi$ is 1-Lipschitz and thus $\varphi$ canonically induces a map $\varphi_\eta:X\ct{\eta}\rightarrow Y\ct{\eta}$ by Lemma \ref{lm:induce}. 

For any $x,x'$ such that $u_X(x,x')>\eta$, we have that $u_Y(\varphi(x),\varphi(x'))=u_X(x,x')$. Then, we have that 
\begin{align*}
u_{Y\ct{\eta}}([\varphi(x)]\ct{\eta},[\varphi(x')]\ct{\eta})=u_Y(\varphi(x),\varphi(x'))=u_X(x,x')=u_{X\ct{\eta}}([x]\ct{\eta},[x']\ct{\eta}).
\end{align*}
Therefore, $\varphi_\eta$ is an isometric embedding. Similarly, we can prove that $\psi_\eta$ is an isometric embedding. This implies, by a standard argument in \cite[Theorem 1.6.14]{burago2001course}, that both $\varphi_\eta$ and $\psi_\eta$ are isometries, which shows $X\ct{\eta}\cong Y\ct{\eta}$. Then, by Theorem \ref{thm:ugh-structure}, we have that $\ugh(X,Y)\leq\eta$.
\end{proof}

\section{Relationship between $\dghp p$ and interleaving type  distances}  \label{sec:interleavings}

The interleaving distance has been widely used in the community of topological data analysis for comparing persistence modules \cite{chazal2009proximity,morozov2013interleaving,bubenik2014categorification,bubenik2015metrics}. Since a dendrogram is a special persistence module, it is also natural to compare dendrograms via the interleaving distance. In this section, we introduce an interleaving distance between ultrametric spaces through the interleaving distance between dendrograms. We further study the relationship between the interleaving distance and the Gromov-Hausdorff distance between ultrametric spaces.

We first introduce the notion of interleaving between two given dendrograms (cf. \Cref{def:proper dendrogram}). This definition is adapted from \cite{morozov2013interleaving,bubenik2014categorification}. To introduce this notion, we first introduce the category $\mathbf{Part}$ of all partitions as follows. The objects are pairs $(X,P_X)$ where $X$ is a set and $P_X\in \mathbf{Part}(X)$ is a partition. A morphism $\varphi:(X,P_X)\rightarrow(Y,P_Y)$ is any set map $\varphi:X\rightarrow Y$ such that for any block $B$ in $P_X$, $\varphi(B)\subseteq C$ for some block $C\in P_Y$.

\begin{definition}[Interleaving I]\label{def:catint}
Given two dendrograms $(X,\theta_X)$ and $(Y,\theta_Y)$, we say they are \emph{$\varepsilon$-interleaved} for a fixed $\eps\geq 0$, if for each $t\geq 0$, there exist morphisms $\varphi_t :(X,\theta_X(t))\rightarrow (Y,\theta_Y(t+\eps))$ and $\psi_t :(Y,\theta_Y(x))\rightarrow (X,\theta_X(t+\eps))$ such that for any  $x\in X$ and $y\in Y$, and any $0\leq s\leq t<\infty$,
\begin{enumerate}
    \item $[\varphi_s(x)]^{\theta_Y}_{t+\eps}=[\varphi_{t}(x)]^{\theta_Y}_{t+\eps}$ and $[\psi_t(y)]^{\theta_X}_{s+\eps}=[\psi_{s}(y)]^{\theta_X}_{s+\eps}$;
    \item $\psi_{t+\eps}\circ\varphi_t\!\lc [x]^{\theta_X}_{t}\rc\subseteq[x]^{\theta_X}_{t+2\varepsilon}$ and $\varphi_{t+\eps}\circ\psi_t\!\lc [y]^{\theta_Y}_{t}\rc\subseteq [y]^{\theta_Y}_{t+2\varepsilon}.$
\end{enumerate}
\end{definition}

\begin{remark}\label{rmk:interleaving of dendrogram and merge tree}
As mentioned in \cite[Remark 62]{memoli2021gromov}, when $X$ is finite, the dendrogram $(X,\theta_X)$ maps to a canonical merge tree $(M_X,h_X)$. It is not hard to see that two finite dendrograms are $\eps$-interleaved if and only if their corresponding merge trees are $\eps$-interleaved in the sense as described in \cite{morozov2013interleaving}.
\end{remark}

It turns out that the family of morphisms $\{\varphi_t\}_{t\in[0,\infty)}$ in the above definition can be replaced by a single set map $\varphi:X\rightarrow Y$. This observation leads to the following definition. See also \Cref{fig:interleaving} for an illustration.

\begin{figure}[ht]
    \centering
    \includegraphics[width=0.5\textwidth]{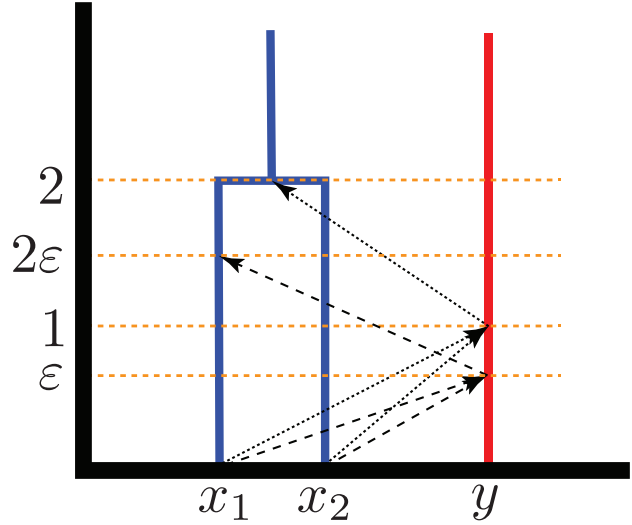}
    \caption{\textbf{Illustration of Definition \ref{def:int}.} Here we have two dendrograms with underlying sets $X=\{x_1,x_2\}$ and $Y=\{y\},$ respectively. There is only one set map $\varphi:X\rightarrow Y$ sending both points to $y$, while there are two set maps from $Y$ to $X$ sending $y$ to either $x_1$ or $x_2$. Without loss of generality, we assume that $\psi:Y\rightarrow X$ sends $y$ to $x_1$. Then, it is easy to see from the figure that when $\eps\geq 1$, $\varphi$ and $\psi$ will satisfy the conditions in Definition \ref{def:int}. However, for $\eps<1$, we see that $\psi\circ\varphi([x_2]_0^{\theta_X})=\{x_1\}$ is not a subset of $ [x_2]_{2\eps}^{\theta_X}=\{x_2\}$. Therefore, $X$ and $Y$ are not $\eps$-interleaved for $\eps<1$.}
    
    \label{fig:interleaving}
\end{figure}

\begin{definition}[Interleaving II]\label{def:int}
Given two dendrograms $(X,\theta_X)$ and $(Y,\theta_Y)$, we say they are \emph{$\varepsilon$-interleaved} for a fixed $\eps\geq 0$ if there exist set maps $\varphi :X\rightarrow Y$ and $\psi :Y\rightarrow X$ such that $\forall t\geq 0,x\in X$ and $y\in Y$ we have
\begin{enumerate}
\item $\varphi\lc[x]^{\theta_X}_t\rc\subseteq[\varphi(x)]^{\theta_Y}_{t+\varepsilon}$ and $\psi\lc[y]^{\theta_Y}_t\rc\subseteq[\psi(y)]^{\theta_X}_{t+\varepsilon},$
\item
$\psi\circ\varphi\!\lc [x]^{\theta_X}_{t}\rc\subseteq[x]^{\theta_X}_{t+2\varepsilon}$ and $\varphi\circ\psi\!\lc [y]^{\theta_Y}_{t}\rc\subseteq [y]^{\theta_Y}_{t+2\varepsilon}.$
\end{enumerate}
\end{definition}

\begin{proposition}[Equivalence of the two notions of interleaving]
Given two dendrograms $(X,\theta_X)$ and $(Y,\theta_Y)$ and $\eps\geq 0$, they are $\eps$-interleaved as in Definition \ref{def:catint} if and only if they are $\eps$-interleaved as in Definition \ref{def:int}.
\end{proposition}
\begin{proof}
The `$\Leftarrow$' direction follows easily by taking $\varphi_t\coloneqq\varphi$ and $\psi_t\coloneqq\psi$ for all $t\geq 0$.

As for the `$\Rightarrow$' direction, let $\varphi\coloneqq\varphi_0$ and $\psi\coloneqq\psi_0$. For any $t\geq 0$ and $x\in X$, we have that $[\varphi(x)]_{t+\eps}^{\theta_Y}=[\varphi_0(x)]^{\theta_Y}_{t+\eps}=[\varphi_t(x)]^{\theta_Y}_{t+\eps}$ by item 1 in Definition \ref{def:catint}. Consider the set $[x]^{\theta_X}_t$. Since $[x']_0^{\theta_X}=\{x'\}$ holds for all $x'\in X$, we have that $[x]^{\theta_X}_t=\cup_{x'\in [x]^{\theta_X}_t}[x']_0^{\theta_X}$. Then, we have
\begin{align*}
    \varphi\lc[x]^{\theta_X}_t\rc&=\varphi_0\lc[x]^{\theta_X}_t\rc=\bigcup_{x'\in [x]^{\theta_X}_t}\varphi_0([x']_0^{\theta_X})\subseteq\bigcup_{x'\in [x]^{\theta_X}_t}[\varphi_0(x')]_\eps^{\theta_Y}\\
    &\subseteq \bigcup_{x'\in [x]^{\theta_X}_t}[\varphi_0(x')]_{t+\eps}^{\theta_Y}=\bigcup_{x'\in [x]^{\theta_X}_t}[\varphi_t(x')]_{t+\eps}^{\theta_Y}\subseteq[\varphi_t(x)]_{t+\eps}^{\theta_Y}=[\varphi_0(x)]_{t+\eps}^{\theta_Y}=[\varphi(x)]_{t+\eps}^{\theta_Y},
\end{align*}
where the last inclusion map follows from the assumption that $\varphi_t:(X,\theta_X(t))\rightarrow (Y,\theta_Y(t+\eps))$ is a morphism. Therefore, $\varphi$ satisfies the condition 1 in Definition \ref{def:int}. The same result also holds for $\psi$.

As for the second condition in Definition \ref{def:int}, we have the following relations:
\begin{align*}
    \psi\circ\varphi\lc[x]^{\theta_X}_t\rc\subseteq\psi\lc[\varphi_0(x)]_{t+\eps}^{\theta_Y}\rc=\psi\lc[\varphi_t(x)]_{t+\eps}^{\theta_Y}\rc\subseteq[\psi_{0}\circ\varphi_t(x)]_{t+2\eps}^{\theta_X}=[\psi_{t+\eps}\circ\varphi_t(x)]_{t+2\eps}^{\theta_X}.
\end{align*}
Then, since $\psi_{t+\eps}\circ\varphi_t\lc[x]^{\theta_X}_t\rc\subseteq[x]_{t+2\eps}^{\theta_X}$, we have that $\psi_{t+\eps}\circ\varphi_t(x)\in[x]_{t+2\eps}^{\theta_X}$. Thus $[x]_{t+2\eps}^{\theta_X}=[\psi_{t+\eps}\circ\varphi_t(x)]_{t+\eps}^{\theta_X}$.
\end{proof}

\begin{framed}
In the rest of this section, we will adopt Definition \ref{def:int} as our definition of interleaving, which is easier to analyze.
\end{framed}

\begin{definition}[Interleaving distance]\label{def:int-dist-ddg}
Given two dendrograms $(X,\theta_X)$ and $(Y,\theta_Y)$, we define the \emph{interleaving distance} $\dI$ between them as
$$\dI\left((X,\theta_X),(Y,\theta_Y)\right)\coloneqq \inf\{\varepsilon>0:\,(X,\theta_X)\text{ and $(Y,\theta_Y)$ are $\varepsilon-$interleaved}\}.$$
\end{definition}

\subsection{Interleaving distance between ultrametric spaces} Given the equivalence between dendrograms and compact ultrametric spaces (cf. \Cref{thm:compact ultra-dendro}), we define the interleaving distance $\dI$ between compact ultrametric spaces $(X,u_X)$ and $(Y,u_Y)$ as the interleaving distance between their corresponding dendrograms $\theta_X$ and $\theta_Y$:

$$\dI((X,u_X),(Y,u_Y))\coloneqq \dI\lc(X,\theta_X),(Y,\theta_Y)\rc.$$

The following theorem characterizes interleaving between ultrametric spaces completely in terms of distance functions.

\begin{theorem}[Interleaving between ultrametric spaces]\label{thm:ibu}
Two compact ultrametric spaces $(X,u_X)$ and $(Y,u_Y)$ are $\eps$-interleaved if and only if there exist set maps $\varphi :X\rightarrow Y$ and $\psi :Y\rightarrow X$ such that for any $x,x'\in X$ and $y,y'\in Y$
\begin{enumerate}
\item $u_Y(\varphi(x),\varphi(x'))\leq u_X(x,x')+\eps$ and $u_X(\psi(y),\psi(y'))\leq u_Y(y,y')+\eps$.
\item $u_X\left(x,\psi\circ\varphi(x)\right)\leq 2\eps$ and $u_Y\left(y,\varphi\circ\psi(y)\right)\leq 2\eps. $
\end{enumerate}
\end{theorem}

\begin{proof}
$(X,u_X)$ and $(Y,u_Y)$ are $\eps$-interleaved if and only if the dendrograms $(X,\theta_X)$ and $(Y,\theta_Y)$ are $\eps$-interleaved. This is equivalent to the condition that there exist set maps $\varphi :X\rightarrow Y$ and $\psi :Y\rightarrow X$ such that $\forall t\geq 0,x\in X$ and $y\in Y$ we have
\begin{enumerate}
\item $\varphi\lc[x]^{\theta_X}_t\rc\subseteq[\varphi(x)]^{\theta_Y}_{t+\varepsilon}$ and $\psi\lc[y]^{\theta_Y}_t\rc\subseteq[\psi(y)]^{\theta_X}_{t+\varepsilon},$
\item
$\psi\circ\varphi\lc[x]_{t}^{\theta_X}\rc\subseteq[x]^{\theta_X}_{t+2\varepsilon}$ and $\varphi\circ\psi\lc[y]^{\theta_Y}_{t}\rc\subseteq[y]^{\theta_Y}_{t+2\varepsilon},$
\end{enumerate}

Since $x'\in[x]_t^{\theta_X}$ if and only if $u_X(x,x')\leq t$, the first item implies that for any $x'\in X$ such that $u_X(x,x')\leq t$, we have $u_Y(\varphi(x'),\varphi(x))\leq t+\eps$. By letting $t \coloneqq u_X(x,x')$, then $u_Y(\varphi(x'),\varphi(x))\leq u_X(x,x')+\eps$ and symmetrically, $u_X(\psi(y),\psi(y'))\leq u_Y(y,y')+\eps$. 
It is easy to derive from item 2 that $u_X(x,\psi\circ\varphi(x))\leq t+2\eps$ for any $t\geq 0$ and thus by taking $t=0$, we obtain $u_X(x,\psi\circ\varphi(x))\leq 2\eps$ and similarly $u_Y\left(y,\varphi\circ\psi(y)\right)\leq 2\eps. $

Conversely, let $\varphi:X\rightarrow Y$ and $\psi:Y\rightarrow X$ be such that the conditions in the theorem hold. Given any $t\geq 0$ and $x\in X$, if $x'\in X$ is such that $u_X(x,x')\leq t$, then
$$u_Y( \varphi (x), \varphi (x'))\leq u_X(x,x')+\eps\leq t+\eps ,$$
which implies that $ \varphi (x')\in[ \varphi (x)]^{\theta_Y}_{t+\eps}$. Therefore, $ \varphi \lc[x]^{\theta_X}_t\rc\subseteq[ \varphi (x)]^{\theta_Y}_{t+\eps}.$ Moreover,
\begin{align*}
    u_X(\psi\circ \varphi(x'),x)&\leq \max\left(u_X(\psi\circ \varphi(x'),x'),u_X(x',x)\right)\\
    &\leq \max(2\eps,t)\leq t+2\eps.
\end{align*}
Hence $\psi\circ \varphi\lc[x]^{\theta_X}_t\rc\subseteq [x]^{\theta_X}_{t+2\eps}$. Similarly for any $y\in Y$, $\psi\lc[y]^{\theta_Y}_t\rc\subseteq[\psi(y)]^{\theta_Y}_{t+\eps}$ and $\varphi\circ \psi\lc[y]^{\theta_Y}_t\rc\subseteq[y]^{\theta_Y}_{t+2\eps}.$ This shows that $\varphi$ and $\psi$ induce an $\eps$-interleaving between $(X,\theta_X)$ and $(Y,\theta_Y).$
\end{proof}

Due to the above theorem, we propose the following alternative definition for the interleaving distance $\dI$ on the collection $\ums$ of all compact ultrametric spaces. 
\begin{definition}\label{def:interleaving ums}
Given $X,Y\in\ums$ and $\eps\geq 0$, we say $X$ and $Y$ are $\eps$-interleaved if there exist set maps $\varphi :X\rightarrow Y$ and $\psi :Y\rightarrow X$ such that for any $x,x'\in X$ and $y,y'\in Y$
\begin{enumerate}
\item $u_Y(\varphi(x),\varphi(x'))\leq u_X(x,x')+\eps$ and $u_X(\psi(y),\psi(y'))\leq u_Y(y,y')+\eps$.
\item $u_X\left(x,\psi\circ\varphi(x)\right)\leq 2\eps$ and similarly, $u_Y\left(y,\varphi\circ\psi(y)\right)\leq 2\eps. $
\end{enumerate}
We define the interleaving distance between $X$ and $Y$ as follows:
$$\dI\left((X,u_X),(Y,u_Y)\right)\coloneqq \inf\{\varepsilon>0:\,(X,u_X)\text{ and $(Y,u_Y)$ are $\varepsilon-$interleaved}\}.$$
\end{definition}

\subsection{Characterization of the interleaving distance between ultrametric spaces} Given compact ultrametric spaces $X$ and $Y$ and a map $\varphi:X\rightarrow Y$,  we define the \emph{I-distortion} of $\varphi$ as follows: 
\begin{equation}\label{eq:i-distortion}
    \disi(\varphi,u_X,u_Y)\coloneqq \inf\left\{\eps\geq 0:\, u_Y(\varphi(x),\varphi(x'))\leq u_X(x,x')+\eps,\,\,\forall x,x'\in X\right\}.
\end{equation}
Given another map $\psi:Y\rightarrow X$, we define the I-codistortion of $(\varphi,\psi)$ as follows:
\begin{equation}\label{eq:i-codistortion}
    \codisi(\varphi,\psi,u_X,u_Y)\coloneqq \frac{1}{2}\max\left(\sup_{x\in X}u_X(x,\psi\circ\varphi(x)),\sup_{y\in Y}u_Y(y,\varphi\circ\psi(y))\right).
\end{equation}
We will use the abbreviations $\disi(\varphi)$ and $\codisi(\varphi,\psi)$ when the underlying metric structures are clear.

\begin{remark}\label{rmk:disanddisi}
It is easy to check that 
$$\disi(\varphi)=\sup_{x,x'\in X}\lc u_Y(\varphi(x),\varphi(x'))-u_X(x,x')\rc. $$
Hence, by Equation (\ref{eq:dismap}), we have that $\disi(\varphi)\leq \dis(\varphi)$. Moreover,
$$2\,\codisi(\varphi,\psi)=\sup\left\{|u_X(x,\psi(y))-u_Y(\varphi(x),y)|:\, {x\in X,y=\varphi(x)\text{ or }y\in Y, x=\psi(x)}\right\}.$$
Hence, by Equation (\ref{eq:codismap}), we have that $2\,\codisi\leq \codis$.
\end{remark}

\begin{theorem}\label{thm:intdist}
Given $X,Y\in\ums$,
$$ \dI(X,Y)=\inf_{\varphi,\psi}\max\left(\disi(\varphi),\disi(\psi),\codisi(\varphi,\psi)\right),$$
where the infimum is taken over all maps $\varphi:X\rightarrow Y$ and $\psi:Y\rightarrow X$.
\end{theorem}

Notice that the structure of the right-hand side in the above equation is almost the same (up to a $\frac{1}{2}$ factor) as in the Gromov-Hausdorff distance formula given in \Cref{eq:distmap}. This characterization allows us to directly compare $\dI$ and $\dgh$ (see \Cref{coro:int}).

\begin{proof}[Proof of Theorem \ref{thm:intdist}]
We first assume that $X$ and $Y$ are $\eps$-interleaved through the maps $\varphi:X\rightarrow Y$ and $\psi:Y\rightarrow X$ for some $\eps\geq 0$. Then, by condition 1 of \Cref{def:interleaving ums}, one has $u_Y(\varphi(x),\varphi(x'))\leq u_X(x,x')+\eps$ for any $x,x'\in X$ and thus $\disi(\varphi)\leq \eps.$ Similarly, $\disi(\psi)\leq \eps.$ Directly from condition 2 of \Cref{def:interleaving ums}, we conclude that $\codisi(\varphi,\psi)\leq \eps$. Therefore, $\max\left(\disi(\varphi),\disi(\psi),\codisi(\varphi,\psi)\right)\leq \eps$ and thus $\max\left(\disi(\varphi),\disi(\psi),\codisi(\varphi,\psi)\right)\leq \dI(X,Y)$.

Conversely, assume that $\max\left(\disi(\varphi),\disi(\psi),\codisi(\varphi,\psi)\right)\leq\eps$ for $\varphi:X\rightarrow Y$ and $\psi:Y\rightarrow X$ and some $\eps\geq 0$. Then, by \Cref{eq:i-distortion} and \Cref{eq:i-codistortion}, it is easy to check the following claims.
\begin{enumerate}
    \item $u_Y( \varphi(x), \varphi(x'))\leq u_X(x,x')+\eps$ and $u_X( \psi(y), \psi(y'))\leq u_Y(y,y')+\eps$.
\item  $u_X\left(x, \psi\circ \varphi(x)\right)\leq 2\eps $ and similarly, $u_Y\left(y, \varphi\circ \psi(y)\right)\leq 2\eps. $
\end{enumerate}
Then, we conclude that $\dI(X,Y)\leq \eps$ and thus $\dI(X,Y)\leq\max\left(\disi(\varphi),\disi(\psi),\codisi(\varphi,\psi)\right)$.
\end{proof}

\begin{corollary}[Bi-Lipschitz equivalence with $\dgh$]\label{coro:int}
For any $X,Y\in\ums$, we have
$$\frac{1}{2}\dI(X,Y)\leq \dgh(X,Y)\leq \dI(X,Y). $$
\end{corollary}

\begin{example}[Tightness of the coefficients] Consider the two-point spaces $\Delta_2(2)$ and $\Delta_2(4)$. It is not hard to check that $\dI(\Delta_2(2),*)=1=\dgh(\Delta_2(2),*)$ and $\dI(\Delta_2(2),\Delta_2(4))=2=2\,\dgh(\Delta_2(2),\Delta_2(4)).$
This means that both inequalities above can in fact be equalities. 
\end{example}

\begin{proof}[Proof of Corollary \ref{coro:int}]
We first prove the rightmost inequality. Assume that there exist $\varphi:X\rightarrow Y$ and $\psi:Y\rightarrow X$ inducing an $\eps$-interleaving between $X$ and $Y$. Consider the correspondence between $X$ and $Y$ generated by the interleaving maps $\varphi$ and $\psi$: 
\[R\coloneqq\{(x,y)\in X\times Y:\, \varphi(x)=y\text{ or }\psi(y)=x\}.\]
Now, we prove that $\dis(R)\leq 2\varepsilon$ which will imply that $\dgh(X,Y)\leq \eps$. It suffices to prove for any given $(x,y),(x',y')\in R$ that $\big|u_X(x,x')-u_Y(y,y')\big|\leq 2\eps$. Due to the symmetric role of $\varphi$ and $\psi$, it suffices to check the following two cases:

\begin{enumerate}
\item $y=\varphi(x)$ and $y'=\varphi(x')$. By \Cref{def:interleaving ums} we have that $u_X(x,x')+\eps\geq u_Y(\varphi(x),\varphi(x'))=u_Y(y,y').$ On the other hand, we have that 
\begin{align*}
    u_X(x,x')&\leq\max\left(u_X(x,\psi\circ\varphi(x)),u_X(\psi\circ\varphi(x),\psi\circ\varphi(x')),u_X(\psi\circ\varphi(x'),x')\right)\\
    &\leq \max\left(2\eps,u_Y(\varphi(x),\varphi(x'))+\eps,2\eps\right)\leq u_Y(y,y')+2\eps.
\end{align*}
Hence $\big|u_X(x,x')-u_Y(y,y')\big|\leq 2\eps$.

\item $y=\varphi(x)$ and $x'=\psi(y')$. Then, 
\begin{align*}
    u_X(x,x')&\leq\max\left(u_X(x,\psi\circ\varphi(x)),u_X(\psi\circ\varphi(x),\psi(y'))\right)\\
    &\leq \max\left(2\eps,u_Y(\varphi(x),y')+\eps\right)\leq u_Y(y,y')+2\eps.
\end{align*}
Similarly, $u_Y(y,y')\leq u_X(x,x')+2\eps$, and thus $\big|u_X(x,x')-u_Y(y,y')\big|\leq 2\eps$.
\end{enumerate}

The leftmost inequality follows directly from Theorem \ref{thm:intdist}. Assume that $\dgh(X,Y)\leq \eps$, then by Equation (\ref{eq:distmap}) there are two maps $\varphi:X\rightarrow Y$ and $\psi:Y\rightarrow X$ such that 
$$\dis(\varphi),\dis( \psi),\codis(\varphi, \psi)\leq 2\eps.$$
Then, it is immediate that $\disi(\varphi)\leq \dis(\varphi)\leq 2\eps$ where the first inequality follows from Remark \ref{rmk:disanddisi}. Similarly, $\disi( \psi)\leq2\eps$. As for $\codisi,$ we have by Remark \ref{rmk:disanddisi} again that $\codisi(\varphi,\psi)\leq \frac{1}{2}\codis(\varphi,\psi)\leq \eps$. Thus, $\dI(X,Y)\leq 2\eps$ and since $\eps\geq \dgh(X,Y)$ was arbitrary, we obtain that $\dI(X,Y)\leq 2\,\dgh(X,Y).$
\end{proof}

\subsection{$p$-interleaving distance for dendrograms and compact ultrametric spaces}\label{sec:p-int}

In the definition (cf. \Cref{def:int}) of $\dI$ between dendrograms, we implicitly used a shift operator, namely we considered a $+\eps$ shift of dendrograms. Replacing $+$ with $\ps{p}$, we will obtain the so-called $p$-interleaving. 

\begin{definition}\label{def:int-dis-p}
Given two dendrograms $(X,\theta_X)$ and $(Y,\theta_Y)$, we say they are $(\varepsilon,p)$-interleaved for some $\eps>0$ and $p\in[1,\infty]$ if there exist set maps $\varphi :X\rightarrow Y$ and $\psi :Y\rightarrow X$ such that $\forall t\geq 0,x\in X$ and $y\in Y$ we have
\begin{enumerate}
\item $\varphi\lc[x]^{\theta_X}_t\rc\subseteq[\varphi(x)]^{\theta_Y}_{t\ps{p} \varepsilon}$ and $\psi\lc[y]^{\theta_Y}_t\rc\subseteq[\psi(y)]^{\theta_X}_{t\ps{p} \varepsilon},$
\item
$[x]^{\theta_X}_{t\ps{p} \varepsilon\ps{p} \eps}=[\psi\circ\varphi(x)]^{\theta_X}_{t\ps{p} \varepsilon\ps{p} \eps}$ and $[y]^{\theta_Y}_{t\ps{p} \varepsilon\ps{p} \eps}=[\varphi\circ\psi(y)]^{\theta_Y}_{t\ps{p} \eps\ps{p} \eps}.$
\end{enumerate}

We then define the $p$-interleaving distance between $(X,\theta_X)$ and $(Y,\theta_Y)$ as 
$$\dIp{p}\left((X,\theta_X),(Y,\theta_Y)\right)\coloneqq \inf\{\varepsilon>0:\,(X,\theta_X)\text{ and $(Y,\theta_Y)$ are $(\varepsilon,p)-$interleaved}.\}$$
\end{definition}

\begin{remark}
Note that when $p=1$, $(\eps,1)$-interleaving is exactly the $\eps$-interleaving given in Definition \ref{def:int}. When $p=\infty$, the two conditions become
\begin{enumerate}
\item $\varphi\lc[x]^{\theta_X}_t\rc\subseteq[\varphi(x)]^{\theta_Y}_{\max(t,\eps)}$ and $\psi\lc[y]^{\theta_Y}_t\rc\subseteq[\psi(y)]^{\theta_X}_{\max(t,\eps)},$
\item
$[x]^{\theta_X}_{\max(t,\eps)}=[\psi\circ\varphi(x)]^{\theta_X}_{\max(t,\eps)}$ and $[y]^{\theta_Y}_{\max(t,\eps)}=[\varphi\circ\psi(y)]^{\theta_Y}_{\max(t,\eps)}.$
\end{enumerate}
It is easy to check that if both conditions hold for $t=\eps$, then they hold for all $0\leq t\leq \eps$. This indicates that $(\eps,\infty)$-interleaving is performing some sort of coarsening of dendrograms in that information corresponding to $t<\eps$ is discarded. Careful readers may notice a similar phenomenon in Definition \ref{def:ultraquotient}. Given the later structural result (\Cref{thm:ugh-structure}) characterizing $\ugh$ via closed quotients, this resemblance actually hints at a close relation between $\dIp{\infty}$ and $\ugh$. See also later Remark \ref{rmk:infint}.
\end{remark}

A characterization result similar to Theorem \ref{thm:intdist} also holds for $p$-interleaving distance. We first define the \emph{$p$-I-distortion} of a map $\varphi:X\rightarrow Y$:

$$\disip{p}(\varphi,u_X,u_Y)\coloneqq \sup_{x,x'\in X}A_p\left(u_Y(\varphi(x),\varphi(x')),u_X(x,x')\right).$$
Recall that $A_p$ is the asymmetric $p$-difference defined in Equation (\ref{eq:asymmetric}). 

Similarly given $\psi:Y\rightarrow X$, we define the \emph{$p$-I-codistortion} between $\varphi$ and $\psi$ by $$\codisip{p}(\varphi,\psi,u_X,u_Y)\coloneqq 2^{-\frac{1}{p}}\max\left(\max_{x\in X}u_X(x,\psi\circ\varphi(x)),\max_{y\in Y}u_Y(y,\varphi\circ\psi(y))\right).$$ 
Same as before, we will use the abbreviations $\disip{p}(\varphi)$ and $\codisip{p}(\varphi,\psi)$ when the underlying metric structures are clear.

\begin{theorem}\label{thm:pintdist}
Given $X,Y\in\ums$ and $p\in[1,\infty]$,
$$ \dIp{p}(X,Y)=\inf_{\varphi,\psi}\max\left(\disip{p}(\varphi),\disip{p}(\psi),\codisip{p}(\varphi,\psi)\right),$$
where the infimum is taken over all maps $\varphi:X\rightarrow Y$ and $\psi:Y\rightarrow X$.  
\end{theorem}

The proof of the theorem is essentially the same as the proof of Theorem \ref{thm:intdist} so we omit it.

With this theorem, it is easy to derive the following relation between $\dIp{p}$ and $\dghp{p}$ in analogy with Corollary \ref{coro:int}.

\begin{corollary}\label{coro:pint vs dghp}
For any $X,Y\in\ums$, one has for any $p\in[1,\infty]$
$$2^{-\frac{1}{p}}\dIp{p}(X,Y)\leq \dghp{p} (X,Y)\leq  \dIp{p}(X,Y).$$
\end{corollary}

\begin{remark}[Relation with $\ugh$]\label{rmk:infint}
Note that when $p=\infty$, we have that $2^{-\frac{1}{\infty}}=1$ and thus $$\dIp{\infty}=\dghp{\infty}=\ugh.$$ This statement provides us an alternative proof to Theorem \ref{thm:ugh-structure}:

Given two maps $\varphi:X\rightarrow Y$ and $\psi:Y\rightarrow X$ such that $X$ and $Y$ are $(t,\infty)$-interleaved (thus $\ugh(X,Y)=\dIp{\infty}(X,Y)\leq t$), we construct $\varphi_t:X_t\rightarrow Y_t$ and $\psi_t:Y_t\rightarrow X_t$ as $\varphi_t([x]^X_t)=[\varphi(x)]^Y_t$ and $\psi_t([y]^Y_t)=[\psi(y)]^X_t$ for $x\in X$ and $y\in Y$. Then, it is easy to show that these two maps are isometries  and $\varphi_t=\psi_t^{-1}$. Conversely, if there are isometries $\varphi_t:X_t\rightarrow Y_t$ and $\psi_t:Y_t\rightarrow X_t$ such that $\varphi_t=\psi_t^{-1}$ at $t\geq 0$, then we construct $\varphi:X\rightarrow Y$ as follows: $\varphi(x)=y$, where $y$ is arbitrarily chosen such that $y\in\varphi_t([x]^X_t)$. We construct $\psi:Y\rightarrow X$ similarly. Then, it is easy to check that $\varphi$ and $ \psi$ make $X$ and $Y$ be $(t,\infty)$-interleaved and thus $\ugh(X,Y)=\dIp{\infty}(X,Y)\leq t$.

\end{remark}

\begin{example}\label{rmk:intdiam}
If $X=*$ is the one point space, then for any $Y\in\ums$, we have
$$\dIp{p}(X,Y)=2^{-\frac{1}{p}}\diam(Y). $$
Indeed, there exists only one map $\psi:Y\rightarrow X$. For any map $\varphi:X\rightarrow Y$, it is easy to check that $\disip{p}(\varphi)=\disip{p}(\psi)=0$. Let $z=\varphi\circ\psi(y)$, which is invariant of choice of $y\in Y$. Since $\max_{y\in Y}u_Y(y,z)=\diam(Y)$, we have that $\codisip{p}(\varphi,\psi)=2^{-\frac{1}{p}}\diam(Y)$ and thus by Theorem \ref{thm:pintdist} we have that $\dIp{p}(X,Y)=2^{-\frac{1}{p}}\diam(Y). $
\end{example}

\section{Topological and geodesic properties of $\left(\msp,\dghp{p}\right)$} \label{sec:topology}
In this section, we study topological and geodesic properties of $\left(\msp,\dghp{p}\right)$. We will characterize convergence sequences in $(\msp,\dghp{p})$ and derive a pre-compactness result. We show that $(\msp,\dghp{p})$ is a complete and separable space when $p<\infty$. Recall from \Cref{prop:inclusionpq} that $\msp\subseteq\ms_q$ when $q<p$. This leads us to also study the subspace topology of $\msp$ inside $\ms_q$. Finally, we study geodesic properties of $\msp$.

\subsection{Convergence under $\dghp{p}$}
In this section, we will study convergent sequences in $\left(\msp,\dghp{p}\right)$. 

\begin{definition}
Let $1\leq p\leq\infty$. We say a sequence $\{X_n\}_{n=1}^\infty$ in $\msp$ converges to $X\in\msp$ if $\dghp{p}(X_n,X)\rightarrow 0$ as $n\rightarrow\infty$. Since $\dghp{p}$ is a $p$-metric, the limit is unique up to isometry. We call $X$ the $\dghp{p}$-limit of $\{X_n\}_{n=1}^\infty$. 
\end{definition}

We have the following convergence criterion for $\dghp{p}$ generalizing the convergence criterion for $\dgh=\dgh^{\scriptscriptstyle{(1)}}$ mentioned in Section 7.4.1 of \cite{burago2001course}: a sequence $\{X_n\}_{n\in\N}$ of $p$-metric spaces converges to a $p$-metric space $X$ if and only if there are a sequence $\{\eps_n\}_{n\in\N}$ of positive numbers and a sequence of maps $\{f_n:X_n\rightarrow X\}_{n\in\N}$ (or, alternatively, $f_n:X\rightarrow X_n$) such that every $f_n$ is an $(\eps_n,p)$-isometry and $\eps_n\rightarrow 0$.

\begin{example}\label{ex:conv-pms}
Every compact $p$-metric space $X$ is the $\dghp{p}$-limit of some sequence of finite $p$-metric spaces. This is the counterpart to \cite[Example 7.4.9]{burago2001course} and the proof is similar (it follows by considering $\eps$-nets of $X$).
\end{example}

This example actually indicates that convergence of compact $p$-metric spaces may reduce to convergence of their corresponding $\eps$-nets. To make this precise, we define the notion of $(\eps,\delta,p)$-approximation as follows:

\begin{definition}\label{def:approximation}
Fix $1\leq p\leq \infty$. Let $X$ and $Y$ be two compact $p$-metric spaces, and $\eps,\delta\geq0$. We say that $X$ and $Y$ are \emph{$(\eps,\delta,p)$-approximation} of each other if there exist finite sets $\{x_i\}_{i=1}^N\subseteq X$ and $\{y_i\}_{i=1}^N \subseteq Y$ such that:
\begin{enumerate}
    \item $\{x_i\}_{i=1}^N$ is an $\eps$-net for $X$ and $\{y_i\}_{i=1}^N$ is an $\eps$-net for $Y$.
    \item $\Lambda_p(d_X(x_i,x_j), d_Y(y_i,y_j))\leq\delta$ for all $i,j=1,\cdots,N$.
\end{enumerate}
\end{definition}

\begin{proposition}\label{thm:approxdlp}
Fix $1\leq p\leq\infty$. Let $X$ and $Y$ be two compact $p$-metric spaces.
\begin{enumerate}
    \item If $X$ and $Y$ are $(\eps,\delta,p)$-approximation of each other, then $\dghp{p}(X,Y)\leq \delta\ps{p}2^\frac{1}{p}\eps$.
    \item If $\dghp{p}(X,Y)\leq\eps$, then $Y$ (resp. $X$) is a $\left(5^\frac{1}{p}\eps,2^\frac{1}{p}\eps,p\right)$-approximation of $X$ (resp. $Y$).
\end{enumerate}
\end{proposition}
The proof is similar to the one for \cite[Proposition 7.4.11]{burago2001course}.
\begin{proof}
\begin{enumerate}
    \item Let $X_0=\{x_i\}_{i=1}^N$ and $Y_0=\{y_i\}_{i=1}^N$ be as in Definition \ref{def:approximation}. Then, the second condition in that definition implies that the correspondence $\{(x_i,y_i):\,i=1,\cdots,N\}$ between $X_0$ and $Y_0$ has $p$-distortion bounded above by $\delta$. Hence, $\dghp{p}(X_0,Y_0)\leq 2^{-\frac{1}{p}}\delta\leq \delta$. By Example \ref{ex:dlpnet} we know that $\dghp{p}(X_0,X),\dghp{p}(Y_0,Y)\leq \eps$. Thus 
    $$\dghp{p}(X,Y)\leq \dghp{p}(X,X_0)\ps{p}\dghp{p}(X_0,Y_0)\ps{p}\dghp{p}(Y_0,Y)\leq\delta\ps{p}2^\frac{1}{p}\eps. $$
    
    \item By \Cref{coro:isometrychar} there exists a $(2^\frac{1}{p}\eps,p)$-isometry $f:X\rightarrow Y$. Let $X_0=\{x_i\}_{i=1}^N$ be an $\eps$-net in $X$ and let $y_i\coloneqq f(x_i)$ for each $i=1,\cdots,N$. Then, $\Lambda_p(d_X(x_i,x_j), d_Y(y_i,y_j))\leq \disp{p}(f)\leq 2^\frac{1}{p}\eps$ for all $i,j$. Now, since $f(X)$ is a $2^\frac{1}{p}\eps$-net of $Y$, for any $y\in Y$, there exists $x\in X$ such that $d_Y(f(x),y)\leq 2^\frac{1}{p}\eps$. Since $X_0$ is an $\eps$-net in $X$, we can choose $x_i\in X_0$ such that $d_X(x,x_i)\leq \eps$. Then, we have
    \begin{align*}
        d_Y(y,y_i)&\leq d_Y(y,f(x))\ps{p} d_Y(f(x),f(x_i))\\
        &\leq 2^\frac{1}{p}\eps\ps{p} d_X(x,x_i)\ps{p} 2^\frac{1}{p}\eps\leq 5^\frac{1}{p}\eps
    \end{align*}
    Thus $\{y_i\}_{i=1}^N=f(X_0)$ is a $5^\frac{1}{p}\eps$-net of $Y$.
\end{enumerate}
\end{proof}

In \cite{qiu2009geometry}, Qiu introduced a notion called \emph{strong $\eps$-approximation}, which is exactly the $(\eps,0,\infty)$-approximation using our language. The following corollary is a restatement of \cite[Theorem 3.5]{qiu2009geometry} regarding the strong $\eps$-approximation. Though the corollary seems stronger than the result in the case $p=\infty$ of our \Cref{thm:approxdlp}, it turns out that they are equivalent which will be clarified in the proof. We will include a proof of the following corollary for completeness.

\begin{corollary}
Let $X$ and $Y$ be two compact ultrametric spaces.
\begin{enumerate}
    \item If $X$ and $Y$ are $(\eps,0,\infty)$-approximation of each other, then $\ugh(X,Y)\leq \eps$.
    \item If $\ugh(X,Y)\leq\eps$, then $Y$ is a $\left(\eps,0,\infty\right)$-approximation of $X$.
\end{enumerate}
\end{corollary}

\begin{proof}
The first claim follows directly from claim 1 of \Cref{thm:approxdlp}. 

For the second claim, first note that claim 2 of \Cref{thm:approxdlp} shows that $Y$ is a $(\eps,\eps,\infty)$-approximation of $X$. To conclude the proof, we only need to show that an $(\eps,\eps,\infty)$-approximation is automatically an $(\eps,0,\infty)$-approximation.

Since $Y$ is a $(\eps,\eps,\infty)$-approximation of $X$, there exist $\eps$-nets $\{x_i\}_{i=1}^N\subseteq X$ and $\{y_i\}_{i=1}^N \subseteq Y$ such that $\Lambda_\infty(u_X(x_i,x_j),u_Y(y_i,y_j))\leq \eps$ for all $i,j$. If $u_X(x_i,x_j)\leq \eps$, then we also have that $u_Y(y_i,y_j)\leq\eps$. Then, $B_\eps(x_i)=B_\eps(x_j)$ and $B_\eps(y_i)=B_\eps(y_j)$ by \Cref{prop:basic property ultrametric}. This implies that after discarding $x_j$ and $y_j$, $\{x_i\}_{i\neq j}$ and $\{y_i\}_{i\neq j} $ remain $\eps$-nets of $X$ and of $Y$, respectively. We continue this process to obtain two subsets $\{x_{n_i}\}_{i=1}^M$ and $\{y_{n_i}\}_{i=1}^M$ which are still $\eps$-nets of $X$ and of $Y$, respectively, while $u_X(x_{n_i},x_{n_j}),u_Y(y_{n_i},y_{n_j})>\eps$ for all $i\neq j$. Then, by $\Lambda_\infty(u_X(x_{n_i},x_{n_j}), u_Y(y_{n_i},y_{n_j}))\leq \eps$ we have that $u_X(x_{n_i},x_{n_j})=u_Y(y_{n_i},y_{n_j})$ and thus $\Lambda_\infty(u_X(x_{n_i},x_{n_j}), u_Y(y_{n_i},y_{n_j}))\leq 0$. Then, we conclude that $Y$ is an $(\eps,0,\infty)$-approximation of $X$.
\end{proof}

\subsection{Pre-compactness theorems} In \cite{gromov1981groups}, Gromov proved a well known pre-compactness theorem stating that any uniformly totally bounded collections of compact metric spaces are pre-compact in the Gromov-Hausdorff sense. We included the result as follows for completeness.

\begin{definition}[Uniformly totally bounded class]\label{def:CND}
We say a class $\mathcal{K}$ of compact metric spaces is \emph{uniformly totally bounded}, if there exist a bounded function $Q:\left(0,\infty\right)\rightarrow\mathbb{N}$ and $D>0$ such that each $X\in\mathcal{K}$ satisfies the following properties: 
\begin{enumerate}
    \item $\diam\left(X\right)\leq D$,
    \item for any $\eps>0$, $\mathrm{cov}_\eps\left(X\right)\leq Q\left(\eps\right)$.
\end{enumerate}
We denote by $\mathcal{K}\left(Q,D\right)$ the uniformly totally bounded class consisting of all $X\in\ms$ satisfying the conditions above.
\end{definition}

\begin{theorem}[Gromov's pre-compactness theorem]\label{thm:pre-compact}
For any given bounded function $Q:\left(0,\infty\right)\rightarrow\mathbb{N}$ and $D>0$, the class $\mathcal{K}\left(Q,D\right)$ is pre-compact in $\left(\ms,\dgh\right)$, i.e., any sequence in $\mathcal{K}\left(Q,D\right)$ has a convergent subsequence.
\end{theorem}

Interested readers are referred to \cite[Section 7.4.2]{burago2001course} for a proof.

In this section, we generalize \Cref{thm:pre-compact} to the setting of $p$-metric spaces and the $\dghp{p}$ distance, for $1<p<\infty$, by invoking both Theorem \ref{thm:eqdgh} and the following lemma.

\begin{lemma}
For any $1\leq p\leq \infty$, $(\msp,\dgh)$ is a closed subspace of $(\ms,\dgh)$.
\end{lemma}

This lemma is in fact a special case of the more general \Cref{prop:closed-psub} so we omit its proof here.

\begin{proposition}[$\dghp{p}$ pre-compactness theorem]\label{coro:precpt}
Fix $1\leq p<\infty$. Any uniformly totally bounded collection $\mathcal{K}$ of compact $p$-metric spaces is pre-compact, i.e., any sequence in $\mathcal{K}$ has a convergent subsequence in the sense of $\dghp{p}$. 
\end{proposition}

\begin{proof}
Given any sequence $\{X_n\}_{n=1}^\infty$ in $\mathcal{K}$, by Gromov's pre-compactness theorem, there exists a $\dgh$ convergent subsequence. Without loss of generality, we assume that $\{X_n\}_{n=1}^\infty$ is itself a convergent sequence and that $X$ is its Gromov-Hausdorff limit. By the previous lemma, we have that $X\in\msp$. For any $n\in\mathbb{N}$, since $\diam(X_n)\leq D$ for some $D>0$, by Theorem \ref{thm:eqdgh}, we have that $\{X_n\}_{n=1}^\infty$ will also converge to $X$ in the sense of $\dghp{p}$. 
\end{proof}

Note that the uniformly totally boundedness condition does not guarantee pre-compactness of collections of ultrametric spaces. 

\begin{example}
Consider the collection of 2-point spaces $\left\{\Delta_2\left(1+\frac{1}{n}\right)\right\}_{n=1}^\infty$. This collection is obviously uniformly totally bounded. However, for any $n,m\in\mathbb{N}$, we have $$\ugh\left(\Delta_2\left(1+\frac{1}{n}\right),\Delta_2\left(1+\frac{1}{m}\right)\right)=1+\max\left(\frac{1}{n},\frac{1}{m}\right)>1.$$

Therefore, $\left\{\Delta_2\left(1+\frac{1}{n}\right)\right\}_{n=1}^\infty$ contains no Cauchy subsequence and thus it is not pre-compact.
\end{example}

Under a certain variant of the notion of uniformly totally boundedness, in \cite{qiu2009geometry} Qiu proved a pre-compactness theorem for $\ugh$. We  include it here for completeness.

\begin{definition}
A collection $\mathcal{K}$ of compact ultrametric spaces is called \emph{strongly uniformly totally bounded}, if for any $\eps>0$, there exist a positive integer $N=N(\eps)$ and a finite set $R(\eps)\subseteq\R_{\geq 0}$ such that every $X\in\mathcal{K}$ contains an $\eps$-net $S_X$ with $\#S_X\leq N$ and $\mathrm{spec}(S_X)\subseteq R(\eps)$.

\end{definition}

\begin{theorem}[$\ugh$ pre-compactness theorem, \cite{qiu2009geometry}]
Any strongly uniformly totally bounded collection $\mathcal{K}$ of compact ultrametric spaces is pre-compact.
\end{theorem}

\subsection{Completeness and separability of $(\msp,\dghp{p})$} With the tools we have developed so far, we can establish the following theorem.

\begin{theorem}\label{thm:complete-sep}
For each $1\leq p< \infty$, $(\msp,\dghp{p})$ is complete and separable.
\end{theorem}

\begin{proof}
Fix a Cauchy sequence $\{X_n\}_{n\in\N}$ in $\msp$. Then, obviously there exists $D>0$ such that $\diam(X_n)\leq D$ for any $n\in\mathbb{N}$. Given any $\eps>0$, let $M\in\N$ be such that for any $n>M$, one has $\dghp{p}(X_M,X_n)<\eps$. Then, by \Cref{thm:approxdlp} we have that $X_M$ is a $(5^\frac{1}{p}\eps,2^\frac{1}{p}\eps,p)$-approximation of $X_n$. Fix an $\eps$-net $M_\eps$ of $X_M$. Then, as shown in the proof of item 2 of \Cref{thm:approxdlp}, there exists an $\eps$-net in $X_n$ with the same cardinality as $M_\eps$. This implies that there exists $N=N(\eps)$ such that for all $n\in \mathbb{N}$ there exists an $\eps$-net in $X_n$ with cardinality bounded by $N$. Applying the $\dghp{p}$-pre-compactness theorem (\Cref{coro:precpt}) we have that there exists a convergent subsequence of $\{X_n\}_{n\in\N}$, which implies that $\{X_n\}_{n\in\N}$ itself is convergent since it is Cauchy. Therefore, $(\msp,\dghp{p})$ is complete.

By $\msp^{(n)}$ denote the set of all $n$-point $p$-metric spaces with rational distances. Then, it is easy to check that $\bigcup_{n=1}^\infty\msp^{(n)}$ is a countable dense set in $\msp$ and thus $(\msp,\dghp{p})$ is separable.
\end{proof}

\begin{remark}
The proof above does not directly apply to the case when $p=\infty$. In fact, by using Qiu's pre-compactness theorem and the notion of strong approximation in \cite{qiu2009geometry}, a slight modification of the above proof will establish the completeness of $\ums$. Interested readers are also referred to \cite{zarichnyi2005gromov} for a different method that proves completeness of $(\ums,\ugh)$. However, it is shown in \cite{zarichnyi2005gromov} that $(\ums,\ugh)$ is not a separable space, which suggests that $\ums$ enjoys some special properties over all other $\msp$.

\end{remark}


\subsection{Subspace topology}\label{sec:subspace topology}
As shown in Proposition \ref{prop:inclusionpq} that $\msp\subseteq \ms_q$ when $1\leq q<p\leq \infty$, we now study the topology of $(\msp,\dghp{q})$ as a subspace of $(\ms_q,\dghp{q}).$ We need the following technical lemma about the relation between $ A_q$ and $\ps{p}$ when $p\neq q$ (see \Cref{eq:asymmetric} for the definition of $A_q$).

\begin{lemma}\label{lm:pqineq}
For $1\leq q<p\leq\infty$ and $a,b,c\geq 0$, we have
$$ A_q(a,c)\ps{p} A_q(b,c)\geq  A_q\lc a\ps{p} b, c\ps{p} c\rc. $$
\end{lemma}
\begin{proof}
When $p=\infty$, it is easy to see that $\max\lc A_q(a,c), A_q(b,c)\rc= A_q\lc \max(a,b),c\rc$, which is exactly what we want.

When $p<\infty$, we have the following cases:
\begin{enumerate}
    \item $a,b\leq c$. Then, both sides of the inequality become 0, and thus the equality holds
    \item $a,b\geq c$. Then, we need to prove the following:
    $$\lc\lc a^q-c^q\rc^\frac{p}{q}+\lc b^q-c^q\rc^\frac{p}{q}\rc^\frac{1}{p}\geq \lc\lc a^p+b^p\rc^\frac{q}{p}-(c^p+c^p)^\frac{q}{p}\rc^\frac{1}{q}, $$
    which is equivalent to 
    $$\lc\lc a^q-c^q\rc^\frac{p}{q}+\lc b^q-c^q\rc^\frac{p}{q}\rc^\frac{q}{p}+(c^p+c^p)^\frac{q}{p}\geq \lc a^p+b^p\rc^\frac{q}{p}. $$
    This inequality follows directly from Minkowski inequality with the power $\frac{p}{q}>1$.
    \item $a\leq c, b>c$. It is easy to see that $ A_q\lc c\ps{p} b,c\ps{p} c\rc\geq  A_q\lc a\ps{p} b,c\ps{p} c\rc$. Then, it suffices to show that $ A_q(b,c)= A_q(c,c)\ps{p} A_q(b,c)\geq  A_q\lc c\ps{p} b,c\ps{p} c\rc$, which follows from case 2. 
\end{enumerate}
\end{proof}

\begin{proposition}\label{prop:closed-psub}
For $1\leq q<p\leq\infty$, $(\msp,\dghp{q})$ is a closed subspace of $(\ms_q,\dghp{q})$.
\end{proposition}

\begin{proof}
Given any $\dghp{q}$ convergent sequence $\{X_n\}_{n=1}^\infty$ with $X_n\in\msp\subseteq\ms_q$ for all $n\in \mathbb{N}$, we need to show that $X=\lim_{n\rightarrow\infty}X_n$ belongs to $\msp$. Take arbitrarily three distinct points $x_1,x_2,x_3\in X$. For any small $\eps>0$, there exists $N>0$, such that for any $n>N$, we have $\dghp{q}(X_n,X)\leq \frac{\eps}{2}.$ Hence, there exists correspondence $R_n\in\mathcal{R}(X_n,X)$ such that $\dis_q(R_n)\leq\eps$. Choose $x_1^n,x_2^n,x_3^n\in X_n$ such that $(x_i^n,x_i)\in R_n $ for $i=1,2,3$. Then, for $i,j=1,2,3$, we have that 
$$\Lambda_q\left(d_X(x_i,x_j), d_{X_n}\!\left(x_i^n,x_j^n\right)\right)\leq\eps.$$ 
Therefore, we have that
\begin{align*}
    d_X(x_1,x_2)\ps{p} d_X(x_2,x_3)&\geq   A_q\lc d_{X_n}( x_1^n,x_2^n),\eps\rc\ps{p} A_q\lc d_{X_n}( x_2^n,x_3^n),\eps\rc\\
    &\geq  A_q\lc d_{X_n}( x_1^n,x_2^n)\ps{p} d_{X_n}( x_2^n,x_3^n), \eps\ps{p} \eps\rc\\
    &\geq  A_q\lc d_{X_n}(x_1^n,x_3^n),\eps\ps{p}\eps \rc\\
    &\geq  A_q\lc d_X(x_1,x_3),\eps\ps{q}(\eps\ps{p}\eps) \rc.
\end{align*}
The first and the last inequalities follow from Proposition \ref{prop:asyineq}. The second inequality follows from Lemma \ref{lm:pqineq}.
Since $\eps>0$ is arbitrary, we have that 
$$d_X(x_1,x_2)\ps{p} d_X(x_2,x_3)\geq d_X(x_1,x_3),$$
and thus $X\in\msp$.
\end{proof}

\begin{proposition}\label{prop:nowhere-dense}
Given $1\leq q<p\leq\infty$, $(\msp,\dghp{q})$ is a nowhere dense subset of $(\ms_q,\dghp{q})$, i.e., the closure $\overline{(\msp,\dghp{q})}$ has no interior in $(\ms_q,\dghp{q})$.
\end{proposition}
\begin{proof}
By Proposition \ref{prop:closed-psub}, $\overline{(\msp,\dghp{q})}=(\msp,\dghp{q})$. Assume that $X\in\msp$ is an interior point. Without loss of generality, by Example \ref{ex:conv-pms}, we can assume that $(X,d_X)$ is a finite space. Define a set $\hat{X}\coloneqq X\cup\{x_1,x_2\}$ where $x_1$ and $x_2$ are two additional points. Pick an arbitrary point $x_0\in X$. For any $\eps>0$, define a symmetric function $d_\eps:\hat{X}\times\hat{X}\rightarrow\Rp$ as follows:
\begin{enumerate}
    \item $d_\eps|_{X\times X}\coloneqq d_X$;
    \item for any $x\in X$ such that $x\neq x_0$, let $d_\eps(x,x_i)\coloneqq d_X(x,x_0)$ for $i=1,2$.
    \item $d_\eps(x_0,x_i)\coloneqq\eps$ for $i=1,2$ and $d_\eps(x_1,x_2)\coloneqq\eps\ps{q} \eps$.
    
\end{enumerate}
Recall that $\mathrm{sep}(X)\coloneqq\min\{d_X(x,x'):\,x,x'\in X\}$. Then, it is easy to verify that when $\eps\leq \mathrm{sep}(X)\ps{q}\mathrm{sep}(X)$, $X_\eps\coloneqq\lc\hat{X},d_\eps\rc\in\ms_q$. Moreover, $X_\eps\notin\msp$ since $x_0,x_1,x_2$ does not satisfy the $p$-triangle inequality: 
$$d_\eps(x_0,x_1)\ps{p} d_\eps(x_0,x_2)=\eps\ps{p} \eps=2^\frac{1}{p}\eps<2^\frac{1}{q}\eps=\eps\ps{q}\eps=d_\eps(x_1,x_2). $$
Consider the correspondence $R$ between $X$ and $X_\eps$ defined by
$$R=\{(x,x):\,x\in X\}\bigcup\{(x_0,x_i):\,i=1,2\}.$$
Then, we have $\dis_q(R)=\eps\ps{q}\eps$. Thus
$$ \lim_{\eps\rightarrow 0}\dghp{q}(X_\eps,X)=0.$$
This contradicts the assumption that $X$ is an interior point
\end{proof}

Proposition \ref{prop:nowhere-dense} indicates the following result stating that $\msp$ is a very `thin' subset of $\ms_q$ for $1\leq q<p\leq\infty$. In fact, we have the following stronger result.

\begin{theorem}
Let $q\in[1,\infty)$, then $\bigcup_{p\in(q,\infty]}\msp\subsetneq\ms_q$. In particular when $q=1$, we have $\bigcup_{p\in(1,\infty]}\msp\subsetneq\ms$. 
\end{theorem}

\begin{proof}
Obviously, by Proposition \ref{prop:inclusionpq}, $\bigcup_{p\in(q,\infty]}\msp\subseteq\ms_q$.

Let $\{p_n\}_{n=1}^\infty$ be a strictly decreasing sequence with $q$ being the limit point. Let $p_0=\infty$. Then, we have the sequence $\ms_{p_0}\subseteq\ms_{p_1}\subseteq\cdots$. By Proposition \ref{prop:inclusionpq}, we know $\bigcup_{p\in(q,\infty]}\msp=\bigcup_{n=0}^\infty\ms_{p_n}$, which is a countable union of nowhere dense sets. Since $(\ms_q,\dghp{q})$ is a complete metric space (Theorem \ref{thm:complete-sep}), by the  Baire category theorem, $\ms_q\neq\bigcup_{p\in(q,\infty]}\msp$. 
\end{proof}

In the proof we know that $\bigcup_{p\in(q,\infty]}\msp$ is actually a meager set of $\ms_q$, which implies that most elements of $\ms_q$ cannot be captured by $p$-metric spaces with $p>q$.

\begin{example}
Consider the interval $[0,1]\subseteq\R$. Then, $[0,1]\in\ms\backslash\bigcup_{p>1}\msp$. This fact implies that any geodesic space $X\in\ms\backslash\bigcup_{p>1}\msp$.
\end{example}

\begin{example}
Consider the unit circle $\mathbb{S}^1=\{(x,y):\,x^2+y^2=1\}$ on $\R^2$ with the Euclidean distance. Then, this is a non-geodesic space and $\mathbb{S}^1\in\ms\backslash\bigcup_{p>1}\msp$.
\end{example}

\subsection{Geodesic properties}\label{sec:geodesic}

In this section, we will discuss geodesic properties of $\msp$. In particular, we study the notion of \emph{$p$-geodesic property} of $\msp$. Unless otherwise specified, we always assume that $p\in[1,\infty)$ in this section.

\begin{definition}[$p$-length]
For a $p$-metric space $(X,d_X)$ and a continuous curve $\gamma:[0,1]\rightarrow X$, we define its $p$-length as
$$ \mathrm{length}_p(\gamma)\coloneqq\sup\left\{\mathop{\ps{\mathrlap{p}}}_{i=1}^{n-1}\, d_X(\gamma(t_i),\gamma(t_{i+1})):\,0=t_0<t_1<\cdots<t_n=1\right\}.$$
\end{definition}

\begin{remark}\label{rmk:lpdx}
It is  clear that for any continuous curve $\gamma:[0,1]\rightarrow X$, 
$$d_X(\gamma(0),\gamma(1))\leq\mathrm{length}_p(\gamma).$$
\end{remark}

\begin{definition}[$p$-geodesic]
Given any $p$-metric space $X$, a continuous curve $\gamma:[0,1]\rightarrow X$ is called a \emph{$p$-geodesic}, if 
$$d_X(\gamma(s),\gamma(t))=|s-t|^\frac{1}{p}\,d_X(\gamma(0),\gamma(1)),\forall s,t\in[0,1]. $$
We say $X$ is $p$-geodesic, if any two points in $X$ can be connected by a $p$-geodesic.
\end{definition}
\begin{remark}
Note that when $p=1$, the notion of $p$-geodesics coincides with the usual notion of geodesics.
\end{remark}

\begin{remark}[What if $p=\infty$?]\label{rmk:why not ult}
Ultrametric spaces are \emph{totally disconnected}, i.e., any subspace with at least two elements is disconnected \cite{semmes2007introduction}. This in turn implies that each continuous curve in an ultrametric space is constant. Therefore, it is meaningless to discuss about $\infty$-geodesic property for ultrametric spaces.
\end{remark}

Geodesics are also known to be the shortest path connecting points. A property similar to this holds for $p$-geodesics.
\begin{lemma}
Consider any $p$-metric space $X$. Let $x$ and $x'$ be two distinct points in $X$. Then, among all curves connecting $x$ and $x'$, a $p$-geodesic has the smallest $p$-length.
\end{lemma}

\begin{proof}
It is easy to show that $\mathrm{length}_p(\gamma)=d_X\lc\gamma(0),\gamma(1)\rc$. Then, by Remark \ref{rmk:lpdx}, we know that $\gamma$ is a curve connecting $x$ and $x'$ with smallest $p$-length.
\end{proof}

The notions of $p$-geodesic and geodesic are related by the snowflake transform (Example \ref{ex:snowflake}).

\begin{proposition}\label{thm:geotopgeo}
Let $X$ be a metric space. If $X$ is geodesic, then $S_\frac{1}{p}(X)$ is $p$-geodesic.
\end{proposition}

\begin{proof}
Given two point $x,x'\in X$, there exists a geodesic $\gamma:[0,1]\rightarrow X$ connecting them. Then, for any $s,t\in[0,1]$, we have
$$(d_X)^\frac{1}{p}(\gamma(s),\gamma(t))=\left(|s-t|\,d_X(x,x')\right)^\frac{1}{p}=|s-t|^\frac{1}{p}\,(d_X)^\frac{1}{p}(x,x'). $$
This implies that $\gamma$ is a $p$-geodesic in $S_\frac{1}{p}(X)$ connecting $x$ and $x'$. Therefore, $S_\frac{1}{p}(X)$ is $p$-geodesic.
\end{proof}

\begin{example}
For any $l>0$, the subspace $([0,l],d)\subseteq\R$ is geodesic. Then, by \Cref{thm:geotopgeo},
$([0,l],d^\frac{1}{p})$ is $p$-geodesic for any $1\leq p<\infty$. 
\end{example}

As a partial generalization of midpoint criterion (cf. \cite[Theorem 2.4.16]{burago2001course}), we have the following necessary condition for the $p$-geodesic property.
\begin{theorem}\label{lm:midgeo}
Let $X$ be complete $p$-metric space. Then, $X$ is a $p$-geodesic space if and only if for any two distinct points $x,x'\in X$, there exists $m\in X$ such that 
$$ d_X(x,m)=d_X(x',m)=\lc\frac{1}{2}\rc^\frac{1}{p}\,d_X(x,x').$$
Any such point $m$ is called a \emph{$p$-midpoint} between $x$ and $x'$.
\end{theorem}

\begin{proof}
We first assume that $X$ is $p$-geodesic. Then, for any two distinct points $x,x'\in X$, there exists a $p$-geodesic $\gamma:[0,1]\rightarrow X$ connecting them. Consider $m\coloneqq\gamma\lc\frac{1}{2}\rc$. By definition of $p$-geodesic, we have
$$d_X(x,m)=d_X(x',m)=\lc\frac{1}{2}\rc^\frac{1}{p}\,d_X(x,x'). $$

Conversely, by the midpoint criterion \cite[Theorem 2.4.16]{burago2001course}, it is easy to check that $S_p(X)$ is a geodesic space. Then, by \Cref{thm:geotopgeo} we have that $X=S_\frac{1}{p}(S_p(X))$ is $p$-geodesic.
\end{proof}

\begin{proposition}\label{prop:p metric not geodesic}
Let $X$ be a $p$-metric space. If $X$ is $p$-geodesic, then for any $1\leq q<p$, $X$ is not $q$-geodesic.
\end{proposition}

\begin{proof} Note that the proposition is trivially true when $p=1$. 
Suppose on the contrary that $p>1$ and that $X$ is $q$-geodesic for some $1\leq q <p$. Then, by \Cref{lm:midgeo}, for any two distinct points $x,x'\in X$, there exists a $q$-midpoint $x''\in X$ between $x$ and $x'$, such that 
$$d_X(x,x'')=d(x',x'')=\lc\frac{1}{2}\rc^\frac{1}{q}d_X(x,x'). $$
Therefore, 
$$d_X(x,x'')\ps{p} d_X(x'',x')=2^\frac{1}{p}\cdot\lc\frac{1}{2}\rc^\frac{1}{q}d(x,x')< d(x,x'),$$
which contradicts the fact that $X$ is a $p$-metric space.
\end{proof}

Next, we establish ($p$-)geodesic properties regarding $\msp$ and $\ums$.

\subsubsection{$p$-metric spaces}
We know from \cite{ivanov2016gromov,chowdhury2018explicit} that $(\ms,\dgh)$ is a geodesic space. This leads us to wondering whether $(\msp,\dghp{p})$ is a geodesic space as well. The following theorem provides a complete answer.

\begin{theorem}
$(\msp,\dghp{p})$ is $p$-geodesic but not $q$-geodesic for any $q<p$. In particular, $(\msp,\dghp{p})$ is not geodesic when $p>1$.
\end{theorem}

\begin{proof}

By the previous proposition, we only need to show that $(\msp,\dghp{p})$ is $p$-geodesic. Since by \Cref{thm:geotopgeo}, $(\ms,\dgh)$ is geodesic, the proof follows from $(\msp,\dghp{p})\cong \left(\ms,(\dgh)^\frac{1}{p}\right)$ (cf. \Cref{thm:isompms}). 
\end{proof}

\subsubsection{Ultrametric spaces}
We know from \Cref{rmk:why not ult} that each continuous curve in an ultrametric space is trivial due to total disconnectedness. We provide an alternative proof of the fact via $\mathfrak{S}_\infty$ as follows.

\begin{proposition}\label{prop:trivialcurve}
If $X$ is an ultrametric space, then any continuous curve $\gamma:[0,1]\rightarrow X$ is a trivial curve, i.e., there exists $x\in X$ such that $\gamma(t)\equiv x$ for all $t\in[0,1]$.
\end{proposition}

\begin{proof}
Let $X_0\coloneqq\mathrm{image}(\gamma)$. We then obtain an ultrametric space $(X_0,u_X|_{X_0\times X_0})$ by restricting $u_X$ to $X_0\times X_0$. Since $\gamma$ is continuous, we have that $X_0$ is path-connected. By Proposition \ref{prop:inftykernel} we have that $\mathfrak{S}_\infty(X_0)=*$. By Proposition \ref{prop:propofsp}, $\mathfrak{S}_\infty(X_0)=X_0$. Therefore, $X_0=*$ is a one point space and thus $\gamma$ is a trivial curve. 
\end{proof}

We know that $(\ums,\ugh)$ is an ultrametric space. Then, the proposition above precludes $(\ums,\ugh)$ from being geodesic. However, if we consider other distance functions on $\ums$, there may still exist geodesic structure on $\ums$. In fact, we have:

\begin{theorem}\label{thm:ums-geod}
$(\ums, \dgh)$ is geodesic. 
\end{theorem}

\begin{proof}
Let $X$ and $Y$ be two compact ultrametric spaces. Let $\gamma:[0,1]\rightarrow\ms$ be a geodesic connecting $X$ and $Y$ in $(\ms,\dgh)$. Let $\tilde{\gamma}\coloneqq\mathfrak{S}_\infty\circ\gamma:[0,1]\rightarrow\ums$. Then, by Proposition \ref{prop:propofsp} we have that $\tilde{\gamma}(0)=\gamma(0)=X$ and $\tilde{\gamma}(1)=\gamma(1)=Y$. By \Cref{thm:slstab}, we have that for any $s,t\in[0,1]$,
$$\dgh(\tilde{\gamma}(s),\tilde{\gamma}(t))\leq\dgh(\gamma(s),\gamma(t))\leq |t-s|\,\dgh(\gamma(0),\gamma(1))=|t-s|\,\dgh(\tilde{\gamma}(0),\tilde{\gamma}(1)). $$
This shows that $\tilde{\gamma}$ is a geodesic connecting $X$ and $Y$ and thus $(\ums,\dgh)$ is geodesic.
\end{proof}

\begin{remark}
More generally, one can show that $(\ums,\dghp{p})$ is $p$-geodesic by modifying the previous proof slightly, e.g., replacing \Cref{thm:slstab} in the proof with \Cref{thm:p-stable}.
\end{remark}

We can also consider the  $p$-interleaving distance for $p\in[1,\infty]$ on $\ums$. However, none of $\dIp{p}$ will impose geodesic structure on $\ums$.

\begin{proposition}
$(\ums,d_{\mathrm{I},p})$ is not geodesic for any $ p\in[1,\infty]$. 
\end{proposition}
\begin{proof}
Let $X=*$ be the one point space and $Y=\Delta_2(2)$ be the two-point space with inter-point distance 2. We prove that there is no midpoint between $X$ and $Y$. Then, by Lemma \ref{lm:midgeo} we have that there is no geodesic connecting $X$ and $Y$.

Fix $p\in[1,\infty]$. It is easy to show that $\dip(X,Y)=2^{1-\frac{1}{p}}$. Suppose there exists a 1-midpoint $Z\in\mathcal{U}$ such that $\dip(X,Z)=\dip(Y,Z)=2^{-\frac{1}{p}}$. Then, by Remark \ref{rmk:intdiam}, $\diam(Z)=2^{\frac{1}{p}}\,\dip(X,Z)=1$. 

First consider the case when $p>1$. By \Cref{coro:pint vs dghp}, $\dip(Y,Z)\geq\dlp(Y,Z).$ By \Cref{prop:diam-dghp}, we have $$\dlp(Y,Z)\geq2^{-\frac{1}{p}}\Lambda_p(\diam(Y),\diam(Z))=2^{-\frac{1}{p}}\Lambda_p(2,1)>2^{-\frac{1}{p}}.$$
Hence $\dip(Y,Z)\geq\dlp(Y,Z)>2^{\frac{1}{p}}$, contradiction!

Now suppose $p=1$, then the argument above does not work since $\Lambda_1(2,1)=1$. Consider any two maps $\varphi:Y\rightarrow Z,\psi:Z\rightarrow Y$. If $\varphi(y_1)=\varphi(y_2)$, then 
$$\codis_{\mathrm{I},1}(\varphi,\psi)\geq \frac{1}{2}\max(u_Y(y_1,\psi\circ\varphi(y_1)),u_Y(y_2,\psi\circ\varphi(y_2))=1.$$
Otherwise suppose $z_1\coloneqq\varphi(y_1)\neq\varphi(y_2)\eqqcolon z_2$. Since $\diam(Z)=1$, we have that $u_Z(z_1,z_2)\leq 1$. If $\psi(z_i)=y_i$ for $i=1,2$, then $\dis_{\mathrm{I},i}(\psi)\geq 1$. Otherwise, $$\codis_{\mathrm{I},1}(\varphi,\psi)\geq \frac{1}{2}\max(u_Y(y_1,\psi\circ\varphi(y_1)),u_Y(y_2,\psi\circ\varphi(y_2)))=1.$$ 
In conclusion, $\dint(Y,Z)\geq 1>\frac{1}{2}$ by Theorem \ref{thm:pintdist}, contradiction!
\end{proof}

\begin{remark}
We can modify the case of $p=1$ in the proof above to show that $(\ums,\dip)$ is not $p$-geodesic for all $p\in[1,\infty)$.
\end{remark}

\section{Discussion}

We introduced a one parameter family of Gromov-Hausdorff like distances $\dghp{p}$ which defines a $p$-metric on the collection $\ms$ of all compact metric spaces and the collection $\msp$ of all compact $p$-metric spaces. We studied the convergence of $p$-metric spaces under $\dlp$ and established a pre-compactness theorem for $(\msp,\dlp)$. When $p=\infty$, $\dghp{\infty}$ coincides with the Gromov-Hausdorff ultrametric $\ugh$. We established a special structural result for $\ugh$ which in turn gives rise to several characterizations of $\ugh$ via, e.g., curvature sets. On the collection $\ums$ of all ultrametric spaces, there is a natural extant distance called the interleaving distance. We found a distortion characterization for this interleaving distance in Theorem \ref{thm:intdist}. We further generalized the interleaving distance to $p$-interleaving distances and establish its Lipschitz equivalence with $\dlp$ for all $p\in[1,\infty]$.

\subsection*{Acknowledgements} We thank Prof. Phillip Bowers from FSU for posing questions leading to the results in \Cref{sec:subspace topology}. We also thank Samir Chowdhury for interesting conversations about geodesics on Gromov-Hausdorff space. We thank Zane Smith who suggested studying the notion of kernel of the projection maps $\mathfrak{S}_p$ which we discussed in  \Cref{sec:kernel}.
This work was partially supported by the NSF through grants  DMS-1723003, CCF-1740761,  and CCF-1526513.


\bibliography{biblio-dghp}
\bibliographystyle{alpha}


\end{document}